\documentclass[a4paper]{amsart}

\usepackage[all]{xy}
\usepackage{amsmath}
\usepackage{amssymb}
\usepackage{amsthm}
\usepackage{latexsym}
\usepackage{enumerate}
\usepackage{prettyref}
\usepackage{hyperref}
\usepackage{bm}
\usepackage{tikz}
\usepackage{genyoungtabtikz}
\usepackage{mathrsfs}

\newtheorem{theorem}{Theorem}[section]
\newrefformat{thm}{\hyperref[{#1}]{Theorem~\ref*{#1}}}
\newtheorem{definition}[theorem]{Definition}
\newrefformat{def}{\hyperref[{#1}]{Definition~\ref*{#1}}}
\newtheorem{lemma}[theorem]{Lemma}
\newrefformat{lem}{\hyperref[{#1}]{Lemma~\ref*{#1}}}
\newtheorem{proposition}[theorem]{Proposition}
\newrefformat{prop}{\hyperref[{#1}]{Proposition~\ref*{#1}}}
\newtheorem{corollary}[theorem]{Corollary}
\newrefformat{cor}{\hyperref[{#1}]{Corollary~\ref*{#1}}}
\newtheorem{remark}[theorem]{Remark}
\newrefformat{rem}{\hyperref[{#1}]{Remark~\ref*{#1}}}

{
\newtheorem{examplecore}[theorem]{Example}}
\newrefformat{ex}{\hyperref[{#1}]{Example~\ref*{#1}}}

\newenvironment{example}{\begin{examplecore}}{\hspace*{\fill}
$\square$\par\vspace{.1cm}\end{examplecore}}

\newcommand{\op}{\operatorname}

\begin{document}

\title{Oriented Schubert calculus in Chow--Witt rings of Grassmannians}  

\author{Matthias Wendt}

\date{July 2018}

\address{Matthias Wendt}
\email{m.wendt.c@gmail.com}

\subjclass[2010]{}

\begin{abstract}
We apply the previous calculations of Chow--Witt rings of Grassmannians to develop an oriented analogue of the classical Schubert calculus. As a result, we get complete diagrammatic descriptions of the ring structure in Chow--Witt rings and twisted Witt groups. In the resulting arithmetic refinements of Schubert calculus, the multiplicity of a solution subspace is a quadratic form encoding additional orientation information. We also discuss a couple of applications, such as a Chow--Witt version of the signed count of balanced subspaces of Feh{\'e}r and Matszangosz.
\end{abstract}

\maketitle
\setcounter{tocdepth}{1}
\tableofcontents

\section{Introduction}

Schubert calculus is a collection of techniques to compute the intersection products in the Chow rings of Grassmannians. It can be used to answer questions about intersections of subspaces of an $F$-vector space $V$, like ``how many $k$-dimensional subspaces of $V$ satisfy specified intersection conditions with respect to a finite set of flags in $V$?'' For every flag, the $k$-dimensional subspaces satisfying the specified intersection conditions form a Schubert variety. Over an algebraically closed field, the number of subspaces satisfying the specified conditions (assuming that the flags are in general position) can be computed as degree of the intersection of the corresponding Schubert varieties (provided that intersection is zero-dimensional). 

If the field $F$ is not algebraically closed, then some of the solution subspaces may not be defined over $F$ and then the degree of the intersection of Schubert varieties is not necessarily equal to the number of solution subspaces. For example, over $F=\mathbb{R}$, the real points of the variety of generic flags for the Schubert problem (with the analytic topology) is not path-connected and the numbers of solution subspaces in different path-components are usually different. In particular, the number of solution subspaces defined over the base field $F$ is not invariant in general. In real enumerative geometry and real Schubert calculus, this problem has been solved by introducing signs depending on additional orientations. Counting the solution subspaces with signs produces an invariant number which can be related to intersection products in cohomology rings of real Grassmannians, cf. \cite{finashin:kharlamov:3planes,feher:matszangosz}. 

The point of the present paper is to push this idea further: we want to provide an arithmetic refinement of Schubert calculus which provides more information about the number of solution subspaces over arbitrary fields of characteristic $\neq 2$. Following the ideas of Kass and Wickelgren \cite{kass:wickelgren}, the number of solution subspaces should be replaced by a quadratic form depending on the field of definition and additional orientation information of the solution subspaces. The right setting to realize this idea is given by the Chow--Witt rings $\widetilde{\op{CH}}^\bullet(\op{Gr}(k,n),\mathscr{L})$ of Grassmannians. The structure of these rings was determined in \cite{real-grassmannian}. One of the goals of the present paper is to make the intersection product more transparent by providing diagrammatical rules in terms of even Young diagrams of \cite{balmer:calmes}, thus also settling the open question of multiplicative structure of twisted Witt groups of Grassmannians. The other main goal of the present paper is to explain how (canonical) choices of orientations on Schubert varieties yield cycles in the Chow--Witt ring and how the multiplicity of a solution subspace in the intersection of such Schubert cycles is a quadratic form encoding information on the field of definition and orientation of the solution subspace.

Now for a more detailed description of the results of the paper. While we are mostly interested in Chow--Witt rings of Grassmannians, many of the results can be reduced to statements about the cohomology rings $\bigoplus_{j,\mathscr{L}}\op{H}^j(\op{Gr}(k,n),\mathbf{I}^j(\mathscr{L}))$. The first thing to note is that the natural morphism $\widetilde{\op{CH}}^\bullet(\op{Gr}(k,n),\mathscr{L})\to \op{CH}^\bullet(\op{Gr}(k,n))$ is not generally surjective. The obstruction is the Steenrod square, and there is a precise formula for the Steenrod square in terms of Young diagrams. The following is only an informal statement of the  combined results of Proposition~\ref{prop:sq2lift}, Theorem~\ref{thm:steenrodlift}, Proposition~\ref{prop:compareBC} and Theorem~\ref{thm:lift}. The precise formulations are given there and require quite a bit more of notation to be set up.

\begin{theorem}
\label{thm:main1}
\begin{enumerate}
\item
A class $x\in\op{CH}^j(\op{Gr}(k,n))$ lifts to the Chow--Witt group $\widetilde{\op{CH}}^j(\op{Gr}(k,n),\mathscr{L})$ if and only if the Steenrod square $\op{Sq}^2_{\mathscr{L}}$ is trivial.
\item For a Schubert class in $\op{CH}^j(\op{Gr}(k,n))$, the Steenrod squares $\op{Sq}^2_{\mathscr{O}}$ and $\op{Sq}^2_{\det}$ of a Young diagram can be computed as follows: filling the boxes with a checkerboard pattern starting in the upper left corner, the respective Steenrod square is given as the sum of those Young diagrams with checkerboard pattern in $k\times(n-k)$-frame which can be obtained by adding a single white or black box, respectively. In particular, a class lifts if all the corners of the Young diagram have even length. 
\item For a Schubert class $[\sigma_\Lambda]$ with vanishing $\op{Sq}^2_{\mathscr{L}}$, there is a canonical orientation of the twisted relative canonical bundle $\omega_{p_\Lambda}\otimes p_\Lambda^\ast\mathscr{L}$, where $p_\Lambda:\op{Fl}(\Lambda)\to\op{Gr}(k,n)$ is a specific resolution of singularities. A canonical lift of the Schubert class to the Chow--Witt ring is given by the pushforward of the constant form 1 from the resolution (with the canonical choice of orientation). 
\end{enumerate}
\end{theorem}

In particular, we can determine additive generators for the $\mathbf{I}^\bullet$-cohomology ring. Moreover, the even Young diagrams provide an additive basis of its quotient mod torsion, the $\mathbf{W}$-cohomology ring, generalizing classical results of Pontryagin \cite{pontryagin}. To understand the multiplication of non-torsion classes one can use an isomorphism from a suitable Chow ring, based on doubled partitions. On the other hand, multiplication with torsion classes is much simpler and can be determined by reduction to the mod 2 Chow ring. We only state the simplest case, for $\op{Gr}(2a,2b)$. For the other cases, which require more notation to set up, cf. Proposition~\ref{prop:folklore}, Proposition~\ref{prop:addbasis1} and Theorem~\ref{thm:multw}.

\begin{theorem}
\label{thm:main2}
\begin{enumerate}
\item The even Young diagrams, cf. Definition~\ref{def:evenyoung} form a $\op{W}(F)$-basis for the $\mathbf{W}$-cohomology ring $\op{H}^\bullet(\op{Gr}(k,n),\mathbf{W}\oplus\mathbf{W}(\det))$. 
\item There is a natural ring isomorphism 
\[
\op{Ch}^\bullet(\op{Gr}(a,b))\to\op{H}^{4\bullet}(\op{Gr}(2a,2b),\mathbf{W}):\op{c}_i\mapsto \op{p}_i,
\]
This allows to completely describe the multiplication in $\mathbf{W}$-cohomology in terms of classical Schubert calculus for the Chow ring. 
\item The product of any class with an $\op{I}(F)$-torsion class is $\op{I}(F)$-torsion and completely determined by its image in the mod 2 Chow ring $\op{Ch}^\bullet(\op{Gr}(k,n))$. 
\end{enumerate}
\end{theorem}

As a direct consequence, we can also describe the multiplicative structure of the twisted Witt groups of Grassmannians. Their additive structure was determined in \cite{balmer:calmes},  but the ring structure was left open in loc.cit. 

\begin{corollary}
There is a natural isomorphism 
\[
\alpha:\op{W}^i(\op{Gr}(k,n),\mathscr{L})\xrightarrow{\cong}  \bigoplus_{j\equiv i\bmod 4} \op{H}^j(\op{Gr}(k,n),\mathbf{W}(\mathscr{L})).
\]
In particular, the description of the multiplicative structure on $\mathbf{W}$-cohomology in Theorem~\ref{thm:main2} completely determines the multiplicative structure on the twisted Witt groups of Grassmannians. 
\end{corollary}

Now we can state how ``oriented'' Schubert calculus works: a classical Schubert problem in the Chow ring can be lifted to the Chow--Witt ring if the Steenrod square condition from Theorem~\ref{thm:main1} is satisfied, and in this case, we have a canonical lift provided by canonical orientations of the Schubert varieties. This provides an ``oriented Schubert problem'': a collection of Schubert varieties with additional orientations. On the one hand, we can then use the above description of the Chow--Witt ring structure in Theorem~\ref{thm:main2} to compute the degree of the intersection of these oriented Schubert cycles in the top Chow--Witt group. On the other hand, the local contribution of a solution subspace in the intersection of the Schubert cycles is given by a quadratic form depending on the field of definition of the solution subspace as well as the additional orientation information. The following is Theorem~\ref{thm:types} which describes the local contributions in the case of smooth oriented Schubert problems. There is a similar statement for the singular case, cf. Theorem~\ref{thm:typessing}, which again requires more notation for a precise statement. 

\begin{theorem}
\label{thm:main3} 
Let $F$ be a field of characteristic $\neq 2$. 
\begin{enumerate}
\item For any oriented Schubert problem, the degree of the intersection in the top Chow--Witt group always has the form $a \langle 1 \rangle + b \langle -1 \rangle$. If one of the Schubert classes is torsion, then $a=b$.
\item
Fix a smooth oriented Schubert problem in the Grassmannian $\op{Gr}(k,n)$. For any solution subspace $[W]\in\op{Gr}(k,n)$ defined over $E/F$, there is a natural identification $\xi_W$ of the tangent space of $\op{Gr}(k,n)$ at $[W]$ with the direct sum of the normal spaces to the smooth Schubert varieties intersecting at $[W]$, and the oriented multiplicity of $[W]$ in the intersection cycle is $\op{tr}_{E/F}\langle \det\xi_W\rangle$, i.e., it encodes the square class of the unit comparing the two natural orientations above.
\end{enumerate}
\end{theorem}

It should be pointed out clearly that the above results describe two sides of the multiplicative structure of the $\mathbf{I}^\bullet$-cohomology which are both necessary to understand enumerative consequences: on the one hand, we have the picture of characteristic classes, and a presentation of the cohomology ring in terms of Pontryagin classes which allows to compute degrees of intersection products via the folklore isomorphism of Theorem~\ref{thm:main2}. On the other hand, we have very canonical cohomology classes of geometric origin, arising as pushforward of $\langle 1\rangle$ from resolutions of singularities for Schubert varieties corresponding to even Young diagrams; and oriented multiplicities for intersections of such geometric cycles can be determined very explicitly as in Theorem~\ref{thm:main3}. But to get enumerative consequences, we need that both pictures coincide, i.e., that the geometric cycles are actually representatives of the respective characteristic classes. Then we can compute the degree of the intersection (i.e., the refined number of solution subspaces) using Theorem~\ref{thm:main2} and its relatives \emph{and} identify the local contribution of each solution subspace in terms of orientation information encoded in the notion of type. This identification of geometric cycles and characteristic classes is provided in Proposition~\ref{prop:identification}. 

As a sample application, we have the following generalization of the signed count of balanced subspaces in \cite{feher:matszangosz}, cf. Theorem~\ref{thm:fmgw}. 

\begin{theorem}
Let $F$ be a field of characteristic $\neq 2$. Let $a=2i$ and $b=2j$ be even numbers with $a<b$. Then the following equality holds in $\op{GW}(F)\cong\widetilde{\op{CH}}^{4(ab-a^2)}(\op{Gr}(2a,2b))$:
\[
\sum_{\op{balanced} W}\op{tr}_{F(W)/F}\langle\det\xi_W\rangle=\frac{1}{2}\left(\binom{2j}{2i}+\binom{j}{i}\right)\langle 1\rangle + \frac{1}{2}\left(\binom{2j}{2i}-\binom{j}{i}\right)\langle-1\rangle.
\]
\end{theorem}

For $F$ an algebraically closed field, this recovers the classical results from Schubert calculus. For $F=\mathbb{R}$, this recovers results of \cite{feher:matszangosz}, and generally encodes lower bounds for the number of real solutions of a Schubert problem. For $F=\mathbb{F}_q$, this provides mod 2 congruences of solution subspaces over even-degree extensions with oriented splitting and solution subspaces over odd-degree extensions with non-oriented splitting, cf. Proposition~\ref{prop:fq}. 


\emph{Structure of the paper:} 
We begin in Section~\ref{sec:recall1} by recalling the descriptions of Chow--Witt rings of Grassmannians from \cite{real-grassmannian}. Then we describe the multiplication of torsion-free classes (in the $\mathbf{W}$-cohomology ring) in terms of even Young diagrams in Section~\ref{sec:cohomology}, and the lifting conditions and multiplication of torsion classes are discussed in Section~\ref{sec:torsion}. After a short recollection of pushforwards and products in Witt groups and $\mathbf{I}^\bullet$-cohomology in Section~\ref{sec:wittbasics}, we discuss the explicit geometric representatives of lifts of Schubert classes in Section~\ref{sec:geometricreps}. Then we discuss the explicit geometric identification of oriented intersection multiplicities in Section~\ref{sec:intersection}. Finally, all the results are combined in Section~\ref{sec:main}, including the comparison of $\mathbf{W}$-cohomology rings and twisted Witt groups, and a couple of sample enumerative applications are given in Section~\ref{sec:results}. 

\emph{Acknowledgements:} First of all, many thanks to Kirsten Wickelgren; without the discussions about refined enumerative geometry and her work  during the ``\'etale and motivic homotopy 2017''  conference in Heidelberg and the ``motives and refined enumerative geometry'' workshop in May 2018 in Essen, the present work would not exist. I would also like to thank to Thomas Hudson for discussions about resolutions of singularities for Schubert varieties. The package \verb!genyoungtabtikz! was used to generate Young diagrams, and example computations were done using the \verb!SchurRings! packages of \verb!Macaulay2!.

\section{Recollection on Chow--Witt rings of Grassmannians}
\label{sec:recall1}

In this section, we recall the computations of the Chow--Witt rings resp. $\mathbf{I}^\bullet$-cohomology rings of the Grassmannians from \cite{real-grassmannian}. 

Recall that the Grassmannian $\op{Gr}(k,n)$ carries an exact sequence of vector bundles 
\[
0\to\mathscr{S}_k\to\mathscr{O}_{\op{Gr}(k,n)}^{\oplus n}\to\mathscr{Q}_{n-k}\to 0
\] 
where the fiber of $\mathscr{S}_k$ at a point $[W]$ is the subspace $W\subseteq V$, and the fiber of $\mathscr{Q}_{n-k}$ is $V/W$. We speak of the tautological sub- and quotient bundle, respectively. After pullback along $\op{GL}_n/\op{GL}_k\times\op{GL}_{n-k}\to\op{Gr}(k,n)$, the bundles become the tautological $k$-plane bundle  $\mathscr{E}_k$ and the complementary $n-k$-plane bundle $\mathscr{E}_{n-k}^\perp$, with $\mathscr{E}_k\oplus\mathscr{E}_{n-k}^\perp\cong \mathscr{O}_{\op{GL}_n/\op{GL}_k\times\op{GL}_{n-k}}^{\oplus n}$. 
The short description of the cohomology of the Grassmannians is that the characteristic classes of these two bundles generate the cohomology rings, modulo the relations coming from the Whitney sum formula. The one exception to this rule is the case where $\dim\op{Gr}(k,n)=k(n-k)$ is odd, which has a new exterior class $\op{R}$. A recollection on definitions and properties of the relevant characteristic classes (Pontryagin classes $\op{p}_i$, Euler classes $\op{e}_n$, Chern classes $\op{c}_i$, Stiefel--Whitney classes $\overline{\op{c}}_i$ and their Bocksteins $\beta_{\mathscr{L}}(\overline{\op{c}}_{2i_1}\cdots\overline{\op{c}}_{2i_j})$) can be found in \cite{chow-witt,real-grassmannian}. 

As a matter of notation, the cohomology rings we consider here (Chow--Witt as well as cohomology with $\mathbf{I}^\bullet$ or $\mathbf{W}$-coefficients) appear with possibly nontrivial twists by line bundles. On the Grassmannians, we have $\op{Pic}(\op{Gr}(k,n))/2\cong\mathbb{Z}/2\mathbb{Z}$, generated by $\det\mathscr{E}_k^\vee$. Sometimes, the twist of a sheaf $\mathbf{A}$ by $\det\mathscr{E}_k^\vee$ will be denoted by $\mathbf{A}(\det)$.

The first relevant statement, is the following, cf.~\cite[Theorem 5.10]{real-grassmannian}.

\begin{proposition}
\label{prop:fiberproduct}
Let $F$ be a perfect field of characteristic $\neq 2$, and let $1\leq k<n$. Then there is a cartesian square of $\mathbb{Z}\oplus\mathbb{Z}/2\mathbb{Z}$-graded $\op{GW}(F)$-algebras: 
\[
\xymatrix{
\widetilde{\op{CH}}^\bullet(\op{Gr}(k,n),\mathscr{O}\oplus\det\mathscr{E}_k^\vee) \ar[r] \ar[d] & \ker \partial_{\mathscr{O}}\oplus \ker\partial_{\det\mathscr{E}_k^\vee} \ar[d]^{\bmod 2}\\
\op{H}^\bullet_{\op{Nis}}(\op{Gr}(k,n),\mathbf{I}^\bullet\oplus\mathbf{I}^\bullet(\det\mathscr{E}_k^\vee)) \ar[r]_>>>>>>\rho & \op{Ch}^\bullet(\op{Gr}(k,n))^{\oplus 2}
}
\]
Here $\det\mathscr{E}_k^\vee$ is the determinant of the dual of the tautological rank $k$ bundle on $\op{Gr}(k,n)$, 
\[
\partial_{\mathscr{L}}:\op{CH}^\bullet(\op{Gr}(k,n))\to \op{Ch}^\bullet(\op{Gr}(k,n)) \xrightarrow{\beta_{\mathscr{L}}} \op{H}^{\bullet+1}(\op{Gr}(k,n),\mathbf{I}^{\bullet+1}(\mathscr{L})) 
\]
is the (twisted) integral Bockstein map. 
\end{proposition}

As a direct consequence, computations in the Chow--Witt ring of $\op{Gr}(k,n)$ can be done by separate computations in the Chow ring (which usually are classical and well-known, cf. \cite{3264}) and the $\mathbf{I}^\bullet$-cohomology ring. For this reason, we will focus on computations in the $\mathbf{I}^\bullet$-cohomology from now on. 

We next formulate the relevant statements describing the $\mathbf{I}^\bullet$-cohomology of the Grassmannians, cf.~\cite[Theorems 5.6, 5.8, 6.7]{real-grassmannian}. First, we have a result which essentially allows to completely determine the torsion in the cohomology:

\begin{proposition}
\label{prop:incoh}
Let $F$ be a field of characteristic $\neq 2$, and let $1\leq k<n$. 
\begin{enumerate}
\item The image of the Bockstein morphism
\[
\beta_{\mathscr{L}}:\op{Ch}^\bullet(\op{Gr}(k,n))\to \op{H}^{\bullet+1}(\op{Gr}(k,n),\mathbf{I}^{\bullet+1}(\mathscr{L}))
\]
coincides exactly with the $\op{I}(F)$-torsion. The images of the Bockstein morphisms form an ideal in $\op{H}^{\bullet}(\op{Gr}(k,n),\mathbf{I}^{\bullet}\oplus\mathbf{I}^\bullet(\det\mathscr{E}_k))$. 
\item The reduction morphism 
$\rho_{\mathscr{L}}: \op{H}^{\bullet}(\op{Gr}(k,n),\mathbf{I}^{\bullet}(\mathscr{L})) \to \op{Ch}^\bullet(\op{Gr}(k,n))$ is injective on the image of $\beta_{\mathscr{L}}$, and $\rho_{\mathscr{L}}\circ\beta_{\mathscr{L}}(x)=\op{Sq}^2_{\mathscr{L}}(x)$. In particular, products of Bockstein classes $\beta_{\mathscr{L}}(x)$ with other classes can be determined by computing the products of the reductions in $\op{Ch}^\bullet(\op{Gr}(k,n))$.
\item Products in $\op{Ch}^\bullet(\op{Gr}(k,n))$ with elements $\op{Sq}^2(x)$ are computed using that the  Steenrod square $\op{Sq}^2_{\mathscr{O}}$ is a derivation, $\op{Sq}^2_{\det}(x)=\overline{\op{c}}_1x+\op{Sq}^2_{\mathscr{O}}(x)$ and the Wu formula: 
\[
\op{Sq}^2_{\mathscr{O}}: \overline{\op{c}}_j\mapsto \overline{\op{c}}_1\overline{\op{c}}_j+(j-1)\overline{\op{c}}_{j+1}; \qquad 
\op{Sq}^2_{\det}:\overline{\op{c}}_j\mapsto (j-1)\overline{\op{c}}_{j+1}.
\]
\end{enumerate}
\end{proposition}

Next, to study the  $\op{W}(F)$-torsionfree quotient of the cohomology, there is a natural exact sequence 
\[
\op{Ch}^j(X)\xrightarrow{\beta_{\mathscr{L}}} \op{H}^{j+1}(X,\mathbf{I}^{j+1}(\mathscr{L}))\to \op{H}^{j+1}(X,\mathbf{W})\to 0
\]
by means of which the cokernel of $\beta$ can be identified with the cohomology ring $\op{H}^\bullet(\op{Gr}(k,n),\mathbf{W})$ of the Grassmannians with coefficients in the sheaf of Witt groups $\mathbf{W}$. It has a description completely analogous to the integral cohomology mod torsion of the real Grassmannians. Some of these formulas already appeared in the computations of SL-oriented theories (in particular Witt groups) of Grassmannians in \cite{ananyevskiy}. The exact relation with the computations of twisted Witt groups of Grassmannians in \cite{balmer:calmes} will be discussed in Section~\ref{sec:comparisonwitt}.

\begin{proposition}
\label{prop:invertedeta}
Let $F$ be a perfect field of characteristic $\neq 2$ and let $1\leq k<n$. The cohomology $\op{H}^\bullet(\op{Gr}(k,n),\mathbf{W})$ has the following presentation, as commutative graded $\op{W}(F)$-algebra: 
\begin{enumerate}
\item For $k,n$ even, 
\[ \frac{\op{W}(F)[\op{p}_2,\dots,\op{p}_{k-2},\op{e}_k,\op{p}_2^\perp,\dots,\op{p}_{n-k-2}^\perp,\op{e}_{n-k}^\perp]}{(\op{p}\cdot\op{p}^\perp=1,\op{e}_k\cdot\op{e}_{n-k}^\perp=0)}
\]
\item If $n$ is odd, 
\[
 \left\{\begin{array}{ll} \frac{\op{W}(F)[\op{p}_2,\dots,\op{p}_{k-2},\op{e}_k,\op{p}_2^\perp,\dots,\op{p}_{n-k-1}^\perp]}{(\op{p}\cdot\op{p}^\perp=1)} & k\textrm{ even}\\
\frac{\op{W}(F)[\op{p}_2,\dots,\op{p}_{k-1},\op{p}_2^\perp,\dots,\op{p}_{n-k-2}^\perp,\op{e}_{n-k}^\perp]}{(\op{p}\cdot\op{p}^\perp=1)} & k\textrm{ odd}
\end{array}\right.
\]
\item  For $k,n-k$ odd, 
\[ \frac{\op{W}(F)[\op{p}_2,\dots,\op{p}_{k-1},\op{p}_2^\perp,\dots,\op{p}_{n-k-1}^\perp]}{(\op{p}\cdot\op{p}^\perp=1)}\otimes \bigwedge[R]
\]
\end{enumerate}
Here the bidegrees of the even Pontryagin classes $\op{p}_{2i}$ are $(4i,0)$, the bidegrees of the Euler classes $\op{e}_k$ and $\op{e}_{n-k}^\perp$ are $(k,1)$ and $(n-k,1)$, respectively, and the class $\op{R}$ in the last case has bidegree $(n-1,0)$. 
\end{proposition}

The above statements on torsion (Proposition~\ref{prop:incoh}) and the torsion-free quotient (Proposition~\ref{prop:invertedeta}) can be combined into a description of the $\mathbf{I}^\bullet$-cohomology: the cohomology is a direct sum of the torsion and the torsion-free quotient. With the above notation, we can also describe the reduction morphism relevant for computations with the fiber square of Proposition~\ref{prop:fiberproduct}. 

\begin{proposition}
\label{prop:reduction}
The reduction morphism 
\[
\rho_{\mathscr{L}}: \op{H}^{\bullet}(\op{Gr}(k,n),\mathbf{I}^{\bullet}(\mathscr{L})) \to \op{Ch}^\bullet(\op{Gr}(k,n))
\]
 is given by 
\begin{eqnarray*}
\op{p}_{2i}^{(\perp)}&\mapsto& (\overline{\op{c}}_{2i}^{\perp})^2,\\
\beta_{\mathscr{L}}(x)&\mapsto&\op{Sq}^2_{\mathscr{L}}(x),\\
\op{e}_k&\mapsto &\overline{\op{c}}_k,\\
\op{e}_{n-k}^\perp &\mapsto & \overline{\op{c}}_{n-k}^\perp,\\
\op{R}&\mapsto&\overline{\op{c}}_{k-1}\overline{\op{c}}_{n-k}^\perp=\overline{\op{c}}_k\overline{\op{c}}_{n-k-1}^\perp.
\end{eqnarray*}
\end{proposition}

\begin{remark}
The Grassmannian $\op{Gr}(k,n)$ has dimension $k(n-k)$. From the above presentation, we have 
\[
\op{H}^d(\op{Gr}(k,n),\mathbf{I}^d)\cong \left\{ \begin{array}{ll}
\op{W}(F) & n \textrm{ even} \\
\mathbb{Z}/2\mathbb{Z} & n \textrm{ odd}\end{array}\right.
\]
For twisted coefficients $\mathbf{I}^d(\det(\mathscr{E}_k))$, the situation is exactly reversed. This is a reflection of the classical fact that the real Grassmannian $\op{Gr}_k(\mathbb{R}^n)$ is orientable if and only if $n$ is even. 

It follows that the image of $\beta_{\mathscr{O}}$ in $\op{H}^d(\op{Gr}(k,n),\mathbf{I}^d)$ is trivial for $n$ even. This has direct consequences for products of elements which land in top degree (which is the relevant case for Schubert calculus): the factors of a product can be changed by elements from the image of $\beta$ without affecting the product, and a product with an $\op{I}(F)$-torsion class is always trivial.
\end{remark}

\section{Multiplicative structure via even Young diagrams}
\label{sec:cohomology}

Now we want to describe the multiplicative structure of the cohomology ring $\bigoplus_{j,\mathscr{L}}\op{H}^j(\op{Gr}(k,n),\mathbf{I}^j(\mathscr{L}))$ more explicitly. From Proposition~\ref{prop:incoh}, this cohomology ring splits additively as the direct sum of the torsion part, given by the image of 
\[
\beta_{\mathscr{L}}:\op{Ch}^\bullet(\op{Gr}(k,n))\to \op{H}^\bullet(\op{Gr}(k,n),\mathbf{I}^\bullet(\mathscr{L})),
\]
and the cokernel of $\beta$, which coincides with $\op{H}^\bullet(\op{Gr}(k,n),\mathbf{W}\oplus\mathbf{W}(\det))$. Moreover, the image of the Bockstein maps is in fact an ideal. Hence, to determine the multiplicative structure on $\mathbf{I}^\bullet$-cohomology, it suffices to determine the multiplicative structure on $\op{H}^\bullet(\op{Gr}(k,n),\mathbf{W}\oplus\mathbf{W}(\det))$ and the image of $\beta_{\mathscr{L}}$ separately. 

The torsion part will be discussed in the next section, for now we will focus on the multiplicative structure of $\mathbf{W}$-cohomology. Essentially, the multiplication in this part of the  $\mathbf{I}^\bullet$-cohomology ring can be understood in terms of diagrammatical calculations with the even Young diagrams of \cite{balmer:calmes}, similar to the computations with Young diagrams which describe computations in Chow rings of Grassmannians. This provides analogues of some relevant pieces of classical Schubert calculus: the Pieri and Giambelli formulas and the Littlewood--Richardson coefficients. 

\subsection{Young diagrams and multiplication in Chow rings}

We shortly recall how Young diagrams for partitions are used in the description of Chow rings of Grassmannians. This is a very well-known subject which is discussed in detail e.g. in \cite{3264}.\footnote{Most of the notation for Schubert calculus in Chow rings will be following that reference.} Similar to the discussion in Section~\ref{sec:recall1}, the Chow ring $\op{CH}^\bullet(\op{Gr}(k,n))$ is generated by the Chern classes $\op{c}_1,\dots,\op{c}_k$ of the tautological $k$-plane bundle $\mathscr{E}_k$ and the Chern classes $\op{c}_1^\perp,\dots,\op{c}_{n-k}^\perp$ of the complementary $(n-k)$-plane bundle $\mathscr{E}_{n-k}^\perp$, with the Whitney sum formula as defining relation:
\[
\op{CH}^\bullet(\op{Gr}(k,n))\cong \mathbb{Z}[\op{c}_1,\dots,\op{c}_k, \op{c}_1^\perp,\dots,\op{c}_{n-k}^\perp]/(\op{c}\cdot\op{c}^\perp=1).
\]

Fix a complete flag $\mathcal{F}: 0=V_0\subset V_1\subset \cdots\subset V_{n-1}\subset V_n=V$, $\dim V_i=i$ of subspaces of $V=F^n$. For a partition $\underline{a}=(a_1,\dots,a_k)$ with $n-k\geq a_1\geq a_2\geq \cdots\geq a_k\geq 0$, there is an associated \emph{Schubert variety}
\[
\Sigma_{\underline{a}}(\mathcal{F})=\left\{U\in \op{Gr}(k,V)\mid \forall i:\dim (V_{n-k+i-a_i}\cap U)\geq i\right\}\subseteq \op{Gr}(k,V)
\]
of dimension $\dim \Sigma_{\underline{a}}(\mathcal{F})=\sum a_i=:|\underline{a}|$. Partitions can be visualized by Young diagram, where the partition $\underline{a}=(a_1,\dots,a_k)$ corresponds to $k$ left-aligned rows of boxes, with $a_i$ boxes in the $i$-th row. The corresponding cycle class $\sigma_{\underline{a}}=[\Sigma_{\underline{a}}]\in\op{Ch}^{|a|}(\op{Gr}(k,n))$ is independent of the choice of flag $\mathcal{F}$. The partitions $(a_1,\dots,a_k)$ with $n-k\geq a_1\geq a_2\geq \cdots\geq a_k\geq 0$ correspond to Young diagrams fitting into a $k\times (n-k)$ rectangle, and they form an additive basis of $\op{CH}^\bullet(\op{Gr}(k,n))$. There are two types of \emph{special Schubert classes}, the classes $\sigma_{1^i}$ which correspond to a single column of $i$ boxes correspond to the Chern classes $\op{c}_i$, and the classes $\sigma_j$ with a single row of $j$ boxes correspond to the Chern classes $\op{c}_j^\perp$. 

With the above additive basis, the multiplicative structure of the Chow ring $\op{CH}^\bullet(\op{Gr}(k,n))$ can be described by the following classical formulas: 
\begin{enumerate}
\item The \emph{Pieri formula} describes multiplication with special Schubert classes: for a Schubert class $\sigma_{\underline{a}}$ and any integer $b$, we have 
\[
\sigma_b\cdot \sigma_{\underline{a}}=\sum_{\begin{array}{c}\scriptstyle|c|=|a|+b\\\scriptstyle a_i\leq c_i\leq a_{i-1}\forall i\end{array}}\sigma_c\qquad\textrm{ and }\qquad
\sigma_{1^b}\cdot \sigma_{\underline{a}}=\sum_{\begin{array}{c}\scriptstyle|c|=|a|+b\\\scriptstyle a_i\leq c_i\leq a_i+1\forall i\end{array}}\sigma_c 
\]
Diagrammatically, the product is given by the sum over all Young diagrams (fitting into a $k\times(n-k)$-rectangle) which are obtained from the Young diagram for $\underline{a}$ by adding $a$ boxes with at most one box added to each column (in case $\sigma_b$) resp. to each row (in case $\sigma_{1^b}$). 
\item Arbitrary Schubert classes for partitions $\underline{a}=(a_1,\dots,a_k)$ can be expressed in terms of special Schubert classes via the \emph{Giambelli formula}:
\[
\sigma_{\underline{a}}=\left|\begin{array}{ccccc}\sigma_{a_1} & \sigma_{a_1+1} & \sigma_{a_1+2} & \cdots & \sigma_{a_1+k-1}\\
\sigma_{a_2-1} & \sigma_{a_2} & \sigma_{a_2+1} & \cdots& \sigma_{a_2+k-2}\\
\sigma_{a_3-2} & \sigma_{a_3-1} & \sigma_{a_3} & \cdots& \sigma_{a_3+k-3}\\
\vdots&\vdots&\vdots&\ddots&\vdots\\
\sigma_{a_k-k+1} & \sigma_{a_k-k+2} & \sigma_{a_k-k+3} & \cdots& \sigma_{a_k} \end{array}\right|.
\]
\item In a slightly different formulation, if $\sigma_{\underline{a}}$ and $\sigma_{\underline{b}}$ are Schubert classes, then the product is given by 
\[
\sigma_{\underline{a}}\cdot\sigma_{\underline{b}}=\sum_{|c|=|a|+|b|}\gamma_{\underline{a},\underline{b};\underline{c}}\sigma_{\underline{c}}
\]
where $\gamma_{\underline{a},\underline{b};\underline{c}} =\deg(\sigma_{\underline{a}}\sigma_{\underline{b}}\sigma_{\underline{c}^\ast})$ are the \emph{Littlewood--Richardson coefficients}. 
\end{enumerate}

\subsection{Chow ring vs. $\mathbf{W}$-cohomology ring}

Now we want to set up a similar picture for the $\mathbf{W}$-cohomology ring, i.e., the torsion-free quotient of $\mathbf{I}^\bullet$-cohomology. To understand the multiplicative structure of $\mathbf{W}$-cohomology we will make use of a ring homomorphism from a suitable Chow ring. The following result is a version for $\mathbf{W}$-cohomology of what appears to be a folklore result concerning the rational cohomology of the real Grassmannians $\op{Gr}_k(\mathbb{R}^n)$. The result identifies the trivial duality part of $\mathbf{W}$-cohomology in terms of the Chow ring of a different Grassmannian via a morphism which multiplies degrees by 4. I learned about this from \cite[Section 3]{feher:matszangosz} as well as the MathOverflow question 119273 ``Real Schubert calculus'' by L{\'a}szl{\'o} Feh{\'e}r.

\begin{proposition}
\label{prop:folklore}
Let $1\leq k<n$ be natural numbers. 
\begin{enumerate}
\item
If $k(n-k)$ is even, then mapping $\op{c}_i\mapsto \op{p}_{2i}, \op{c}_j^\perp\mapsto \op{p}_{2j}^\perp$ induces a natural $\op{W}(F)$-algebra isomorphism
\[
\omega:\op{CH}^\bullet(\op{Gr}(\lfloor k/2\rfloor,\lfloor n/2\rfloor))\otimes_{\mathbb{Z}}\op{W}(F) \to \op{H}^{4\bullet}(\op{Gr}(k,n),\mathbf{W}).
\]
\item If $k(n-k)$ is odd, then mapping $\op{c}_i\mapsto \op{p}_{2i}, \op{c}_j^\perp\mapsto \op{p}_{2j}^\perp$ and $R\to\op{R}$ induces a natural $\op{W}(F)$-algebra isomorphism
\[
\omega:\op{CH}^\bullet(\op{Gr}(\lfloor k/2\rfloor,\lfloor n/2\rfloor))\otimes_{\mathbb{Z}}\op{W}(F)[R]/(R^2) \to \op{H}^{4\bullet}(\op{Gr}(k,n),\mathbf{W}).
\]
\end{enumerate}
\end{proposition}

\begin{proof}
In the case where $k(n-k)$ is even, parts (1) and (2) of Proposition~\ref{prop:invertedeta} imply that $\mathbf{W}$-cohomology (with trivial duality) has the presentation 
\[
\op{W}(F)[\op{p}_2,\dots,\op{p}_{2\lfloor k/2\rfloor},\op{p}_2^\perp,\dots,\op{p}_{2\lfloor(n-k)/2\rfloor}^\perp]/ (\op{p}\cdot\op{p}^\perp=1)
\]
If we replace $\op{p}_{2i}^{(\perp)}$ be $\op{c}_i^{(\perp)}$ and the coefficient ring $\op{W}(F)$ by $\mathbb{Z}$, we get exactly the well-known presentation of $\op{Ch}^\bullet(\op{Gr}(\lfloor k/2\rfloor,\lfloor n/2\rfloor))$. The existence of the natural ring isomorphism is then immediately clear. 

Similarly, in the second case, part (3) of Proposition~\ref{prop:invertedeta} implies that the $\mathbf{W}$-cohomology (with trivial duality) is an exterior algebra over the Chow ring of the appropriate Grassmannian, with coefficients extended to $\op{W}(F)$. 
\end{proof}

\begin{remark}
\label{rem:omegalift}
Actually, the assignment $\op{c}_i^{(\perp)}\mapsto \op{p}_{2i}^{(\perp)}$  lifts to a $\op{W}(F)$-linear map
\[
\tilde{\omega}: \op{CH}^\bullet(\op{Gr}(\lfloor k/2\rfloor, \lfloor n/2\rfloor)) \otimes_{\mathbb{Z}}\op{W}(F) \to \bigoplus_j\op{H}^{j}(\op{Gr}(k,n),\mathbf{I}^j).
\]
Note, however, that this morphism will not respect the ring structures. The relations defining the Chow ring are encoded in the Whitney sum formula $\op{c}\cdot\op{c}^\perp=1$. The image of this relation under the above assignment is 
\[
(1+\op{p}_2+\cdots+\op{p}_{2\lfloor k/2\rfloor})\cdot (1+\op{p}_2^\perp+\cdots+\op{p}_{2\lfloor (n-k)/2\rfloor}^\perp)=1
\]
This relation doesn't hold in the $\mathbf{I}^\bullet$-cohomology ring. Instead, we have the Whitney sum formula $\op{p}\cdot\op{p}^\perp=1$, which additionally contains all the odd Pontryagin classes (which are in the image of $\beta$). The above assignment fails to be a ring homomorphism because the image of the Whitney sum formula from the Chow ring will be a non-trivial $\op{I}(F)$-torsion element in $\mathbf{I}^\bullet$-cohomology.
\end{remark}

The next step is to identify the composition of $\tilde{\omega}$ with the reduction morphism $\op{H}^j(\op{Gr}(k,n),\mathbf{I}^j)\to\op{Ch}^j(\op{Gr}(k,n))$. It actually suffices to identify this composition modulo images of torsion classes. Using that, it will be possible to identify the composition of reduction and $\tilde{\omega}$ with the following morphism given by doubling partitions. 

\begin{definition}
We denote by 
\[
\gamma:\op{CH}^\bullet(\op{Gr}(\lfloor k/2\rfloor,\lfloor (n-k)/2\rfloor))\to \op{CH}^{4\bullet}(\op{Gr}(k,n))
\]
the morphism of abelian groups which maps a class corresponding to a partition $(a_1,a_2,\dots,a_r)$ to the class corresponding to the doubled partition 
\[
(2a_1,2a_1,2a_2,2a_2,\dots,2a_r,2a_r). 
\]
\end{definition}

\begin{remark}
In terms of Young diagrams, the doubling of partitions corresponds to replacing every box of the Young diagram by a $(2\times 2)$-square of boxes
\[
\Yboxdim{10pt}\Yvcentermath1
\yng(1)\mapsto \yng(2,2)
\]
\end{remark}

\begin{remark}
\label{rem:noring}
The map $\gamma:\op{CH}^\bullet(\op{Gr}(k,n))\to \op{CH}^{4\bullet}(\op{Gr}(2k,2n))$ is only a homomorphism of abelian groups. It is not compatible with the ring structures. The simplest example is as follows. In $\op{Gr}(2,4)$, we have $\Yboxdim{5pt}\Yvcentermath1\yng(1)^2=\yng(2)+\yng(1,1)$. However, in $\op{Gr}(4,8)$, we have 
\[
\Yboxdim{10pt}\Yvcentermath1
\yng(2,2)^2=\yng(4,4)+\yng(4,3,1)+\yng(4,2,2)+\yng(3,3,1,1)+\yng(3,2,2,1)+\yng(2,2,2,2)
\]
Only the first and last summand are doubled partitions, the presence of the other classes implies that partition doubling doesn't induce a ring homomorphism (even after reduction mod 2).
\end{remark}

\begin{proposition}
\label{prop:pontryaginlift}
Let $1\leq k<n$ be a natural numbers. Then the difference of the two compositions
\[
\op{CH}^j(\op{Gr}(\lfloor k/2\rfloor, \lfloor n/2\rfloor))  \xrightarrow{\tilde{\omega}} \op{H}^{4j}(\op{Gr}(k,n),\mathbf{I}^{4j}) \xrightarrow{\rho} \op{Ch}^{4j}(\op{Gr}(k,n))
\]
\[
\op{CH}^j(\op{Gr}(\lfloor k/2\rfloor, \lfloor n/2\rfloor)) \xrightarrow{\gamma} \op{CH}^{4j}(\op{Gr}(k,n))\xrightarrow{\bmod 2} \op{Ch}^{4j}(\op{Gr}(k,n))
\]
lands in the image of $\op{Sq}^2:\op{Ch}^\bullet(\op{Gr}(k,n))\to \op{Ch}^{\bullet+1}(\op{Gr}(k,n))$. 
\end{proposition}

\begin{proof}
By the Giambelli formula, we have 
\[
\sigma_{2i,2i}= \left|\begin{array}{cc} \sigma_{2i}&\sigma_{2i+1}\\ \sigma_{2i-1}&\sigma_{2i} \end{array}\right| = \sigma_{2i}^2 - \sigma_{2i+1}\sigma_{2i-1}.
\]
The classes $\sigma_{2i+1}$ and $\sigma_{2i-1}$ correspond to odd Chern classes in the Chow ring, and consequently they have lifts $\beta(\overline{\op{c}}_{2i})$ and $\beta(\overline{\op{c}}_{2i-2})$ in $\op{H}^\bullet(\op{Gr}(k,n),\mathbf{I}^\bullet)$, cf. Proposition~\ref{prop:reduction}. Note that the odd Chern classes cannot lift to non-torsion classes since all non-torsion classes live in even degrees, by Proposition~\ref{prop:invertedeta}. The reduction homomorphism $\op{H}^{4i}(\op{Gr}(k,n),\mathbf{I}^{4i})\to\op{Ch}^{4i}(\op{Gr}(k,n))$ maps the Pontryagin class $\op{p}_{2i}$ to $\overline{\op{c}}_{2i}^2$. Combined with the previous observation, we find that the reduction of  $\op{p}_{2i}-\beta(\overline{\op{c}}_{2i-2})\beta(\overline{\op{c}}_{2i})$ is exactly the class $\sigma_{2i,2i}$ of the doubled partition and this proves the claim. 
\end{proof}

\begin{remark}
In particular (up to Bockstein classes) the Pontryagin class $\op{p}_{2i}^{(\perp)}$ is actually a lift of the doubled partition for the Chern class $\underline{\op{c}}_i^{(\perp)}$ in the mod 2 Chow ring. Combining Proposition~\ref{prop:folklore} and \ref{prop:pontryaginlift}, we get an additive basis of classes in $\op{H}^{\bullet}(\op{Gr}(k,n),\mathbf{W})$ whose ``underlying Schubert varieties'' corrrespond to doubled partitions. We will see later in Section~\ref{sec:main} that the Pontryagin classes can be interpreted as Schubert cycle equipped with an additional choice of orientation. 
\end{remark}

\subsection{Multiplication of even Young diagrams} 

Now we can discuss the multiplicative structure in $\mathbf{W}$-cohomology from a diagrammatical perspective. For this, we first recall definitions of even Young diagrams from \cite[Section 2]{balmer:calmes} and then explain how the previous degree-doubling morphism explains multiplication of classes in terms of even Young diagrams. 

\begin{definition}
\label{def:evenyoung}
Let $1\leq k<n$. 
\begin{enumerate}
\item A Young diagram in $k\times(n-k)$-frame is a $k$-tuple $\Lambda=(\Lambda_1,\Lambda_2,\dots,\Lambda_k)$ of integers such that $n-k\geq\Lambda_1\geq\Lambda_2\geq\cdots\geq\Lambda_k\geq 0$. 
\item Such a diagram is called \emph{even} if all segments of the boundary which are strictly inside the $k\times (n-k)$-frame have even length. More precisely, the $k$-tuple $\Lambda$ is of the form
\[
\underbrace{\Lambda_1=\cdots=\Lambda_{d_1}}_{d_1\textrm{ terms}}>\underbrace{\Lambda_{d_1+1}=\cdots=\Lambda_{d_2}}_{d_2-d_1\textrm{ terms}}>\cdots >
\underbrace{\Lambda_{d_{p-1}+1}=\cdots=\Lambda_{d_p}}_{d_p-d_{p-1}\textrm{ terms}}=\Lambda_k
\]
and consequently, we obtain a sequence of integers $(d_1,\dots,d_p)$. There is another sequence of integers $(e_1,\dots,e_p)$ given by $e_i=n-k-\Lambda_{d_i}$. Now a diagram $\Lambda$ is called even, if the following conditions are satisfied:
\begin{enumerate}
\item $d_{i+1}-d_i$ is even for all $i=1,\dots,p-2$,
\item $e_{i+1}-e_i$ is even for all $i=1,\dots,p-1$,
\item when $0<e_1<e$ further require that $d_1$ is even,
\item when $0<e_p<e$ further require that $d_p-d_{p-1}$ is even
\end{enumerate}
\item
We call a Young diagram \emph{completely even} if all the segments of the boundary have even length (even those which lie on the frame). In the above notation, this corresponds to requiring that $d_1$ and $d_p-d_{p-1}$ are even (irrespective of the parities of $e_1$ and $e_p$).
\end{enumerate}
\end{definition}

\begin{remark}
The definition of the tuples $(d_1,\dots,d_p)$ and $(e_1,\dots,e_p)$ is illustrated in Figure~\ref{fig:bc} which is taken from \cite[Figure 3]{balmer:calmes}.

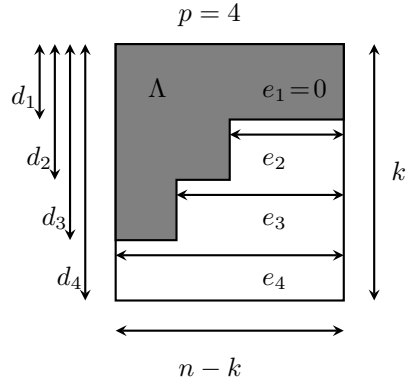
\begin{figure}[!ht]
\begin{center}
\begin{tikzpicture}[y=-1cm]
\draw[thick,black] (1.5,0.8) rectangle (4.5,4.2);
\draw[thick,arrows=stealth-stealth,black] (3,2) -- (4.5,2); 
\draw[thick,arrows=stealth-stealth,black] (2.3,2.8) -- (4.5,2.8); 
\draw[thick,arrows=stealth-stealth,black] (1.5,3.6) -- (4.5,3.6); 
\path[draw=black,thick,arrows=stealth-stealth] (4.9,4.2) -- (4.9,0.8); 
\path[draw=black,thick,arrows=stealth-stealth] (1.5,4.6) -- (4.5,4.6); 
\path[draw=black,thick,arrows=stealth-stealth] (0.5,0.8) -- (0.5,1.8); 
\path[draw=black,thick,arrows=stealth-stealth] (0.7,0.8) -- (0.7,2.6); 
\path[draw=black,thick,arrows=stealth-stealth] (0.9,0.8) -- (0.9,3.4); 
\path[draw=black,thick,arrows=stealth-stealth] (1.1,0.8) -- (1.1,4.2); 
\path[draw=black,thick,fill=white!50!black] (1.5,0.8) -- (4.5,0.8) -- (4.5,1.8) -- (3,1.8) -- (3,2.6) -- (2.3,2.6) -- (2.3,3.4) -- (1.5,3.4) -- cycle; 
\path (5,2.6) node[text=black,anchor=base west] {$k$};
\path (2.2,0.5) node[text=black,anchor=base west] {$p=4$};
\path (0.6,4) node[text=black,anchor=base west] {$d_4$};
\path (0.4,3.2) node[text=black,anchor=base west] {$d_3$};
\path (0.2,2.4) node[text=black,anchor=base west] {$d_2$};
\path (0,1.6) node[text=black,anchor=base west] {$d_1$};
\path (2.2,5.2) node[text=black,anchor=base west] {$n-k$};
\path (3.3,3.2) node[text=black,anchor=base west] {$e_3$};
\path (3.3,4) node[text=black,anchor=base west] {$e_4$};
\path (3.3,2.4) node[text=black,anchor=base west] {$e_2$};
\path (1.8,1.5) node[text=black,anchor=base west] {$\Lambda$};
\path (3.3,1.5) node[text=black,anchor=base west] {$e_1\!=\!0$};
\end{tikzpicture}
\end{center}
\caption{An example of a Young diagram in $k\times(n-k)$-frame with the relevant tuples $(d_1,\dots,d_4)$ and $(e_1,\dots,e_4)$.}
\label{fig:bc}
\end{figure}
\end{remark}

\begin{example}
\label{ex:twistclasses}
\begin{itemize}
\item 
If $k$ is even, then $\Lambda=\underbrace{(1,\dots,1)}_{k\textrm{ times}}$ is an even Young diagram in $k\times(n-k)$-frame. This is the Young diagram for the top Chern class $\op{c}_k$ of the tautological $k$-plane bundle on $\op{Gr}(k,n)$. 
\item If $n-k$ is even, then $\Lambda=(n-k,0,\dots,0)$ is an even Young diagram in $k\times (n-k)$-frame. This is the Young diagram for the top Chern class $\op{c}^\perp_{n-k}$ of the complementary $(n-k)$-plane bundle on $\op{Gr}(k,n)$. 
\item If both $k$ and $n-k$ are odd, then $\Lambda=(n-k,1,\dots,1)$ is an even Young diagram in $k\times (n-k)$-frame. This is the Young diagram for the product $\op{c}_{k-1}\op{c}_{n-k}^\perp=\op{c}_k\op{c}_{n-k-1}^\perp$. 
\end{itemize}
\end{example}

\begin{definition}
\label{def:evenmap}
Fix $1\leq k<n$. We define the following map from the set of even Young diagrams in $k\times(n-k)$-frame to $\op{H}^\bullet(\op{Gr}(k,n),\mathbf{I}^\bullet\oplus\mathbf{I}^\bullet(\det))$: 
\begin{itemize}
\item If the diagram is completely even, then it is the Young diagram for the double $\gamma(\underline{a})$ of some partition $\underline{a}$. Denoting by $\sigma_{\underline{a}}$ the Schubert class for the partition $\underline{a}$, the corresponding class in the $\mathbf{I}^n$-cohomology ring is given by $\tilde{\omega}(\sigma_{\underline{a}})$.
\item If the diagram is not completely even, then it is obtained from a completely even Young diagram by adding one of the diagrams from Example~\ref{ex:twistclasses}. The corresponding class is the product of the doubled class for the completely even part and the Euler class $\op{e}_k$ in case (1), the Euler class $\op{e}_{n-k}^\perp$ in case (2) and the class $\op{R}$ in case (3). 
\end{itemize}
\end{definition}

\begin{example}
The following Young diagram in $k\times 8$-frame with $k\geq 5$ is not completely even, since $d_1=1$ (and $d_4-d_3$ is odd for even $k$):
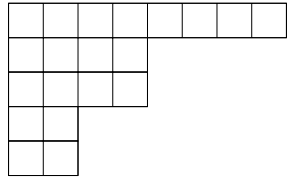
\[
\Yvcentermath1
\yng(8,4,4,2,2)
\]
It is the product of $\op{e}_8^\perp$, corresponding to the first line, with a completely even Young diagram corresponding to the class $\op{p}_4^\perp\op{p}_2^\perp-\op{p}_6^\perp$, cf. also Example~\ref{ex:giambelli}
\end{example}

\begin{remark}
We can exactly determine the appropriate Chow--Witt group that an even Young diagram $\Lambda=(\Lambda_1,\dots,\Lambda_k)$ is mapped to: the cohomological degree is given by the number $\sum_i\Lambda_i$ of boxes in the diagram (this is called the area in \cite{balmer:calmes}). This follows easily: for a partition $\underline{a}$, the corresponding Schubert class $\sigma_{\underline{a}}$ is a class in $\op{CH}^{|\underline{a}|}$, and its image under $\tilde{\omega}$ lands in cohomological degree $4|\underline{a}|$ which equals exactly the area of the doubled diagram $\gamma(\underline{a})$. 

The twist/half-perimeter $t(\Lambda)$ of an even Young diagram $\Lambda$, defined in \cite[Definition 4.3]{balmer:calmes} also determines the twist for the $\mathbf{I}^\bullet$-cohomology: the class corresponding to $\Lambda$ under the assignment in Definition~\ref{def:evenmap} has coefficients in $\mathbf{I}^\bullet$ if $t(\Lambda)=0$ and coefficients in $\mathbf{I}^\bullet(\det)$ otherwise. In the case distinction of Definition~\ref{def:evenmap}, the Young diagram has odd half-perimeter if and only if it corresponds to a product of a completely even diagram with an Euler class.
\end{remark}

\begin{proposition}
\label{prop:addbasis1}
The map of Definition~\ref{def:evenmap} provides a bijection between the even Young diagrams in $k\times(n-k)$-frame and a $\op{W}(F)$-basis of  $\op{H}^\bullet(\op{Gr}(k,n),\mathbf{W}\oplus\mathbf{W}(\det))$. 
\end{proposition}

\begin{proof}
If $k(n-k)$ is even, the part with nontwisted coefficients follows from Proposition~\ref{prop:folklore} and the fact that the ordinary Young diagrams provide an additive basis for the Chow ring. For the twisted coefficients, Proposition~\ref{prop:invertedeta} implies that any class in $\op{H}^\bullet(\op{Gr}(k,n),\mathbf{W})$ is the product of a nontwisted class with an Euler class $\op{e}_k$ or $\op{e}_{n-k}^\perp$. Since $\op{e}_k\cdot\op{e}_{n-k}^\perp=0$ only one Euler class factor can appear. These classes correspond to the even Young diagrams which are not completely even and have $t(\Lambda)=1$. In particular, everything is in the image of the map from Definition~\ref{def:evenmap}. It can be checked from the presentation of Proposition~\ref{prop:invertedeta} that different even Young diagrams with $t(\Lambda)=1$ also correspond to different classes in $\op{H}^\bullet(\op{Gr}(k,n),\mathbf{W})$.

For the case where $k(n-k)$ is odd, Proposition~\ref{prop:folklore} implies that an additive basis of $\op{H}^\bullet(\op{Gr}(k,n),\mathbf{W})$ is given by the images of an additive basis of the Chow ring, together with the images of products of such an additive basis with the class $\op{R}$. The former classes are again exactly the classes corresponding to completely even Young diagrams. The latter classes are even Young diagrams which are not completely even but have $t(\Lambda)=0$, i.e., they are obtained from an even Young diagram by adding a maximal hook. Nontwisted coefficients don't appear in this case, and there are no even Young diagrams which are not of the above form.
\end{proof}

\begin{remark}
Analogous statements on additive bases for integral cohomology of real Grassmannians mod torsion are well-known, cf. \cite{finashin:kharlamov:3planes}. It seems the classical topological result was first proved by Pontryagin \cite{pontryagin}. 
\end{remark}

Now we can formulate a complete description of the multiplication in the $\mathbf{W}$-cohomology rings of Grassmannians. To introduce appropriate notation, let $2\leq k<n$ and denote by $\mathfrak{S}_{\underline{a}}$ the image of the Schubert class $\sigma_{\underline{a}}$ under the homomorphism 
\[
\op{CH}^\bullet(\op{Gr}(\lfloor k/2\rfloor,\lfloor n/2\rfloor))\to\op{H}^{4\bullet}(\op{Gr}(k,n),\mathbf{W})
\] 
of Proposition~\ref{prop:folklore}. By Proposition~\ref{prop:addbasis1}, every class in $\op{H}^\bullet(\op{Gr}(k,n),\mathbf{W})$ is a $\op{W}(F)$-linear combination of classes of one of the four types $\mathfrak{S}_{\underline{a}}$, $\mathfrak{S}_{\underline{a}}\op{e}_k$, $\mathfrak{S}_{\underline{a}}\op{e}_{n-k}^\perp$ or $\mathfrak{S}_{\underline{a}}\op{R}$. Therefore, we only need to describe the products of these four types of classes. Essentially, everything will be easily reduced to products of the form $\mathfrak{S}_{\underline{a}}\mathfrak{S}_{\underline{b}}$ which can be computed using the folklore homomorphism of Proposition~\ref{prop:folklore}. 

\begin{theorem}
\label{thm:multw}
With the above notation, we have the following statements:
\begin{enumerate}
\item The following table explains how to compute products of the four types in terms of products of the form $\mathfrak{S}_{\underline{a}}\mathfrak{S}_{\underline{b}}$: 
\begin{center}
\begin{tabular}{|c||c|c|c|c|}
\hline
&$\mathfrak{S}_{\underline{b}}$ & $\mathfrak{S}_{\underline{b}}\op{e}_k$ & $\mathfrak{S}_{\underline{b}}\op{e}_{n-k}^\perp$ & $\mathfrak{S}_{\underline{b}}\op{R}$ \\\hline\hline
$\mathfrak{S}_{\underline{a}}$ & $\mathfrak{S}_{\underline{a}}\mathfrak{S}_{\underline{b}}$ & $\mathfrak{S}_{\underline{a}}\mathfrak{S}_{\underline{b}}\op{e}_k$ &  $\mathfrak{S}_{\underline{a}}\mathfrak{S}_{\underline{b}}\op{e}_{n-k}^\perp$ & $\mathfrak{S}_{\underline{a}}\mathfrak{S}_{\underline{b}}\op{R}$
\\\hline
$\mathfrak{S}_{\underline{a}}\op{e}_k$ & $\mathfrak{S}_{\underline{a}}\mathfrak{S}_{\underline{b}}\op{e}_k$ & $\mathfrak{S}_{\underline{a}}\mathfrak{S}_{\underline{b}}\mathfrak{S}_{1^k}$ & $0$ & ---
\\\hline
$\mathfrak{S}_{\underline{a}}\op{e}_{n-k}^\perp$ & $\mathfrak{S}_{\underline{a}}\mathfrak{S}_{\underline{b}}\op{e}_{n-k}^\perp$ & $0$ & $\mathfrak{S}_{\underline{a}}\mathfrak{S}_{\underline{b}}\mathfrak{S}_{n-k}$ & ---
\\\hline
$\mathfrak{S}_{\underline{a}}\op{R}$ & $\mathfrak{S}_{\underline{a}}\mathfrak{S}_{\underline{b}}\op{R}$ & --- & --- & $0$ \\
\hline 
\end{tabular}
\end{center}
Note that the $\mathbf{W}$-cohomology rings in question are commutative so that the order of factors doesn't play a role in the above table. The combinations in the table marked by --- are impossible: the Euler classes $\op{e}_k$ and $\op{e}_{n-k}^\perp$ require $k$ and $n-k$ to be even, respectively, but the class $\op{R}$ requires $k(n-k)$ to be odd.
\item There is an oriented version of the Pieri formula: for any Schubert class $\sigma_a$ and any integer $b$, we have 
\[
(\mathfrak{S}_{b}\cdot\mathfrak{S}_{\underline{a}})=\sum_{\begin{array}{c}\scriptstyle|\underline{c}|=|\underline{a}|+b\\\scriptstyle a_i\leq c_i\leq a_{i-1}\forall i\end{array}}\mathfrak{S}_{\underline{c}}. 
\]
\item
There is an oriented version of the Giambelli formula: 
\[
\mathfrak{S}_{(a_1,a_2,\dots,a_k)}=\left|\begin{array}{ccccc}\mathfrak{S}_{a_1} & \mathfrak{S}_{a_1+1} & \mathfrak{S}_{a_1+2} & \cdots & \mathfrak{S}_{a_1+k-1}\\
\mathfrak{S}_{a_2-1} & \mathfrak{S}_{a_2} & \mathfrak{S}_{a_2+1} & \cdots& \mathfrak{S}_{a_2+k-2}\\
\mathfrak{S}_{a_3-2} & \mathfrak{S}_{a_3-1} & \mathfrak{S}_{a_3} & \cdots& \mathfrak{S}_{a_3+k-3}\\
\vdots&\vdots&\vdots&\ddots&\vdots\\
\mathfrak{S}_{a_k-k+1} & \mathfrak{S}_{a_k-k+2} & \mathfrak{S}_{a_k-k+3} & \cdots& \mathfrak{S}_{a_k} \end{array}\right|.
\]
\item The following describes the oriented Littlewood--Richardson coefficients for multiplication of classes $\mathfrak{S}_{\underline{a}}$ and $\mathfrak{S}_{\underline{b}}$: for any three Schubert classes $\sigma_{\underline{a}}$, $\sigma_{\underline{b}}$ and $\sigma_{\underline{c}}$, we have the following equality, where $\gamma_{\underline{a},\underline{b};\underline{c}}=\det(\sigma_{\underline{a}}\sigma_{\underline{b}}\sigma_{\underline{c}}^\ast)$ are the classical Littlewood--Richardson coefficients from the Chow ring:
\[
\mathfrak{S}_{\underline{a}}\cdot\mathfrak{S}_{\underline{b}}=\sum_{|\underline{c}|=|\underline{a}|+|\underline{b}|}\gamma_{\underline{a},\underline{b};\underline{c}}\langle 1\rangle\cdot\mathfrak{S}_{\underline{c}}.
\]
\end{enumerate}
\end{theorem}

\begin{proof}
(1) The fact that all classes are $\op{W}(F)$-linear combinations of classes appearing in the table follows from Proposition~\ref{prop:invertedeta}: the relations $\op{e}_k\op{e}_{n-k}^\perp=0$, $\op{e}_k^2=\op{p}_k$, $(\op{e}_{n-k}^\perp)^2=\op{p}_{n-k}^\perp$ and $\op{R}^2=0$ imply that we always can reduce to classes appearing in the table. All the entries of the table describing the multiplication also follow from these relations. 

The items (2), (3) and (4) are simply obtained by taking the classical Pieri and Giambelli formulas resp. Littlewood--Richardson coefficients and using the ring isomorphism of Proposition~\ref{prop:folklore}, and translate the result into even Young diagrams via Definition~\ref{def:evenmap} and Proposition~\ref{prop:addbasis1}. 
\end{proof}

\begin{remark}
In particular, while the coefficients describing the intersection products in $\op{H}^\bullet(\op{Gr}(k,n),\mathbf{W})$ are quadratic forms, they are always of the simple form $m\langle 1\rangle$ with $m\in\mathbb{N}$. Essentially, everything is defined over $\op{Spec}\mathbb{Z}$ and there are not many unramified quadratic forms there. 
\end{remark}

\subsection{Examples} 

We discuss a couple of examples to illustrate the computational rules in Theorem~\ref{thm:multw}.

\begin{example}
\label{ex:gr55}
We discuss the ring structure of $\op{H}^{\bullet}(\op{Gr}(5,10),\mathbf{W})$.\footnote{We will see later that this also describes the ring structure on the Witt group $\op{W}^\bullet(\op{Gr}(5,10))$.} Note that, since $\dim\op{Gr}(5,10)=25$ is odd, there are no classes in cohomology with twisted coefficients. The additive generators have been described in \cite[Figure 16, p. 641]{balmer:calmes}. The multiplicative structure is given as follows: the first row in Fig. 16, loc.cit., are the diagrams obtained by doubling Young diagrams for Schubert classes in $\op{Gr}(2,4)$. Their products can be obtained via doubling of the corresponding computations in $\op{Gr}(2,4)$, e.g.
\[
\Yvcentermath1
\yng(2,2)^4=2\langle 1\rangle\cdot\yng(4,4,4,4)
\]
The second row in \cite[Figure 16]{balmer:calmes} contains the classes which are products of the degree 9 class $\op{R}$ and a class obtained by doubling the Young diagram for a Schubert class in $\op{Gr}(2,4)$. Since $\op{R}^2=0$, products of classes from the second row are all zero. Products of a class from the first row with a class from the second row can be obtained as product of $\op{R}$ with a product of the corresponding classes from the first row. For example, 
\[
\Yvcentermath1
\yng(5,3,3,1,1)\cdot\yng(2,2)^2=2\langle 1\rangle\cdot\yng(5,5,5,3,3)
\]
\end{example} 

\begin{example}
The class $\op{p}_2\op{e}_8^\perp$ in $\op{Gr}(6,14)$ corresponds to the following even Young diagram in $(6\times 8)$-frame:
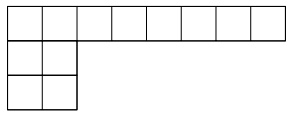
\[
\Yvcentermath1
\yng(8,2,2)
\]
The 4-fold self-intersection of this class is $\op{p}_2^4\op{p}_8^2$. It turns out that this equals the full $(6\times 8)$-rectangle with multiplicity 1: The class $\op{p}_8^2$ in $\op{Gr}(6,14)$ is a $(4\times 8)$-rectangle and by the Pieri formula only the components of $\op{p}_2^4$ which fit into a $(2\times 8)$-rectangle are relevant for the product $\op{p}_2^4\op{p}_8^2$. However, the only component of $\op{p}_2^4$ which fits into such rectangle is $\op{p}_8^2$ with multiplicity 1. 
\end{example}

\begin{example}
\label{ex:giambelli}
An example of the oriented Giambelli formula is the following one: in the notation of Theorem~\ref{thm:multw}, the doubled partition $(2a,2a,2b,2b)$ is the class $\mathfrak{S}_{a,b}$ and consequently, 
\[
\mathfrak{S}_{a,b}=\left|\begin{array}{cc}\mathfrak{S}_a&\mathfrak{S}_{a+1}\\\mathfrak{S}_{b-1}&\mathfrak{S}_b\end{array}\right|=\mathfrak{S}_a\mathfrak{S}_b-\mathfrak{S}_{a+1}\mathfrak{S}_{b-1}.
\]
The corresponding diagrammatic formula for the case $\mathfrak{S}_{2,1}$ is the following, which is the obvious double of the usual Giambelli formula for the Schubert class $\sigma_{2,1}$:
\[
\Yvcentermath1
\yng(4,4,2,2)=\yng(4,4)\cdot\yng(2,2)-\yng(6,6)
\]
In the specific case of the class $\mathfrak{S}_{2,1}\op{R}$ in  $\op{Gr}(5,10)$, the formula would simplify because the second summand doesn't fit the $(5\times 5)$-frame:
\[
\Yvcentermath1
\yng(5,5,5,3,3)=\yng(5,5,5,1,1)\cdot\yng(5,3,3,1,1)
\]
\end{example}

\section{Computations with torsion classes}
\label{sec:torsion}

In this section, we will extend the results on $\mathbf{W}$-cohomology rings from Section~\ref{sec:cohomology} to a discussion of the structure of the cohomology rings $\bigoplus_{j,\mathscr{L}}\op{H}^j(\op{Gr}(k,n),\mathbf{I}^j(\mathscr{L}))$. First, the natural reduction morphism 
\[
\bigoplus_{j,\mathscr{L}}\op{H}^j(\op{Gr}(k,n),\mathbf{I}^j(\mathscr{L}))\to \op{Ch}^j(\op{Gr}(k,n))
\]
is not generally surjective and the obstruction is given by the Steenrod square $\op{Sq}^2$. So we first discuss a description of Steenrod squares in terms of Young diagrams, essentially a version of results of Lenart \cite{lenart}. This allows to explicitly determine which classes in $\op{Ch}^j(\op{Gr}(k,n))$ lift to $\mathbf{I}^\bullet$-cohomology. The multiplicative structure is then determined by reduction to the mod 2 Chow rings, and this completes the description of the ring structure of $\mathbf{I}^\bullet$-cohomology. Then we can discuss some examples of liftable and non-liftable classes, and examples of multiplication of torsion classes in terms of classical Schubert calculus. 

\subsection{Lifting of torsion classes and non-refinable problems}

Recall that the natural projection morphisms 
\[
\widetilde{\op{CH}}^j(\op{Gr}(k,n),\mathscr{L})\to \op{CH}^j(\op{Gr}(k,n))
\,\textrm{ resp. }\, \op{H}^j(\op{Gr}(k,n),\mathbf{I}^j(\mathscr{L}))\to \op{Ch}^j(\op{Gr}(k,n))
\]
are not generally surjective. For a class $x\in\op{Ch}^\bullet(\op{Gr}(k,n))$, the obstruction to lifting to $\mathscr{L}$-twisted coefficients is given by the Steenrod square $\op{Sq}^2_{\mathscr{L}}(x)$:

\begin{proposition}
\label{prop:sq2lift}
A cycle $\sigma\in\op{CH}^i(\op{Gr}(k,n))$ lifts along the natural projection map $\widetilde{\op{CH}}^i(\op{Gr}(k,n),\mathscr{L})\to \op{CH}^i(\op{Gr}(k,n))$ if and only if $\overline{\sigma}\in\op{Ch}^i(\op{Gr}(k,n))$ is in the kernel of $\op{Sq}^2_{\mathscr{L}}$. 
\end{proposition}

\begin{proof}
By the B\"ar sequence, cf. Proposition~\ref{prop:fiberproduct}, a class lifts if and only if the class is in the kernel of the composition
\[
\op{CH}^i(\op{Gr}(k,n))\xrightarrow{\bmod 2} \op{Ch}^i(\op{Gr}(k,n))\xrightarrow{\beta_{\mathscr{L}}} \op{H}^{i+1}(\op{Gr}(k,n),\mathbf{I}^{i+1}(\mathscr{L})). 
\]
By Proposition~\ref{prop:incoh} (2), the reduction $\rho_{\mathscr{L}}$ is injective on the image of $\beta_{\mathscr{L}}$. Therefore, the class is in the kernel of the above composition if and only if its mod 2 reduction is in the kernel of $\op{Sq}^2_{\mathscr{L}}$. 
\end{proof}

The above result can be used in combination with Proposition~\ref{prop:incoh} (3) to decide if classes lift: the Steenrod square $\op{Sq}^2_{\mathscr{L}}$ can be computed using the Wu formula
\[
\op{Sq}^2(\overline{\op{c}}_j)=\overline{\op{c}}_1\overline{\op{c}}_j+(j-1)\overline{\op{c}}_{j+1},\quad 
\op{Sq}^2_{\det}(\overline{\op{c}}_j)=(j-1)\overline{\op{c}}_{j+1}.
\]
Via the Giambelli formula, this can be used to compute arbitrary Steenrod squares in $\op{Ch}^\bullet(\op{Gr}(k,n))$. The following provides an easy diagrammatic description of the Steenrod squares:

\begin{theorem}
\label{thm:steenrodlift}
Fix $1\leq k<n$. For a Young diagram $\Lambda$, fill the boxes with a checkerboard pattern, starting with a black box in the upper left corner. 
\begin{enumerate}
\item The ordinary Steenrod square $\op{Sq}_{\mathscr{O}}^2(\Lambda)$ in $\op{Ch}^\bullet(\op{Gr}(k,n))$ is given as the sum of those Young diagrams with checkerboard pattern in $k\times (n-k)$-frame which can be obtained from $\Lambda$ by adding a single white box. 
\item The twisted Steenrod square $\op{Sq}^2_{\det}(\Lambda)$ in $\op{Ch}^\bullet(\op{Gr}(k,n))$ is given as the sum of those Young diagrams with checkerboard pattern in $k\times (n-k)$-frame which can be obtained from $\Lambda$ by adding a single black box. 
\end{enumerate}
\end{theorem}

\begin{proof}
Since the Chow ring $\op{Ch}^\bullet(\op{Gr}(k,n))$ is generated by the Stiefel--Whitney classes $\overline{\op{c}}_1,\dots,\overline{\op{c}}_k$, the Steenrod squares $\op{Sq}^2_{\mathscr{O}}$ and $\op{Sq}^2_{\det}$ are determined by their values on $\overline{\op{c}}_i$ and the fact that $\op{Sq}^2_{\mathscr{O}}$ is a derivation (resp. that $\op{Sq}^2_{\det}$ has a derivation-like property). To prove the theorem, it suffices to show that the checkerboard rules agree with the Wu formula on the $\overline{\op{c}}_i$, and to show that the checkboard rules have the appropriate multiplicative behaviour. 

We first consider the Stiefel--Whitney classes $\overline{\op{c}}_i$, which correspond to the Young diagram with $i$ boxes in a single column. The checkerboard pattern ends in a white box at the bottom if $i$ is even, and a black box otherwise. To get a Young diagram, we can only add a box in the first row or at the bottom of the column. For the twisted Steenrod square, the box in the first row cannot be black, so the only possibility when we can add a black box (at the bottom of the column) would be when $i$ is even or $i=k$. This corresponds exactly to the Wu formula $\op{Sq}^2_{\op{det}}(\overline{\op{c}}_{i})=(i-1)\overline{\op{c}}_{i+1}$. For the non-twisted Steenrod square, we can always add a white box in the first row, and a white box at the bottom of the column can be added if and only if $i$ is odd. The sum of these two possibilities corresponds exactly to $\overline{\op{c}}_1\overline{\op{c}}_i$ by the Pieri formula. Consequently, the checkerboard description agrees exactly with the Wu formula $\op{Sq}^2_{\mathscr{O}}(\overline{\op{c}}_{i})=\overline{\op{c}}_1\overline{\op{c}}_i+(i-1)\overline{\op{c}}_{i+1}$. This finishes the discussion of the Stiefel--Whitney classes $\overline{\op{c}}_i$, the Stiefel--Whitney classes $\overline{\op{c}}_i^\perp$ are handled similarly. 

Now we show that the checkerboard rules have the right multiplicative behaviour: $\op{Sq}^2_{\mathscr{O}}$ is a derivation and $\op{Sq}^2_{\det}(a\cdot b)=a\cdot \op{Sq}^2_{\det}(b)+b\cdot\op{Sq}^2_{\mathscr{O}}(a)$. Actually, using the Giambelli formula, we only need to show that the derivation property is satisfied for products of an arbitrary Young diagram with one for the $\overline{\op{c}}_i$. The two multiplication statements for $\op{Sq}^2_{\mathscr{O}}(\Lambda\cdot\overline{\op{c}}_i)$ and $\op{Sq}^2_{\det}(\Lambda\cdot\overline{\op{c}}_i)$ are equivalent, because both sides of the equation differ by $\overline{\op{c}}_1\cdot\Lambda\cdot\overline{\op{c}}_i$. Therefore, it suffices to prove the formula for $\op{Sq}^2_{\det}$; this is easier, because the Wu formula for $\op{Sq}^2_{\det}(\overline{\op{c}}_i)$ is the easier one. We only sketch the combinatorial argument: on the left-hand side, the diagrammatic description of $\op{Sq}^2_{\det}(\Lambda\cdot\overline{\op{c}}_i)$ is obtained as sum of all possibilities of adding $i$ boxes according to the Pieri formula, and then adding a black box. Note that the boxes added via Pieri formula cannot be stacked vertically, but the last black box can be added below one of the boxes added before. On the right-hand side of the diagrammatic description, one of the contributions is
\[
\Lambda\cdot\op{Sq}^2_{\det}(\overline{\op{c}}_i)=\left\{\begin{array}{ll} 0 & i \textrm{ odd}\\
\Lambda\cdot \overline{\op{c}}_{i+1} & i \textrm{ even}\end{array}\right.
\] 
and $\Lambda\cdot\overline{\op{c}}_{i+1}$ is obtained by adding a column of $i+1$ boxes to $\Lambda$ such that no two boxes are in the same row. The other contribution on the right-hand side is $\op{Sq}^2_{\mathscr{O}}(\Lambda)\cdot\overline{\op{c}}_i$ which is obtained by first adding a white square to $\Lambda$ and then adding $i$ boxes according to the Pieri formula. 

If on the left-hand side ($i$ boxes according to Pieri formula plus one black box), the black box is in a different row from all the $i$ boxes added before, then this contribution is in the $\Lambda\cdot\overline{\op{c}}_{i+1}$-part of the right-hand side. If it is to the right of a box added before, the previous box must be white, and this contributes to the $\op{Sq}^2_{\mathscr{O}}(\Lambda)\cdot\overline{\op{c}}_i$-part of the right-hand side. So we have the same contributions on either side, proving the required equality. 
\end{proof}

\begin{remark}
Here is an alternative argument: we first consider the real Grassmannians $\op{Gr}_k(\mathbb{R}^n)$. By definition of the Bockstein map, the Bockstein $\beta(x)$ of a class $x\in \op{H}^j(\op{Gr}_k(\mathbb{R}^n),\mathbb{Z}/2\mathbb{Z})$ equals the boundary of a representative of $x$ in the chain complex for $\op{Gr}_k(\mathbb{R}^n)$. The criterion in Theorem~\ref{thm:steenrodlift} then follows directly from the description of the boundary map in \cite{casian:kodama}. The result follows since the real cycle class map $\op{H}^j(\op{Gr}(k,n),\overline{\mathbf{I}}^j)\cong\op{Ch}^j(\op{Gr}(k,n))\to \op{H}^j(\op{Gr}_k(\mathbb{R}^n),\mathbb{Z}/2\mathbb{Z})$ takes the Steenrod square $\op{Sq}^2$ to the Bockstein operation $\beta$, a fact that will be established in forthcoming joint work with Jens Hornbostel, Heng Xie and Marcus Zibrowius. 

On the other hand, the description of $\op{Sq}^1=\beta$ on $\op{H}^j(\op{Gr}_k(\mathbb{R}^n),\mathbb{Z}/2\mathbb{Z})$ agrees with the description of $\op{Sq}^2$ on $\op{H}^j(\op{Gr}_k(\mathbb{C}^n),\mathbb{Z}/2\mathbb{Z})$: essentially, $\op{Sq}^2$ in the complex setting is nontrivial if the attaching map is the Hopf map $\eta$, and in the real case $\op{Sq}^1$ is nontrivial if the attaching map is the natural projection $\op{S}^n\to\mathbb{RP}^n$. Therefore, the diagrammatic description above also describes $\op{Sq}^2$ for $\op{Gr}_k(\mathbb{C}^n)$. Then the cycle class map $\op{Ch}^j(\op{Gr}(k,n))\to\op{H}^j(\op{Gr}_k(\mathbb{C}^n),\mathbb{Z}/2\mathbb{Z})$ is compatible with $\op{Sq}^2$ (essentially because Voevodsky's description of the Steenrod operations is a version of the topological definition in the $\mathbb{A}^1$-homotopy category). This is another way to see that the description of $\op{Sq}^2$ on $\op{Ch}^\bullet(\op{Gr}(k,n))$ agrees with the description of attaching maps in the cell complex for $\op{Gr}_k(\mathbb{R}^n)$. Note that this agrees with the description in \cite{lenart}.
\end{remark}

\begin{example}
To illustrate the conditions in Theorem~\ref{thm:steenrodlift}, consider the following computation of Steenrod squares with checkerboard rule:
\[
\Yboxdim{10pt}
\Yvcentermath1
\newcommand\ylb{\Yfillcolour{black}}
\newcommand\ylw{\Yfillcolour{white}}
\op{Sq}^2_{\mathscr{O}}\left(\gyoung(!\ylb;!\ylw;!\ylb;!\ylw;!\ylb;,!\ylw;!\ylb;)\right)=
\gyoung(!\ylb;!\ylw;!\ylb;!\ylw;!\ylb;!\ylw;,;!\ylb;)+
\gyoung(!\ylb;!\ylw;!\ylb;!\ylw;!\ylb;,!\ylw;!\ylb;!\ylw;)
\]
\[
\Yboxdim{10pt}
\Yvcentermath1
\newcommand\ylb{\Yfillcolour{black}}
\newcommand\ylw{\Yfillcolour{white}}
\op{Sq}_{\op{det}}^2\left(\gyoung(!\ylb;!\ylw;!\ylb;!\ylw;!\ylb;,!\ylw;!\ylb;)\right)=
\gyoung(!\ylb;!\ylw;!\ylb;!\ylw;!\ylb;,!\ylw;!\ylb;,;)
\]
\end{example}

\begin{example}
We work out the Steenrod square of the class $\sigma_{a,b}$. By the Giambelli formula, we have 
\[
\sigma_{a,b}=\det\left(\begin{array}{cc}\sigma_a&\sigma_{a+1}\\ \sigma_{b-1}&\sigma_b\end{array}\right)=\sigma_a\sigma_b-\sigma_{a+1}\sigma_{b-1}.
\]
Now we have a case distinction
\[
\op{Sq}^2_{\mathscr{O}}\sigma_{a,b}=\left\{\begin{array}{ll} -\sigma_{a+1}\sigma_b-\sigma_{b-1}\sigma_{a+2} = \sigma_{a+1,b} & a, b \textrm{ odd}\\
\sigma_a\sigma_{b+1}-\sigma_{b-1}\sigma_{a+2} = \sigma_{a,b+1}+\sigma_{a+2,b} & a,b-1 \textrm{ odd}\\
0 & a-1,b\textrm{ odd}\\
\sigma_a\sigma_{b+1}+\sigma_{a+1}\sigma_b & a,b\textrm{ even}
\end{array}\right.
\]
The simplifications are obtained by expanding the product $\sigma_i\sigma_j$ as $\sigma_{i+j}+\sigma_{i+j-1,1}+\dots+\sigma_{i+j-\min(i,j),\min(i,j)}$ and cancelling corresponding terms. Note that in case $a,b$ even, the result can be simplified to $\sigma_{a,b+1}$ if $a>b$ and $0$ if $a=b$. This fits exactly with the geometric description: in the first row a white square can be added if and only if $a$ is odd, and in the second row a white square can be added if and only if $b$ is even. In the third, we can never add a white square. 
\end{example}

Actually, the triviality of the Steenrod squares can be identified by means of an elementary size condition on the Young diagrams. The condition is exactly the one discussed in \cite[Section 4]{balmer:calmes}; in particular, the vanishing of the Steenrod square $\op{Sq}^2_{\mathscr{L}}$ for a Schubert variety is equivalent to the relative canonical bundle of a resolution of singularities being isomorphic to the pullback of the corresponding line bundle $\mathscr{L}$ from the Grassmannian.

\begin{proposition}
\label{prop:compareBC}
Fix $1\leq k<n$ and let $\Lambda$ be a Young diagram. In the notation from Definition~\ref{def:evenyoung} and Figure~\ref{fig:bc}, we have the following: 
\begin{enumerate}
\item
Then $\op{Sq}^2_{\mathscr{O}}(\Lambda)=0$ if and only if
\begin{enumerate}
\item  for all $1\leq i<p$ we have $d_{i+1}-d_i+e_{i+2}-e_{i+1}$ is even, and 
\item either ($n-k-e_1$ and $d_1+e_2-e_1$ are even) or ($e_1=0$ and both $n-k$ and $d_1+e_2-e_1$ are odd). 
\end{enumerate}
\item
Then $\op{Sq}^2_{\det}(\Lambda)=0$ if and only if 
\begin{enumerate}
\item  for all $1\leq i<p$ we have $d_{i+1}-d_i+e_{i+2}-e_{i+1}$ is even, and 
\item either ($n-k-e_1$ is odd and $d_1+e_2-e_1$ is even) or ($e_1=0$, $n-k$ is even and $d_1+e_2-e_1$ is odd). 
\end{enumerate}
\end{enumerate}
\end{proposition}

\begin{proof}
We prove the first case, the second is similar. The Steenrod square $\op{Sq}^2_{\mathscr{O}}(\Lambda)$ vanishes if, after filling the Young diagram $\Lambda$ with a checkerboard pattern, there is no way to add a white box and still get a Young diagram in $k\times(n-k)$-frame which has a checkerboard pattern. The places where can possibly add a white box to get a Young diagram would be at the end of the first row, or in the corners of hooks; in the notation of Figure~\ref{fig:bc}, such a box would be at the coordinates $(d_i+1,n-k-e_{i+1}+1)$. We cannot add a box in the first row if and only if the length is either even (thus ending in a white square) or has length $n-k$ (so that adding a square would leave the $k\times(n-k)$-frame). We cannot add boxes in the hook corners if and only if the boxes at the coordinates $(d_i,n-k-e_{i+1}+1)$ resp. $(d_i+1,n-k-e_{i+1})$ are white. Now we note that the boxes $(d_i+1,n-k-e_{i+1})$ and $(d_{i+1},n-k-e_{i+2}+1)$ have the same color if and only if $d_{i+1}-d_i+e_{i+2}-e_{i+1}$ is even. Moreover, the box at the end of the first row has the same color as the one in the rightmost hook if and only if $d_1+e_2-e_1$ is even. Therefore, if the first row has even length, then it suffices to require $d_{i+1}-d_i+e_{i+2}-e_{i+1}$ and $d_1+e_2-e_1$ to be even. However, if the first row has odd length $n-k$ (equivalently $e_1=0$), then we need to require $d_1+e_2-e_1$ to be odd.
\end{proof}

\begin{proposition}
\label{prop:ehresmann}
Fix $1\leq k<n$. The Steenrod squares $\op{Sq}^2_{\mathscr{L}}(\sigma_\Lambda)$ of Schubert classes $\sigma_\Lambda\in\op{Ch}^\bullet(\op{Gr}(k,n))$ form a set of additive generators for the $\op{I}(F)$-torsion in  $\bigoplus_{j,\mathscr{L}}\op{H}^j(\op{Gr}(k,n),\mathbf{I}^j(\mathscr{L}))$. 
\end{proposition}

\begin{proof}
This follows directly from Proposition~\ref{prop:incoh}. 
\end{proof}

At this point, getting an additive generating set as above is the best statement we can make (but the previous discussion of the Steenrod squares allows to make the generating set fairly explicit). It seems more difficult to provide a $\mathbb{Z}/2\mathbb{Z}$-basis for the torsion part of the $\mathbf{I}^\bullet$-cohomology. The following example shows that we cannot expect to get an additive basis simply by taking lifts of Schubert classes. 

\begin{example}
One of the easiest such examples appears in $\op{H}^2(\op{Gr}(2,5),\mathbf{I}^2)$. The Chow group $\op{Ch}^2(\op{Gr}(2,5))$ is generated by the two classes $\overline{\op{c}}_2$ and $\overline{\op{c}}_2^\perp$. Both have nontrivial $\op{Sq}^2_{\mathscr{O}}$ and hence do not lift to $\mathbf{I}^2$-cohomology. However, their non-twisted Steenrod squares are both equal to the hook $\sigma_{2,1}$ and consequently $\overline{\op{c}}_2+\overline{\op{c}}_2^\perp=\overline{\op{c}}_1^2$ lifts, namely to the element $\beta_{\mathscr{O}}(\overline{\op{c}}_1)$. 
\end{example}

\subsection{Multiplication of torsion classes}

We finally discuss the multiplicative structure of the $\mathbf{I}^\bullet$-cohomology ring. By Proposition~\ref{prop:addbasis1}, the even Young diagrams provide a $\op{W}(F)$-basis for the free part of $\bigoplus_{j,\mathscr{L}}\op{H}^j(\op{Gr}(k,n),\mathbf{I}^j(\mathscr{L})$. The remaining part is $\op{I}(F)$-torsion, given exactly as the image of $\beta_{\mathscr{L}}$, cf. Proposition~\ref{prop:incoh}. In particular, any class $x$ in $\mathbf{I}^\bullet$-cohomology has a unique decomposition $x=x_{\op{tor}}+x_{\op{ev}}$ as a sum of an element $x_{\op{ev}}$ in the $\op{W}(F)$-subalgebra generated by even Young diagrams and an element $x_{\op{tor}}$ which is $\op{I}(F)$-torsion.

Now if we have two classes $x=x_{\op{tor}}+x_{\op{ev}}$ and $y=y_{\op{tor}}+y_{\op{ev}}$ in $\mathbf{I}^\bullet$-cohomology, we can determine their product $x\cdot y$ as follows. The product $x\cdot y$ is given by $x_{\op{ev}}\cdot y_{\op{ev}}+x_{\op{tor}}\cdot y_{\op{tor}}+x_{\op{tor}}\cdot y_{\op{ev}}+x_{\op{ev}}\cdot y_{\op{tor}}$, and all summands except possibly the first one are torsion. Non-torsion contributions can only arise from $x_{\op{ev}}\cdot y_{\op{ev}}$ but as the example in \ref{rem:noring} shows, $x_{\op{ev}}\cdot y_{\op{ev}}$ can also contain torsion contributions. The non-torsion contributions in $x_{\op{ev}}\cdot y_{\op{ev}}$ can be determined using Theorem~\ref{thm:multw}. The overall class of the product can then be determined by reduction: 

\begin{proposition}
Fix $1\leq k<n$, and let $x,y$ be two classes in $\mathbf{I}^\bullet$-cohomology. Then the reduction of $x\cdot y$ is given by the product of the reductions of $x$ and $y$ in $\op{Ch}^\bullet(\op{Gr}(k,n))$. 
\end{proposition}

\begin{proof}
This follows directly from Proposition~\ref{prop:incoh}. 
\end{proof}

In particular, products with torsion classes are determined completely on the reduction mod 2, and these can be evaluated by classical Young diagram calculations. As a relevant consequence for Schubert calculus (where the intersection product lands in top degree), we have the following: 

\begin{corollary}
Let $F$ be a field of characteristic $\neq 2$, fix $1\leq k<n$ and assume $n$ is even. Any product $x_1\dots x_l\in \op{H}^{k(n-k)}(\op{Gr}(k,n),\mathbf{I}^{k(n-k)})$ involving an $\op{I}(F)$-torsion class is trivial. 
\end{corollary}

\begin{proof}
Under the assumption that $n$ is even, we have $\op{H}^{k(n-k)}(\op{Gr}(k,n),\mathbf{I}^{k(n-k)})\cong\op{W}(F)$, i.e., the top $\mathbf{I}^\bullet$-cohomology group is $\op{W}(F)$-torsion-free. Any product involving $\op{I}(F)$-torsion is again $\op{I}(F)$-torsion and consequently trivial.  
\end{proof}

As an example how to compute products in $\mathbf{I}^\bullet$-cohomology, we resume the discussion of the example in Remark~\ref{rem:noring}. 

\begin{example}
We consider the square of the class $\sigma_{2,2}$ in $\op{Ch}^8(\op{Gr}(4,8))$. We will see in Section~\ref{sec:geometricreps} that the class $\sigma_{2,2}$ has a canonical lift to $\op{H}^4(\op{Gr}(4,8),\mathbf{I}^4)$ given by a canonical orientation on the normal bundle of the smooth Schubert variety $\Sigma_{2,2}$. The corresponding class $\sigma_{2,2}^2$ in $\op{H}^8(\op{Gr}(4,8),\mathbf{I}^8)$ has two non-torsion contributions (computed using Theorem~\ref{thm:multw}): the classes $\sigma_{4,4}$ and $\sigma_{2,2,2,2}$ corresponding to even Young diagrams with trivial twist. The torsion contribution to $\sigma_{2,2}^2$ can be computed in the mod 2 Chow ring, cf. Remark~\ref{rem:noring}; it is given by subtracting from $\sigma_{2,2}^2$ in the mod 2 Chow ring the above even Young diagrams representing the non-torsion part. The class that remains is $\sigma_{4,3,1}+\sigma_{4,2,2}+\sigma_{3,3,1,1}+\sigma_{3,2,2,1}$, and this is a torsion class with non-twisted coefficients. Note that none of the individual Schubert classes in this sum lifts because they all have non-trivial Steenrod square. However, the full sum lifts because
\begin{eqnarray*}
&&\op{Sq}^2_{\mathscr{O}}(\sigma_{4,3,1})+\op{Sq}^2_{\mathscr{O}}(\sigma_{4,2,2})+\op{Sq}^2_{\mathscr{O}}(\sigma_{3,3,1,1})+\op{Sq}^2_{\mathscr{O}}(\sigma_{3,2,2,1})\\ &=&
\sigma_{4,3,2}+\sigma_{4,3,1,1}+\sigma_{4,3,2}+\sigma_{4,2,2,1}+\sigma_{4,3,1,1}+\sigma_{3,3,2,1}+\sigma_{4,2,2,1}+\sigma_{3,3,2,1}\\&=&0
\end{eqnarray*}
The torsion lift is uniquely determined by its reduction mod 2, the lift is given as $\beta_{\mathscr{O}}(\sigma_{4,2,1}+\sigma_{3,2,1,1})$
\end{example}

\section{Recollection on Witt groups and cohomology}
\label{sec:wittbasics}

In this section, we provide a short recollection on Witt groups as well as $\mathbf{I}^n$-cohomology. Most important for our later discussions will be the definition of push-forwards (and the dependence on choices of orientations) and the intersection products. We discuss the natural morphism from Witt groups to $\mathbf{I}^\bullet$-cohomology and its compatibility with push-forwards and intersection products. The compatibility with the pushforwards implies that the results from Section~\ref{sec:cohomology} also describe the multiplication on the twisted Witt groups, and the compatibility with intersection products is an oriented version of Serre's Tor formula for intersection multiplicities. 

\subsection{Witt groups and comparison morphism}

We shortly recall some relevant definitions from the theory of Witt groups of schemes, cf. in particular \cite{balmer}. For a category $\mathcal{C}$ with duality (given by an involutive endofunctor $(-)^\vee:\mathcal{C}^{\op{op}}\to\mathcal{C}$ with an explicit identification $\varpi:\op{id}_{\mathcal{C}}\xrightarrow{\cong} (-)^{\vee\vee}$), a symmetric space in $\mathcal{C}$ is a pair $(X,\phi)$ with $X$ an object from $\mathcal{C}$ and an isomorphism $\phi:X\to X^\vee$ such that 
\[
\xymatrix{
X\ar[r]^\phi \ar[d]^\cong_{\varpi_X} & X^\vee\ar[d]^=\\
X^{\vee\vee} \ar[r]_{\phi^\vee} &  X^\vee
}
\]
The Witt group of a category with duality is given as the quotient of the isometry classes of symmetric spaces modulo metabolic spaces. There is a similar notion of triangulated categories with duality, together with the corresponding notion of Witt group, called triangular Witt groups, cf. \cite[Section 1.4]{balmer}. 

For the present work, we are primarily interested in the following situation. For a smooth scheme $X$ over a field $F$ of characteristic $\neq 2$ and a line bundle $\mathscr{L}$ over $X$, one can consider the category of vector bundles on $X$ equipped with the duality given by $\mathscr{E}\mapsto\op{Hom}(\mathscr{E},\mathscr{O}_X)\otimes_{\mathscr{O}_X}\mathscr{L}$. This exact category with duality produces the classical \emph{Witt groups of schemes}. One can also consider the \emph{derived Witt groups}, which arise from the triangulated category of perfect complexes with the derived version of the above duality. It is also possible to consider the category of bounded complexes of quasi-coherent sheaves with coherent cohomology, with the same duality as above; this gives rise to the \emph{coherent Witt groups} of $X$. In the case of smooth schemes, all these constructions produce isomorphic Witt groups because coherent sheaves have finite resolutions by vector bundles. 

A morphism of categories with duality is a functor $\mathcal{C}\to\mathcal{D}$ with an explicit isomorphism $F\circ (-)^{\vee_{\mathcal{C}}}\xrightarrow{\cong} (-)^{\vee_{\mathcal{D}}}\circ F$. The image of a symmetric space $\phi:X\to X^\vee$ is then given by $F(X)\xrightarrow{F(\phi)} F(X^\vee)\cong F(X)^\vee$, where the latter isomorphism is the one from the definition of morphism of categories with duality. Under apppropriate assumptions, a morphism of schemes $f:X\to Y$ gives rise to functors ${\op{L}}f^\ast$ and ${\op{R}}f_\ast$ of categories of quasi-coherent sheaves or perfect complexes, which preserves the duality. This gives rise to pullback and push-forward morphisms for (ordinary, derived, coherent) Witt groups of schemes, cf. \cite{calmes:hornbostel}.

There is also a base-change formula, cf. \cite[Theorem 5.4]{calmes:hornbostel}: for a pullback diagram of schemes 
\[
\xymatrix{
W \ar[r]^{g'} \ar[d]_{f'} & X \ar[d]^f \\
Y\ar[r]_g &Z,
}
\]
there is a natural morphism of functors $\epsilon: {\op{L}}f^\ast\circ {\op{R}}g_\ast \to {\op{R}}g'_\ast\circ {\op{L}}(f')^\ast$ which is an isomorphism if the diagram is homologically transversal, cf. the discussion in Section~\ref{sec:intersection}. This will be relevant for computation of intersection products.

For coherent and derived Witt groups, the derived tensor product of complexes gives rise to duality-preserving functors and consequently to pairings in triangular Witt groups, cf. \cite{gille:nenashev}. This is usually called the $\star$-product. For complexes concentrated in a single degree, the product on Witt groups reduces to the ordinary tensor product of symmetric forms as defined by Knebusch, cf. the discussion in \cite[Section 2]{balmer:gille}.

We recall the natural morphisms from Witt-groups to $\mathbf{I}^\bullet$-cohomology from \cite[Section 7]{fasel:chowwitt}. 
Let $Z\subset X$ be a closed subset of pure codimension $i$. Denoting by $\op{GW}_Z^i(X)$ the $i$-th Grothendieck--Witt group of perfect complexes on $X$ with support in $Z$, there is a homomorphism
\[
\op{GW}_Z^i(X)\to\op{GW}^i(\op{Der}^{\op{b}}(X)^{(i)})\to\op{GW}^i(\op{Der}^{\op{b}}_i(X)) 
\]
which is obtained as composition of the inclusion of $\op{Der}_Z^{\op{b}}(X)$ in the category $\op{Der}^{\op{b}}(X)^{(i)}$ of complexes with support in codimension $\leq i$ followed by the projection to $\op{Der}^{\op{b}}_i(X)=\op{Der}^{\op{b}}(X)^{(i)}/\op{Der}^{\op{b}}(X)^{(i+1)}$. The composition 
\[
\op{GW}^i(\op{Der}^{\op{b}}(X)^{(i)})\to \op{GW}^i(\op{Der}^{\op{b}}_i(X))\to \op{W}^{i+1}(\op{Der}^{\op{b}}(X)^{(i+1)})
\]
is zero, and consequently, there is a homomorphism $\op{GW}^i_Z(X)\to\widetilde{\op{CH}}^i(X)$. The same argument also furnishes a homomorphism 
\[
\alpha:\op{W}^i_Z(X,\mathscr{L})\to \op{H}^i(X,\mathbf{I}^i(\mathscr{L})).
\]
To accomodate the additional line bundle twist by $\mathscr{L}$, simply use the derived categories above with the duality given by $\op{Hom}_{\mathscr{O}_X}(-,\mathscr{L})$. 

\subsection{Pushforwards and orientations}
\label{sec:pushforward}

We shortly discuss push-forwards in cohomology with coefficients in  $\mathbf{I}^\bullet$ and  $\mathbf{W}$. 

For a proper morphism $f:X\to Y$ between smooth schemes over $F$, there are push-forward maps, cf. \cite[Corollaire 10.4.5]{fasel:chowwitt} or \cite[Section 3]{calmes:fasel}: 
\[
\op{H}^n(X,\mathbf{I}^n(\omega_{X/Y}\otimes f^\ast\mathscr{L}))\to \op{H}^{n-d}(Y,\mathbf{I}^{n-d}(\mathscr{L}))
\]
where $d=\dim X-\dim Y$ and $\omega_{X/Y}=\omega_{X/F}\otimes f^\ast \omega_{Y/F}^\vee$ is the relative canonical bundle. The push-forward is defined in \cite[Section 8]{fasel:memoir} on the level of Gersten--Rost--Schmid complexes $\op{C}^n(X,\mathbf{I}^n(\omega_{X/F}))\to \op{C}^{n-d}(Y,\mathbf{I}^{n-d}(\omega_{Y/F}))$: for a point $x\in X$ with image $y=f(x)$ such that the residue field extension $F(y)\subseteq F(x)$ is finite, the corresponding morphism on the Witt group coefficients is given by the transfers $\op{tr}_x^y:\op{W}(F(x),\omega_{F(x)/F})\to\op{W}(F(y),\omega_{F(y)/F})$. 

The line bundle twists in the above are really necessary. Moreover, for isomorphic line bundles $\mathscr{L}_1$ and $\mathscr{L}_2$, the identification 
$\op{H}^{n}(X,\mathbf{I}^{n}(\mathscr{L}_1))\cong \op{H}^{n}(X,\mathbf{I}^{n}(\mathscr{L}_2))$ depends on the choice of isomorphism $\mathscr{L}_1\cong\mathscr{L}_2$. In particular, to get a pushforward $\op{H}^n(X,\mathbf{I}^n)\to \op{H}^{n-d}(Y,\mathbf{I}^{n-d}(\mathscr{L}))$ for a proper map $f:X\to Y$, one needs to choose an isomorphism $\omega_{X/Y}\otimes f^\ast\mathscr{L} \cong \mathscr{O}_X$. Actually, since everything is well-defined up to squares, it suffices to choose an identification $\omega_{X/Y}\otimes f^\ast \mathscr{L}\cong\mathscr{N}^2$ for some line bundle $\mathscr{N}$, i.e., the pushforward maps for $\mathbf{I}^\bullet$-cohomology depend on the choice of orientations, cf. \cite[Definition 16]{kass:wickelgren}:

\begin{definition}
Let $X$ be a smooth scheme and let $\mathscr{L}$ be a line bundle over $X$. An \emph{orientation} of $\mathscr{L}$ is a pair $(\mathscr{N},\phi)$ where $\mathscr{N}$ is a line bundle over $X$ and $\phi:\mathscr{N}\otimes\mathscr{N}\cong \mathscr{L}$ is an isomorphism. Two orientations $(\mathscr{N}_1,\phi_1)$ and $(\mathscr{N}_2,\phi_2)$ are equivalent if there exists an isomorphism $\alpha:\mathscr{N}_1\to\mathscr{N}_2$ such that $\phi'\circ(\alpha\otimes\alpha) =\phi$. 
\end{definition}

\begin{example}
If the line bundle $\mathscr{L}$ has a global invertible section $\sigma:\mathscr{O}_X\to \mathscr{L}$, then we get an orientation $\mathscr{O}_X\otimes\mathscr{O}_X\cong \mathscr{O}_X\xrightarrow{\sigma}\mathscr{L}$. Two such orientations are equivalent if their quotient is the square of a global invertible section. For a finite-dimensional vector space $V$ over $F$, we recover the classical concept of orientation as a choice of generator of $\bigwedge^{\dim V}V$. Two orientations are equivalent if the corresponding generators differ by the square of a unit in $F^\times$. 
\end{example}

A choice of orientation $\mathscr{N}\otimes\mathscr{N}\xrightarrow{\cong}\mathscr{L}$ provides an isomorphism $\op{W}(X,\mathscr{O})\cong\op{W}(X,\mathscr{L})$ which maps a form $P_\bullet\otimes P_\bullet\to\mathscr{O}$ (alternatively written as $P_\bullet\xrightarrow{\simeq} \op{Hom}(P_\bullet,\mathscr{O})$) to the form $(P_\bullet\otimes\mathscr{N})\otimes (P_\bullet\otimes \mathscr{N})\to\mathscr{L}$ (alternatively written as $P_\bullet\otimes\mathscr{N}\xrightarrow{\simeq} \op{Hom}(P_\bullet\otimes\mathscr{N},\mathscr{O})\otimes\mathscr{L}$). The identifications for $\mathbf{I}^\bullet$-cohomology and $\mathbf{W}$-cohomology can be made similarly precise.

Finally, we discuss the compatibility of the natural comparison homomorphism $\alpha:\op{W}^j_Z(X,\mathscr{L})\to\op{H}^j(X,\mathbf{I}^j(\mathscr{L}))$ with basic maps. The compatibility of the comparison homomorphism with pullbacks is established in \cite[Remark 7.5]{fasel:chowwitt}. To check compatibility with pushforwards, we trace through the definitions: 

\begin{proposition}
\label{prop:compatpush}
Let $F$ be a field of characteristic $\neq 2$, let $f:X\to Y$ be a proper morphisms of smooth schemes, and set $d=\dim X-\dim Y$. Then there is a commutative diagram
\[
\xymatrix{
\op{GW}^i(X,\omega_{X/Y}\otimes f^\ast \mathscr{L}) \ar[r]^>>>>{f_\ast} \ar[d]_{\alpha_X}  & \op{GW}^{i-d}(Y,\mathscr{L}) \ar[d]^{\alpha_Y}\\
\widetilde{\op{CH}}^i(X,\omega_{X/Y}\otimes f^\ast\mathscr{L}) \ar[r]_>>>>{f_\ast} & \widetilde{\op{CH}}^{i-d}(Y,\mathscr{L}).
}
\]
\end{proposition}

\begin{proof}
For a finite morphism $f:X\to Y$ (between smooth schemes, say), the push-forward on Chow--Witt-rings and $\mathbf{I}^\bullet$-cohomology is induced from the push-forward on Witt groups, cf. \cite[Corollaries 5.3.5--5.3.7]{fasel:memoir} resp. \cite[Corollary 6.3.10]{fasel:memoir} for the case of smooth schemes which is more relevant for us. In particular, it follows from the definition of push-forwards on Chow--Witt groups in \cite{fasel:memoir} that the comparison morphism  between Grothendieck--Witt groups and Chow--Witt rings are compatible with push-forwards along finite morphisms. Since the push-forward on Chow--Witt groups for general proper morphisms $f:X\to Y$ is also defined by reduction to the case of finite morphisms, cf. \cite[Section 8]{fasel:memoir}, we get the claim. 
\end{proof}

\subsection{Intersection products and oriented Tor formula}
\label{sec:torformula}

We shortly recall the definition of the intersection products in Chow--Witt rings and $\mathbf{I}^\bullet$-cohomology from \cite{fasel:chowwitt}: first, there is an exterior product 
\[
\boxtimes:\widetilde{\op{CH}}^i(X,\mathscr{L})\times \widetilde{\op{CH}}^j(X,\mathscr{M})\to \widetilde{\op{CH}}^{i+j}(X\times X,\op{pr}_1^\ast(\mathscr{L})\otimes\op{pr}_2^\ast(\mathscr{M})).
\]
On the other hand, we have a pullback morphism 
\[
\Delta^!:\widetilde{\op{CH}}^{i+j}(X\times X,\op{pr}_1^\ast(\mathscr{L})\otimes\op{pr}_2^\ast(\mathscr{M}))\to 
\widetilde{\op{CH}}^{i+j}(X,\mathscr{L}\otimes\mathscr{M})
\]
for the embedding $\Delta:X\hookrightarrow X\times X$ (which is a closed immersion between smooth schemes). The intersection product is then given as  the composition of these two operations: it maps a pair $x_i\in \widetilde{\op{CH}}^{n_i}(X,\mathscr{L}_i)$, $i=1,2$, to $\Delta^!(x_1\boxtimes x_2)$. 

A different way to view the construction of the pullback is given in \cite[Section 2]{AsokFaselEuler}, in particular the discussion leading up to Theorem 2.3.4. There is also a base-change formula for Chow--Witt rings  $\mathbf{I}^\bullet$-cohomology, cf. \cite[Theorem 2.4.1]{AsokFaselEuler}, which can be used to compute some intersection products, cf. the proof of Theorem~\ref{thm:types}. 

The relation between the $\star$-product on (Grothendieck--)Witt groups and the intersection product on Chow--Witt-rings or $\mathbf{I}^\bullet$-cohomology is given by the following commutative diagram which can be interpreted as an oriented analogue of Serre's Tor-formula, cf. \cite[Theorem 7.6]{fasel:chowwitt}:
\[
\xymatrix{
\op{GW}^n_V(X)\times \op{GW}^m_W(X)\ar[rr]^\star \ar[d]_{\alpha_V\times\alpha_W} && \op{GW}^{m+n}_{V\cap W}(X) \ar[d]^{\alpha_{V\cap W}} \\
\widetilde{\op{CH}}^n(X)\times \widetilde{\op{CH}}^m(X) \ar[rr]_\times && \widetilde{\op{CH}}^{n+m}(X).
}
\]

Recall that the classical Tor formula in intersection theory is the following: if $Y$ and $Z$ are closed subschemes of a smooth scheme $X$ and $C$ is a component of the intersection $Y\cap Z$ then the multiplicity of $C$ in the intersection of $Y$ and $Z$ is given by 
\[
i(Y,Z;C)=\sum_{j\geq 0}(-1)^j\op{len}_C\left(\op{Tor}_j^X(\mathscr{O}_Y,\mathscr{O}_Z)\right).
\]
In particular, multiplicities in intersection products can be computed by means of the product structure on K-groups which is given in terms of derived tensor products of complexes of coherent sheaves. In this setting, it is clear why homological transversality is useful: the Tor-formula simply reduces to the term in degree $0$ which is the length of the intersection subscheme. 

The oriented Tor-formula, as expressed in the commutative diagram above, now allows to compute some intersection products in the Chow--Witt ring: the lower composition is the intersection product of two cycles which are in the image of the natural comparison map $\alpha:\op{GW}^i_Z(X,\mathscr{L})\to\widetilde{\op{CH}}^i_Z(X,\mathscr{L})$, and the upper composition is the cycle associated to the $\star$-product, which is itself given in terms of the derived tensor products of forms on perfect complexes over $X$, with $\mathscr{L}$-twisted duality. Again, in the presence of homological transversality, the oriented Tor-formula simplifies: the corresponding derived tensor products are ordinary tensor products, and in the simplest case, when we only consider structure sheaves of closed subvarieties, the $\star$-product is then simply a form on the structure sheaf of the intersection subscheme which encodes a ``refined length'' containing additional orientation information.

\section{Oriented Schubert classes} 
\label{sec:geometricreps}

In this section, we can now write down geometric representatives for the classes in the $\mathbf{I}^\bullet$-cohomology ring. As in \cite{balmer:calmes}, the representatives are obtained by push-forward of 1 from resolutions of singularities of Schubert varieties. While the push-forward usually depends on choices of relative orientations, we will actually show that the relative canonical bundles of these resolutions have canonical (relative) orientations. This implies that Schubert classes for even Young diagrams have canonical lifts to the Chow--Witt ring. 

\subsection{Resolutions of Schubert varieties}

We first discuss the relevant resolutions of singularities for Schubert varieties in Grassmannians, including some geometric facts concerning the descriptions of relative canonical bundles. Most of this is a recollection of statements from \cite{balmer:calmes}. These statements will be required for the discussion of canonical orientations in the next subsection. 

Let $F$ be a field of characteristic $\neq 2$. Fix $1\leq k<n$, and let $V$ be an $n$-dimensional $F$-vector space. Let $\Lambda$ be a Young diagram in $k\times(n-k)$-frame, then we have associated tuples $(d_1,\dots,d_p=k)$ and $(e_1,\dots,e_p)$ as in Definition~\ref{def:evenyoung} and Figure~\ref{fig:bc}. For a given full flag \[
\mathcal{F}: 0\leq V_1\leq \cdots\leq V_{n-1}\leq V
\]
with $\dim_F V_i=i$, the Schubert variety 
\[
\Sigma_\Lambda(\mathcal{F})=\left\{ W\in \op{Gr}(k,V)\mid \dim_F(V_{n-k+i-a_i}\cap W)\geq i \right\}
\]
corresponding to $\Lambda$ is a representative of the Schubert class corresponding to the Young diagram $\Lambda$. 

\begin{definition}
\label{def:flagresolution}
For the partial flag $0\leq V_{d_1+e_1}\leq \cdots\leq V_{d_p+e_p}\leq V$, we can consider the following Schubert variety in the flag variety $\op{Fl}(d_1,\dots,d_p=k;n)$:
\[
\op{Fl}(\underline{d},\underline{e},\mathcal{F}):=\left\{ P_{d_1}\leq P_{d_2}\leq\cdots \leq P_{d_p}\in \op{Fl}(\underline{d};n)\mid
P_{d_i}\leq V_{d_i+e_i}\right\}
\]
Alternatively, we use the notation $\op{Fl}(\Lambda)$ for this closed subvariety of $\op{Fl}(\underline{d};n)$. There is a natural projection morphism 
\[
p_\Lambda:\op{Fl}(\Lambda)\to \op{Gr}(k,n): (P_{d_1},\dots,P_{d_p})\mapsto P_{d_p}. 
\]
The variety $\op{Fl}(\Lambda)$ is smooth and its image is the Schubert variety $\Sigma_\Lambda(\mathcal{F})$. The morphism $p_\Lambda$ is proper and birational, and hence is a resolution of singularities of $\Sigma_\Lambda(\mathcal{F})$. 
\end{definition}

On the flag variety $\op{Fl}(d_1,\dots,d_p;n)$, we have the tautological bundles $\mathscr{P}_{d_i}$. The fiber of the bundle $\mathscr{P}_{d_i}$ at the point corresponding to the flag $P_{d_1}\leq\cdots\leq P_{d_p}$ is $P_{d_i}$. These bundles form a filtration $\mathscr{P}_{d_1}\leq\cdots\leq\mathscr{P}_{d_p}$ of the trivial bundle $\mathscr{O}_{\op{Fl}}^{\oplus n}$. We also have the corresponding tautological subquotient  bundles $\mathscr{S}_i=\mathscr{P}_{d_i}/\mathscr{P}_{d_{i-1}}$. 

Next, we need to describe the tangent bundles of flag varieties as well as the relative canonical bundles for the resolutions $p_\Lambda:\op{Fl}(\Lambda)\to\op{Gr}(k,n)$. In the special case of the Grassmannian, we have an exact sequence  $0\to\mathscr{S}\to\mathscr{O}_{\op{Gr}}^{\oplus n}\to\mathscr{Q}\to 0$ which decomposes the trivial rank $n$ bundle as extension of the tautological rank $k$ subbundle $\mathscr{S}$ by the tautological rank $n-k$ quotient bundle $\mathscr{Q}$. It is a standard fact that the tangent space of the Grassmannian is canonically identified with $\op{Hom}(\mathscr{S},\mathscr{Q})$, cf. \cite[Section 4.1]{feher:matszangosz} or \cite[Theorem 3.5]{3264}. 

To describe the tangent bundles of the flag varieties $\op{Fl}(\underline{d},\underline{e})$, we can use the fact that these flag varieties have a fibration by Grassmannians, cf. \cite[Lemma 1.11]{balmer:calmes}. For two tuples $(d_1,\dots,d_p)$ and $(e_1,\dots,e_p)$ (as the ones arising from even Young diagrams, cf. Figure~\ref{fig:bc}), $\op{Fl}(\underline{d},\underline{e})$ can be identified as Grassmannian of $d_p-d_{p-1}$-dimensional subspaces in the tautological quotient bundle $\mathscr{O}_{\op{Fl}}^{\oplus d_p+e_p}/\mathscr{P}_{d_{p-1}}$ of rank $d_p-d_{p-1}+e_p$ on $\op{Fl}((d_1,\dots,d_{p-1}),(e_1,\dots,e_{p-1}))$. Hence, we have a fibration
\[
\op{Gr}(d_p-d_{p-1},d_p-d_{p-1}+e_p)\to \op{Fl}(\underline{d},\underline{e})\to \op{Fl}((d_1,\dots,d_{p-1}),(e_1,\dots,e_{p-1}))
\]
With these identifications, the flag of subbundles on $\op{Fl}(\underline{d},\underline{e})$ is given as follows: the first $k-1$ subbundles are the pullbacks of the subbundles $\mathscr{P}_i$ for $1\leq i\leq k-1$, and the last two subquotients $\mathscr{P}_{d_p}/\mathscr{P}_{d_{p-1}}$ and $\mathscr{O}_{\op{Fl}(\underline{d},\underline{e})}^{\oplus d_p+e_p}/\mathscr{P}_{d_p}$  identified with the tautological sub- and quotient bundle of the relative Grassmannian $\op{Fl}(\underline{d},\underline{e})\to\op{Fl}((d_1,\dots,d_{p-1}),(e_1,\dots,e_{p-1}))$, respectively. In particular, the relative tangent bundle of $\op{Fl}(\underline{d},\underline{e})/\op{Fl}((d_1,\dots,d_{p-1}),(e_1,\dots,e_{p-1}))$ is canonically identified with $\op{Hom}(\mathscr{P}_{d_p}/\mathscr{P}_{d_{p-1}}, \mathscr{O}_{\op{Fl}(\underline{d},\underline{e})}^{\oplus d_p+e_p}/\mathscr{P}_{d_p})$. The tangent bundle for $\op{Fl}(\underline{d},\underline{e})$ is then an extension of the relative tangent bundle by the pullback of the tangent bundle of $\op{Fl}((d_1,\dots,d_{p-1}),(e_1,\dots,e_{p-1}))$. Inductively, we obtain the following description: 

\begin{proposition}
\label{prop:tangentflag}
The tangent bundle of $\op{Fl}(\underline{d},\underline{e})$ is canonically identified as an iterated extension of pullbacks of the bundles 
\[
\op{Hom}_{\op{Fl}((d_1,\dots,d_j),(e_1,\dots,e_j))}(\mathscr{P}_{d_j}/\mathscr{P}_{d_{j-1}}, \mathscr{O}^{\oplus d_j+e_j}/\mathscr{P}_{d_j})
\]
with $j=1,\dots,p$. In particular, the rank of the tangent bundle is $\sum_{i=1}^p(d_i-d_{i-1})e_i$. 
\end{proposition}


\subsection{Relative orientations of relative normal bundles}

Now we discuss how to provide canonical relative orientations on the relative canonical bundles for the flag resolutions of Schubert varieties. 

\begin{proposition}
\label{prop:uniqueorient}
Fix $1\leq k<n$ and consider a Young diagram $\Lambda$ in $k\times(n-k)$-frame. If $\Lambda$ satisfies the conditions of Proposition~\ref{prop:compareBC} for the line bundle $\mathscr{L}$ but is not even, then the class $(p_\Lambda)_\ast(1)\in \op{H}^{|\Lambda|}(\op{Gr}(k,n),\mathbf{I}^{|\Lambda|}(\mathscr{L}))$ is independent of the choice of relative orientation of $\omega_{p_\Lambda}\otimes p_\Lambda^\ast(\mathscr{L})$. 
\end{proposition}

\begin{proof}
If $\Lambda$ satisfies one of the conditions of Proposition~\ref{prop:compareBC} for the line bundle $\mathscr{L}$, then the corresponding class  $\sigma_\Lambda\in\op{Ch}^{|\Lambda|}(\op{Gr}(k,n))$ lifts to $\op{H}^{|\Lambda|}(\op{Gr}(k,n),\mathbf{I}^{|\Lambda|}(\mathscr{L}))$. If $\Lambda$ is not even, then the corresponding lift is necessarily $\op{I}(F)$-torsion. In particular, by Proposition~\ref{prop:incoh}, it is uniquely determined by requiring that its reduction is the class $\sigma_\Lambda$. Since the reduction morphism $\op{H}^\bullet(X,\mathbf{I}^\bullet(\mathscr{L}))\to\op{Ch}^\bullet(X)$ is compatible with pushforwards, the class $(p_\Lambda)_\ast(1)$ has reduction $\sigma_\Lambda$ (equal to the pushforward of $1$ along $p_\Lambda$ on the level of mod 2 Chow rings), by definition. The claim now follows from the uniqueness statement above.
\end{proof}

In particular, for the torsion Schubert classes in $\op{H}^\bullet(\op{Gr}(k,n),\mathbf{I}^\bullet(\mathscr{L}))$, we don't need to worry about choices of orientations; we can always consider $(p_\Lambda)_\ast(1)$ to be a geometric representative of the corresponding class, independent of the orientation for the relative canonical bundle. Therefore, in the subsequent discussion, we can now focus on even Young diagrams. These correspond to non-torsion classes, and the class $(p_\Lambda)_\ast(1)$ really depends on the choice of orientations. 

Before we can discuss the canonical orientations for relative canonical bundles, we need a couple of preliminary statements on determinants of vector bundles. First, the following provides a canonical identification of determinants for tensor products, cf. \cite[Lemma 13]{okonek:teleman}. 

\begin{lemma}
\label{lem:otdet}
Let $V,W$ be two finite-dimensional $F$-vector spaces. Then there is a canonical identification:
\[
\det(V\otimes_FW)\cong \det(V)^{\otimes \dim_FW}\otimes_F\det(W)^{\otimes\dim_FV}.
\]
\end{lemma}

As a direct consequence, we get a variation of \cite[Lemma 4.4]{feher:matszangosz}: if $\dim_FV$ and $\dim_FW$ are both even, the tensor product $V\otimes_F W$ has a canonical orientation given by $\mathscr{N}\cong \det(V)^{\otimes (\dim_FW/2)}\otimes_F\det(W)^{\otimes(\dim_FV/2)}$ and the above identification; this can also be applied to $\op{Hom}_F(V,W)\cong V^\vee\otimes_F W$. 

\begin{proposition}
\label{prop:orientflag}
For integers $1\leq d_1<d_2<\cdots<d_p=k<n$, consider the flag variety $\op{Fl}(d_1,\dots,d_p;n)$. If  
\[
d_1\equiv d_2-d_{1}\equiv d_3-d_2\equiv\cdots\equiv d_p-d_{p-1}\equiv n-k\mod 2
\]
then the flag variety $\op{Fl}(d_1,\dots,d_p;n)$ has a canonical orientation.
\end{proposition}

\begin{proof}
The tangent bundle of $\op{Fl}(d_1,\dots,d_p;n)$ has a filtration whose subquotients are Hom-bundles $\op{Hom}(\mathscr{S}_i,\mathscr{S}_j)$ where $i<j$ and $\mathscr{S}_i$ is the tautological subquotient bundle of rank $d_i-d_{i-1}$. Using this together with  Lemma~\ref{lem:otdet}, we obtain canonical isomorphisms
\begin{eqnarray*}
\omega_{\op{Fl}(d_1,\dots,d_p;n)}^\vee &\cong& \bigotimes_{i<j} \det\op{Hom}(\mathscr{S}_i,\mathscr{S}_j)\\
&\cong& \bigotimes_{i<j}\left( \det\mathscr{S}_i^{\otimes(-{\op{rk}}\mathscr{S}_j)}\otimes\det\mathscr{S}_j^{\otimes({\op{rk}}\mathscr{S}_i)}\right).
\end{eqnarray*}
This can be rearranged to a tensor product of powers $\det\mathscr{S}_i^{\otimes m_i}$ of the subquotient bundles, with the multiplicity given by
\[
m_i=\sum_{j<i}\op{rk}\mathscr{S}_j-\sum_{j>i}\op{rk}\mathscr{S}_j=\sum_{j<i}(d_j-d_{j-1})-\sum_{j>i}(d_j-d_{j-1})=d_{i-1}+d_i-n.
\]
Therefore, we have a canonical identification $\omega_{\op{Fl}}^\vee\cong\det\mathscr{S}_i^{\otimes(d_{i-1}+d_i-n)}$. 

By the above computation, if all the $d_j-d_{j-1}$ are even, then so are all the $m_i$. In this case, the canonical bundle $\omega_{\op{Fl}}$ has a canonical orientation. If all $d_j-d_{j-1}$ are odd, there are two further cases: if $p$ is odd, then we always add an even number of odd summands, therefore $m_i$ is even and the bundle $\omega_{\op{Fl}}$ has a canonical orientation. Finally, if $p$ is even, then all $m_i$ are odd. Since the subquotient bundles arise from a filtration of the trivial bundle, we have a canonical identification $\bigotimes_i\det\mathscr{S}_i\cong\mathscr{O}_{\op{Fl}}$. In particular, with this identification, $\omega_{\op{Fl}}$ is again a tensor product of even tensor powers of $\mathscr{S}_i$ and hence canonically oriented.
\end{proof}

\begin{remark}
In the special case $\op{Gr}(k,n)$, we have an extension 
\[
0\to\mathscr{S}\to\mathscr{O}_{\op{Gr}}^{\oplus n}\to\mathscr{Q}\to 0
\]
where $\mathscr{S}$ is the tautological rank $k$ subbundle of the trivial bundle and $\mathscr{Q}$ is the tautological rank $n-k$ quotient bundle. The tangent bundle is isomorphic to $\op{Hom}(\mathscr{S},\mathscr{Q})$ and consequently  Lemma~\ref{lem:otdet} provides a canonical identification
\[
\omega_{\op{Gr}(k,n)}^\vee\cong (\det\mathscr{S}^\vee)^{\otimes(n-k)}\otimes(\det\mathscr{Q})^{\otimes k}.
\]
From the above extension we also have a canonical isomorphism $\det\mathscr{S}\otimes\det\mathscr{Q}\cong\mathscr{O}_{\op{Gr}}$. We can use this canonical isomorphism $\det\mathscr{S}^\vee\cong\det\mathscr{Q}$ to write $\omega_{\op{Gr}(k,n)}^\vee\cong \det\mathscr{Q}^{\otimes n}$. In particular, if $n$ is even, the Grassmannian $\op{Gr}(k,n)$ has a canonical orientation. 
\end{remark}

We now can describe the canonical orientations for the relative canonical bundles of the flag resolutions. 

\begin{proposition}
\label{prop:orientschubert}
For an even Young diagram $\Lambda$, consider the associated tuples $\underline{d}=(d_1,\dots,d_p=k)$ and $\underline{e}=(e_1,\dots,e_p)$ with $0<d_1<\cdots<d_p<n$ and $e_1+d_1\leq \cdots\leq e_p+d_p$, cf. Figure~\ref{fig:bc} and Definition~\ref{def:evenyoung}. 
\begin{enumerate}
\item If the Young diagram $\Lambda$ has twist $0$, then the relative canonical bundle $\omega_{p_\Lambda}$ for the resolution $p_\Lambda:\op{Fl}(\underline{d},\underline{e})\to \op{Gr}(k,V)$ is canonically orientable.
\item If the Young diagram $\Lambda$ has twist $1$, then the twisted relative canonical bundle $\omega_{p_\Lambda}\otimes p_\Lambda^\ast\det\mathscr{E}_k^\vee$ is canonically orientable. 
\end{enumerate}
\end{proposition}

\begin{proof}
We compute the canonical bundle of $\op{Fl}(\underline{d},\underline{e})$, this is essentially the same argument as in \cite[Proposition 1.13]{balmer:calmes}. By Proposition~\ref{prop:tangentflag}, the tangent bundle is canonically identified as iterated extension of Hom-bundles. This translates into a canonical identification of $\omega_{\op{Fl}(\underline{d},\underline{e})}$ as tensor product of exterior powers of Hom-bundles, and the latter can be canonically identified via Lemma~\ref{lem:otdet}. Inductively, we obtain a canonical identification of the canonical bundle:  
\[
\omega_{\op{Fl}(\underline{d},\underline{e})}\cong \bigotimes_{i=1}^{p-1}\det\mathscr{P}_{d_i}^{\otimes(d_i-d_{i-1}+e_i-e_{i+1})}\otimes \det \mathscr{P}_{d_p}^{\otimes(d_p-d_{p-1}+e_p)}.
\]
The relative canonical bundle $\omega_{p_\Lambda}$ is obtained from this by tensoring with the dual of the pullback of the canonical bundle of the Grassmannian, which  by Proposition~\ref{prop:orientflag} is $\det\mathscr{P}_{d_p}^{d_p+n}$, cf. also \cite[Corollary 1.14]{balmer:calmes}. 

(1) As in the proof of Proposition~\ref{prop:orientflag}, if the twist of $\Lambda$ is $0$, then all the determinants of subquotient bundles in the relative canonical bundle identified above appear with powers which all have the same parity. We get a canonical orientation if all the powers are even. If all powers are odd, we can again use the identification $\otimes_i\det\mathscr{S}_i\cong\mathscr{O}_{\op{Fl}}$ to reduce to the case where all powers are even, hence we also have a canonical orientation. 

(2) If the twist of $\Lambda$ is $1$, we can remove even powers of line bundles as in (1). What remains is a single copy of $\det \mathscr{P}_{d_p}$ which is the pullback of a representative of the generator of $\op{Pic}(\op{Gr}(k,n))/2$. In particular, we get a canonical orientation on $\omega_{p_\Lambda}\otimes p_\Lambda^\ast\mathscr{L}$. 
\end{proof}

\begin{remark}
Note that the orientability of $\omega_{p_\Lambda}$ resp. $\omega_{p_\Lambda}\otimes p_\Lambda^\ast\mathscr{L}$ is only possible if the class of the relative canonical bundle $\omega_{p_\Lambda}$ belongs to the image of the restriction morphism $p_\Lambda^\ast:\op{Pic}(\op{Gr}(k,V))/2\to \op{Pic}(\mathcal{F}l(\Lambda))/2$. By \cite{balmer:calmes}, this is the case  if and only if the following conditions are satisfied: 
\begin{enumerate}
\item $d_i-d_{i-1}+e_{i+1}-e_i$ is even for every $i=2,\dots,k-1$ (for $k\geq 3$)
\item when $0<e_1<e$ and $k\geq 2$, $d_1+e_2-e_1$ is also even
\end{enumerate}
In turn, these conditions are satisfied if and only if the corresponding Schubert class lifts to the Chow--Witt ring, i.e., one of the conditions in Proposition~\ref{prop:compareBC} must be satisfied. 
\end{remark}

\begin{remark}
Schubert varieties are generally singular, but they are normal and Cohen--Macaulay. In particular, the canonical bundle $\omega_{X_w}$ of a Schubert variety $X_w$ can be obtained as push-forward of the canonical bundle from the smooth locus. This implies that restricting a choice of orientation for the canonical bundle of a singular Schubert variety to the smooth locus is the same thing as an orientation of the canonical bundle of the smooth locus. 
\end{remark}

Now we can discuss our natural choices for fundamental classes of Schubert varieties in $\mathbf{I}^\bullet$-cohomology. We consider an appropriate resolution, a (canonical) choice of orientation, and then the fundamental class of the Schubert variety in the $\mathbf{I}^\bullet$-cohomology ring is given by pushforward of 1 along the resolution. 

\begin{theorem}
\label{thm:lift}
Fix $1\leq k<n$, let $\Lambda$ be an even Young diagram in $k\times(n-k)$-frame and denote by $p_\Lambda:\op{Fl}(\Lambda)\to\op{Gr}(k,n)$ the resolution of singularities for the Schubert variety $\Sigma_\Lambda$. Then $(p_\Lambda)_\ast(1)\in \op{H}^{|\Lambda|}(\op{Gr}(k,n),\mathbf{I}^{|\Lambda|}(\mathscr{L}))$ is a lift of the Schubert class $\sigma_\Lambda\in \op{Ch}^{|\Lambda|}(\op{Gr}(k,n))$. A canonical choice for such lift is given by the canonical orientations of Proposition~\ref{prop:orientschubert}. 
\end{theorem}

\begin{proof}
The morphism $p_\Lambda:\op{Fl}(\Lambda)\to \op{Gr}(k,n)$ is a proper morphism between smooth $F$-schemes, hence we have  a push-forward map, cf. Section~\ref{sec:pushforward}:
\[
\op{H}^0(\op{Fl}(\Lambda),\mathbf{I}^0(\omega_{p_\Lambda}\otimes p^\ast_\Lambda \mathscr{L})) \xrightarrow{(p_\Lambda)_\ast} \op{H}^d(\op{Gr}(k,n),\mathbf{I}^d(\mathscr{L})).
\]
Here $d=\dim \op{Gr}(k,n)-\op{Fl}(\Lambda)=|\Lambda|$, cf. \cite[Propositions 1.3, 1.13]{balmer:calmes}. For an even Young diagram, the appropriate choice of twisting line bundle $\mathscr{L}$ is a generator of $\op{Pic}(\op{Gr}(k,n))/2$ if one of the two possible boundary components of $\Lambda$ is odd; the general twist is determined by the half-perimeter \cite[Definition 4.3]{balmer:calmes}. 

Since the push-forward is compatible with the reduction morphism $\op{H}^j(X,\mathbf{I}^j)\to\op{Ch}^j(X)$, the image of $(p_\Lambda)_\ast(1)$ in $\op{Ch}^{|\Lambda|}(\op{Gr}(k,n))$ is the class of the subscheme given by the image of $p_\Lambda$, which is exactly the Schubert class $\sigma_\Lambda$. 

It remains to note that the push-forward morphism depends on the choice of a relative orientation, i.e., a line bundle $\mathscr{M}$ and an isomorphism $\mathscr{M}^{\otimes 2}\cong \omega_{p_\Lambda}\otimes p_\Lambda^\ast\mathscr{L}$. But we can use the canonical orientations from Proposition~\ref{prop:orientschubert} to get canonical lifts of Schubert classes.
\end{proof}

\begin{remark}
\label{rem:generator}
The canonical lifts in the above theorem are in fact generators of the corresponding $\op{W}(F)$-summand in $\mathbf{I}^{|\Lambda|}(\mathscr{L})$-cohomology. In \cite{balmer:calmes},  inductive arguments involving localization sequences and explicit computations of images of generators under boundary morphisms were used to show that the classes $(p_\Lambda)_\ast(1)$ are generators of $\op{W}^{\op{tot}}(\op{Gr}(k,n))$. The same arguments can now also be used to show that the classes $(p_\Lambda)_\ast(1)$ are generators for $\bigoplus_{j,\mathscr{L}}\op{H}^j(\op{Gr}(k,n),\mathbf{I}^j(\mathscr{L}))$, using the compatibility of the comparison morphism $\alpha:\op{W}^j(X,\mathscr{L})\to \op{H}^j(X,\mathbf{I}^j(\mathscr{L}))$ with pullbacks, pushforwards and localization sequences (possibly up to sign). 
\end{remark}

\section{Geometric interpretation of oriented intersection multiplicities}
\label{sec:intersection}

In this section, we can now describe the oriented intersection multiplicity of a solution subspace in an intersection of Schubert varieties. Essentially, we can use the classical homological transversality results to make the Schubert varieties transverse and then the oriented intersection multiplicity is given by the length of the reduced subscheme plus additional relative orientation information. 

\subsection{Transversality}

Before we can provide explicit geometric identifications of the oriented intersection multiplicities, we need to discuss how to make the Schubert varieties in the Schubert problem transversal. Classically, one uses (homological) transversality results to make the Schubert cycles transverse, and we will discuss how to take care of the additional orientations that are floating around. 

We first recall the very classical Kleiman--Bertini transversality theorem. cf. \cite{kleiman}:

\begin{theorem}
\label{thm:kleiman}
Let $F$ be a field of characteristic $0$, let $G$ be an algebraic group and let $G\looparrowright X$ be a smooth variety over $F$ with $G$-action, with integral closed subschemes $Y,Z\subseteq X$. Then there exists a dense open subset $U$ of $G$ such that for each point $g\in U(F)$, the intersection $gY\cap Z$ is proper, each component appears with multiplicity one, i.e., there exists a dense open in $gY\cap Z$ which is regular and has pure dimension $\dim Y+\dim Z-\dim X$. 
\end{theorem}

In particular, at a general point of the intersection $Y\cap Z$, the tangent space to $X$ will decompose as direct sum of the normal spaces to $Y$ and $Z$ and the tangent space to $Y\cap Z$. Moreover, if we have $Y, Z\subseteq X$ of complementary dimensions, the intersection $Y\cap Z$ is a zero-dimensional reduced scheme which is contained in the smooth loci of $Y$ and $Z$, respectively. In positive characteristic, this strong version of transversality fails, but there is some homological transversality results that can be used (in conjunction with Serre's Tor formula) to compute intersection products in Chow rings. 

\begin{definition}
Let $F$ be a field and let $X$ be a variety over $F$. Two quasi-coherent sheaves $\mathscr{E}$ and $\mathscr{F}$ on $X$ are called \emph{homologically transverse} if the higher Tor-sheaves vanish:
\[
\mathscr{T}\!\!or^X_j(\mathscr{E},\mathscr{F})=0 \textrm{ for all } j\geq 1.
\]
If $\mathscr{E}=\mathscr{O}_Y$ for a closed subvariety $Y\subseteq X$, we say that $\mathscr{F}$ is homologically transverse to $Y$. 
\end{definition}

We consider a scheme $F$ defined over a field $X$, fix an algebraic closure $\overline{F}$ of $F$ and denote $\overline{X}=X\times\op{Spec} \overline{F}$. For a coherent sheaf $\mathscr{F}$ on $X$, we also denote by $\overline{\mathscr{F}}$ the pullback of $\mathscr{F}$ to $\overline{X}$. The following criterion for homological transverse being generic is proved in \cite[Theorem 1.2]{sierra}:

\begin{theorem}
\label{thm:sierra}
Let $F$ be a field with algebraic closure $\overline{F}$. Let $X$ be an $F$-scheme with a left action of a smooth algebraic group $G$, and let $\mathscr{F}$ be a coherent sheaf on $X$. If for all closed points $x\in \overline{X}$ the sheaf $\overline{\mathscr{F}}$ is homologically transverse to the closure of the $G(\overline{F})$-orbit of $x$, then for all coherent sheaves $\mathscr{E}$ on $X$ there is a Zariski open and dense subset $U\subseteq G$ such that for all $g\in U(F)$ the sheaf $g\mathscr{F}$ is homologically transverse to $\mathscr{E}$. 
\end{theorem}

\begin{remark}
We can apply this to $X=\op{Gr}(k,n)$ as variety with $\op{GL}_n$-action over $F$ and $\mathscr{E}=\mathscr{O}_Y$, $\mathscr{F}=\mathscr{O}_Z$ for two Schubert varieties $Y,Z\subseteq X$. Then the conditions of Sierra's theorem are satisfied, and the two Schubert varieties $Y$ and $Z$ can be made homologically transversal by taking a generic translate $gY$ for $g\in \op{GL}_n(F)$. Note that this requires $F$ to be infinite; for a finite field $F=\mathbb{F}_q$, the dense open subset $U\subseteq G$ need not contain $\mathbb{F}_q$-rational points. Note also that this result can be used to make $Y$ transversal to finitely many other Schubert varieties: the intersection of the relevant open dense subsets $U_1,\dots,U_l$ will be open and dense.
\end{remark}

\begin{remark}
We could also apply other criteria to deduce homological transversality for intersections of Schubert varieties in Grassmannians. For example,  \cite[Lemma 1]{brion} states that any proper intersection of equi-dimensional Cohen--Macaulay subschemes $Y$ and $Z$ in a smooth variety $X$ is homologically transverse. 
\end{remark}

Changing a Schubert variety $\Sigma_\Lambda(\mathscr{F})$ to a $\op{GL}_n$-translate $g\Sigma_\Lambda(\mathscr{F})$ doesn't change the corresponding class in the Chow ring. The same is true for the Chow--Witt ring: 

\begin{proposition}
\label{prop:ratequiv}
Let $F$ be a field of characteristic $\neq 2$. Let $Y=\Sigma_\Lambda(\mathscr{F})\subseteq \op{Gr}(k,n)$ be a Schubert variety corresponding to a Young diagram satisfying one of the conditions of Proposition~\ref{prop:compareBC} for the line bundle $\mathscr{L}$. For any $g\in \op{SL}_n(F)$, denote by $p_Y:\op{Fl}(\Lambda,\mathscr{F})\to\op{Gr}(k,n)$ and $p_{gY}:\op{Fl}(\Lambda,g\mathscr{F})\to\op{Gr}(k,n)$ the corresponding resolutions of singularities of $Y$ and $gY$, respectively. Then we have 
\[
(p_Y)_\ast(1)=(p_{gY})_\ast(1)\in \op{H}^{|\Lambda|}(\op{Gr}(k,n),\mathbf{I}^{|\Lambda|}(\mathscr{L}))
\]
\end{proposition}

\begin{proof}
Since $g\in \op{SL}_n(F)$, it has an elementary factorization. Therefore, we have a chain of $\mathbb{A}^1$-families of translates of $Y$ connecting $Y$ and $gY$. It suffices to prove the claim for a single $\mathbb{A}^1$-family $Y\times\mathbb{A}^1\to \op{Gr}(k,n)$. This $\mathbb{A}^1$-family gives rise to a family of resolutions $q:\op{Fl}(\Lambda)\times\mathbb{A}^1\to\op{Gr}(k,n)$, and we want to compare the pushforwards of 1 from the two fibers, for $Y$ and $gY$. The form 1 obviously extends to the whole family $q:\op{Fl}(\Lambda)\times \mathbb{A}^1\to\op{Gr}(k,n)$. Since $\op{Fl}(\Lambda)$ is smooth, the  bundle $\omega_q\otimes q^\ast\mathscr{L}$ is isomorphic to the pullbacks of $\omega_{p_Y}\otimes p_Y^\ast\mathscr{L}$ and $\omega_{p_{gY}}\otimes p_{gY}^\ast\mathscr{L}$ along the respective $\mathbb{A}^1$-projections. Moreover, the corresponding induced orientations on $\omega_q\otimes q^\ast\mathscr{L}$ agree. In particular, we have a pushforward $(p_q)_\ast(1)$ which agrees with the classes $(p_Y)_\ast(1)$ and $(p_{gY})_\ast(1)$, thus proving the claim.
\end{proof}

\begin{remark}
More generally, one would expect for a smooth variety $X$ with an action of an $\mathbb{A}^1$-connected group $G$, that translation of $\mathbf{I}^\bullet$-cohomology cycles by $g\in G(F)$ doesn't change the cohomology class. In the corresponding Borel--Moore homology picture, cf. \cite{deglise:fasel}, the above argument shows that this is true because of $\mathbb{A}^1$-invariance for singular varieties. However, in the $\mathbf{I}^\bullet$-cohomology picture, this is not so clear, but the above special case is enough for our purposes.
\end{remark}

\begin{remark}
Note also that the above result can be applied to arbitrary Schubert varieties. The homological transversality result of Sierra, applied to the $\op{SL}_n$-action on $\op{Gr}(k,n)$, implies that we can make Schubert varieties homologically transversal by elements from $\op{SL}_n$. In particular, the canonical fundamental classes of Schubert varieties in $\mathbf{I}^\bullet$-cohomology discussed in Section~\ref{sec:geometricreps} are independent of the choice of flag defining the Schubert variety. Hence, to compute oriented intersection products of such fundamental classes, we can always assume that the corresponding Schubert varieties are homologically transversal, provided the base field is infinite.
\end{remark}

Finally, we want to mention Vakil's generic smoothness theorem \cite[Theorem 1.5]{vakil} which provides all the relevant consequences of the Kleiman--Bertini theorem in positive characteristic, for zero-dimensional intersections.\footnote{A comment after the statement of Theorem 1.5 in loc.cit. suggests that we might hope for a Kleiman--Bertini type generic smoothness theorem in arbitrary characteristic if we assume generic transversality with respect to Schubert varieties. Such a result would solve all the transversality issues left in the present discussion, but this seems open at the moment.} As a particular consequence of Vakil's generic smoothness result, for a given Schubert problem over a finite field, there is a positive density of flags (possibly defined over finite field extensions $E/\mathbb{F}_q$) where the intersection subscheme is just a union of reduced $E$-rational points. A fortiori, there will be a positive density of flags where the corresponding intersection of Schubert varieties is a reduced subscheme. In particular, while it's not quite clear if Schubert varieties can be made transversal over finite fields, this will always be possible after a finite field extension. 

\subsection{The type of subspaces} 

Finally, we need to interpret the local contribution of each solution subspace to the degree of the oriented zero-cycle computed previously. This will be done via comparison of the canonical orientations of the Schubert varieties involved, generalizing the notion of type of subspaces in \cite{feher:matszangosz} and similar to the type of lines in \cite{kass:wickelgren}. 

We first discuss the geometric interpretation of intersection multiplicities in the case where the relevant Schubert varieties are smooth (or intersect in points where the resolutions of singularities from Section~\ref{sec:geometricreps} are isomorphisms) and there are no twists involved. Let $F$ be a field of characteristic $\neq 2$, let $V$ be an $n$-dimensional $F$-vector space, and let $\mathfrak{S}_1,\dots,\mathfrak{S}_l$ be smooth Schubert varieties in $\op{Gr}(k,V)$ with $\sum_i\dim\mathfrak{S}_i=k(n-k)$, intersecting properly in a closed point $[W]\in \op{Gr}(k,n)$. Then $[W]$ is defined over a finite extension field $E/F$, and corresponds to a $k$-dimensional $E$-vector space $W\subseteq V\otimes_FE$. Since the Schubert varieties are smooth, we have normal spaces $\op{N}_{[W]}\mathfrak{S}_i$ of dimension $k(n-k)-\dim\mathfrak{S}_i$, and the normal space can be explicitly identified as $\op{Hom}(W\cap U,V/\langle U,W\rangle)$ with $U\leq V$ a subspace of suitable dimension.\footnote{The Schubert variety being smooth is equivalent to the Young diagram being a rectangle. In particular, the Schubert variety is determined by intersection conditions w.r.t. a flag $0\leq U\leq V$.} Moreover, we have a canonical quotient map $\op{T}_{[W]}\op{Gr}(k,n)\to\op{N}_{[W]}\mathfrak{S}_i$ which can be explicitly identified as 
\begin{eqnarray*}
\op{T}_{[W]}\op{Gr}(k,n)\cong\op{Hom}(W,V/W)&\to& \op{Hom}(W\cap U,V/\langle U,W\rangle)\cong \op{N}_{[W]}\mathfrak{S}_i\\
\psi:W\to V/W&\mapsto & \left(W\cap U\hookrightarrow W\xrightarrow{\psi}V/W\twoheadrightarrow V/\langle U,W\rangle\right)
\end{eqnarray*}
cf. e.g. the discussion in \cite[Section 4.1]{feher:matszangosz}. Since the Schubert varieties intersect properly at $[W]$, the natural map
\[
\xi_W:\op{T}_{[W]}\op{Gr}(k,n)\to \bigoplus_{i=1}^l\op{N}_{[W]}\mathfrak{S}_i
\]
obtained from the above projections to the normal spaces is an isomorphism, called the \emph{splitting map}. The splitting map induces an isomorphism of determinants, which we can rewrite as a canonical orientation  
\[
\mathscr{O}_{\op{Gr}(k,n),[W]}\cong \omega_{\op{Gr}(k,n)}|_{[W]}\otimes \bigotimes_{i=1}^l\omega_{\mathfrak{S}_i/\op{Gr}(k,n)}|_{[W]}.
\]
On the other hand, the canonical orientations of the tangent bundle of $\op{Gr}(k,n)$ and the normal bundles of $\mathfrak{S}_i$, cf. Propositions~\ref{prop:orientflag} and \ref{prop:orientschubert}, provide another such orientation of the right-hand side tensor product. Note that the right-hand side is a one-dimensional $E$-vector space, and each of the two orientations is a generator. Since an orientation is only well-defined up to multiplication by squares, the two orientations differ by a unit $u\in E^\times$ whose square residue in $E^\times/(E^\times)^2$ only depends on the orientations and not on the specific generators representing them. It makes sense to denote this unit $u$ by $\det\xi_W$: if we choose bases in $\op{T}_{[W]}\op{Gr}(k,n)$ and $\bigoplus_{i=1}^l\op{N}_{[W]}\mathfrak{S}_i$ which induce the orientations, then the unit $u$ is given as the square class of the determinant of the matrix representing the splitting map in the chosen bases. 

We can now define the type of the solution subspace $W$ which will turn out to coincide with the oriented intersection multiplicity of point $[W]$ in the intersection product $[\mathfrak{S}_1]\cdots[\mathfrak{S}_l]$ in the Chow--Witt ring. 

\begin{definition}
\label{def:type}
In the above notation, if $W$ is a solution subspace in the intersection of the Schubert varieties $\mathfrak{S}_1,\dots,\mathfrak{S}_l$, defined over $E/F$, then the \emph{type} of $W$ is defined to be 
\[
\op{tr}_{E/F}\langle \det\xi_W\rangle\in\op{GW}(F),
\]
i.e., it is the Scharlau trace of the one-dimensional quadratic form given by the determinant of the splitting map. 
\end{definition}

\begin{remark}
The same procedure works if we consider non-orientable Schubert varieties. In this case, the normal bundles to some of the $\mathfrak{S}_i$ may not be orientable. However, the twisting by $\mathscr{L}_i\cong\det\mathscr{E}_k^\vee|_{\mathfrak{S}_i}$ makes the bundles orientable. Since we want to consider intersection multiplicities, we have to look at Schubert problems in Grassmannians with $n$ even, and then we must necessarily have an even number of non-orientable Schubert varieties involved in the intersection problem. The type of a subspaces in the intersection is then as in Definition~\ref{def:type} above: the type is given by the square class of units which compares the canonical orientations on $\omega_{\op{Gr}(k,n)}|_{[W]}\otimes\bigotimes_{i=1}^l\omega_{\mathfrak{S}_i/\op{Gr}(k,n)}|_{[W]}(\mathscr{L}_i)$, one given by the natural morphisms and the other one given by the canonical orientations.
\end{remark}

We also need to discuss a notion of types of solution subspaces for the case of intersections of singular Schubert varieties. While this definition is really applicable to all intersection problems, it will usually be difficult to actually evaluate in practice. 

Let $F$ be a field of characteristic $\neq 2$, let $V$ be an $n$-dimensional $F$-vector space, and let $\mathfrak{S}_1,\dots,\mathfrak{S}_l$ be Schubert varieties in $\op{Gr}(k,V)$ with $\sum_i\dim\mathfrak{S}_i=k(n-k)$, intersecting properly in a closed point $[W]\in\op{Gr}(k,n)$. Consider the resolutions of singularities $p_i:\op{Fl}(\Lambda_i)\to\op{Gr}(k,V)$ of Section~\ref{sec:geometricreps}, where $\Lambda_i$ is the Young diagram for the Schubert variety $\mathfrak{S}_i$. For every $i$, we have a line bundle $\mathscr{L}_i$ (either equal to $\mathscr{O}_{\op{Gr}(k,n)}$ or $\det\mathscr{E}_k^\vee$) such that $\omega_{p_i}\otimes p_i^\ast\mathscr{L}_i$ has a canonical orientation, cf. Proposition~\ref{prop:orientschubert}, i.e., there is a line bundle $\mathscr{M}_i$ such that $\mathscr{M}_i\otimes\mathscr{M}_i\xrightarrow{\cong} \omega_{p_i}\otimes p_i^\ast\mathscr{L}_i$. We can now consider the fiber $p_i^{-1}([W])$ of the resolution $p_i:\op{Fl}(\Lambda_i)\to\op{Gr}(k,V)$ over the point $[W]$. The push-forward of the orientation form $\mathscr{M}_i\otimes\mathscr{M}_i\xrightarrow{\cong} \omega_{p_i}\otimes p_i^\ast\mathscr{L}_i$ is a form $\op{R}(p_i)_\ast(\mathscr{M}_i)\otimes \op{R}(p_i)_\ast(\mathscr{M}_i)\to\mathscr{L}_i$, cf. \cite{calmes:hornbostel} for the exact definition of the push-forward of forms. Restricted to the point $[W]$, we get an intersection form on coherent cohomology
\[
\cup:\op{H}^r(p_i^{-1}([W]),\mathscr{M}_i)\otimes \op{H}^s(p_i^{-1}([W]),\mathscr{M}_i)\to \mathscr{O}_{[W]}
\]

\begin{definition}
\label{def:typesing}
In the above notation, if $W$ is a solution subspace in the intersection of the Schubert varieties $\mathfrak{S}_1,\dots,\mathfrak{S}_l$, then the \emph{type} of $W$ is the image of the product in $\op{GW}_{[W]}(\op{Gr}(k,V))$ of the intersection forms
\[
\left(\bigoplus_{j=0}^{\dim p_i^{-1}([W])}\op{H}^j(p_i^{-1}([W]),\mathscr{M}_i)[-j],\cup\right)
\]
for all $i=1,\dots,l$ under the identification $\op{GW}_{[W]}(\op{Gr}(k,V))\cong \op{GW}(F)$ given by the canonical orientation of the tangent space $\op{T}_{[W]}\op{Gr}(k,V)$.
\end{definition}

\begin{remark}
This bears some similarity with the Levine--Serre formula for the motivic Euler classes of smooth varieties in terms of trace forms on $\bigoplus\op{H}^i(X,\Omega^j)$, cf. \cite{levine:schur}. 
\end{remark}

\begin{remark}
Note that Definition~\ref{def:typesing} reduces to the simpler Definition~\ref{def:type} when the resolutions $p:\op{Fl}(\Lambda)\to\op{Gr}(k,V)$ are isomorphisms at the intersection point. In this case, both forms are simply given by the canonical orientations of the normal bundle of the Schubert variety at the (smooth) intersection point.
\end{remark}

\subsection{Oriented intersection multiplicities}

\begin{remark}
In the following, we want to identify the multiplicity of a closed point $[W]\in\op{Gr}(k,V)$ in the intersection product (in the Chow--Witt ring resp. the $\mathbf{I}^\bullet$-cohomology) of oriented Schubert classes $\mathfrak{S}_1,\dots,\mathfrak{S}_l$. This will be done using the oriented Tor formula of Section~\ref{sec:torformula} and computations of $\star$-products in Witt rings. Of course, since $\op{W}^{\op{tot}}(\op{Gr}(k,V))$ is a free $\op{W}(F)$-module generated by the even Young diagrams, this procedure doesn't apply to intersections involving torsion classes. However, the product with torsion classes is a torsion class and can be identified with the corresponding intersection product in the mod 2 Chow ring. In these cases, the only information contained in the intersection multiplicity is the length of the subscheme mod 2, and this can be computed with classical methods. We can therefore restrict to intersections involving only even Young diagrams, where we have canonical lifts to $\op{W}(\op{Gr}(k,V))$, cf. Section~\ref{sec:comparisonwitt}.
\end{remark}

\begin{remark}
In the below results, we have restrictions on the twists of the Schubert classes. These arise because we don't have a pushforward from the top Chow--Witt group to the Grothendieck--Witt ring $\op{GW}(F)$ of the base field if the Grassmannian $\op{Gr}(k,n)$ is not oriented, i.e., if $n$ is odd. Of course, one can ask about the corresponding intersection problems in the orientation cover $\widetilde{\op{Gr}}(k,n)$, and basically the same Schubert calculus applies to such computations. 
\end{remark}

The first result is a special case in which we can apply the simpler notion of types of Definition~\ref{def:type}.

\begin{theorem}
\label{thm:types}
Let $F$ be a field of characteristic $\neq 2$. Let $V=F^n$, and fix an orientation on $V=F^n$. Let $[\mathfrak{S}_1],\dots,[\mathfrak{S}_l]$ be oriented Schubert classes in $\op{Gr}(k,V)$ such that $\sum_i\dim\Sigma_i=k(n-k)$. Assume that $F$ is of characteristic $0$, or that the Schubert varieties are smooth. Let $x$ be a point in the intersection of general translates of $\mathfrak{S}_1,\dots,\mathfrak{S}_l$, corresponding to a subspace $W\subset V\otimes_FE$ over a finite extension field $E/F$. Then the multiplicity of $x$ in the intersection product 
\[
[\mathfrak{S}_1]\cdots[\mathfrak{S}_l]\in \widetilde{\op{CH}}^{k(n-k)}(\op{Gr}(k,n))
\]
is given by $\op{type}(W)=\op{tr}_{E/F}\langle\det\xi_W\rangle\in \op{GW}(F)$. 
\end{theorem}

\begin{proof}
If $\op{char} F=0$, we can use Kleiman's transversality, cf. Theorem~\ref{thm:kleiman}, to see that general translates of the Schubert varieties $\mathfrak{S}_1,\dots,\mathfrak{S}_l$ are transverse. If $\op{char} F=p>0$, we can use Sierra's homological transversality, cf. Theorem~\ref{thm:sierra}, to see that general translates of the Schubert varieties are homologically transverse, and by the smoothness assumption, the intersection points in the Schubert varieties are smooth. Summing up, for the arguments below we can assume that the intersection of Schubert varieties is a reduced zero-dimensional subscheme of $\op{Gr}(k,n)$ and the direct sum of the normal spaces of the Schubert varieties is isomorphic to the tangent space.

The relevant classes $[\sigma_j]$ for the intersection problem are given as $(\iota_j)_\ast(1)$ where $\iota_j:\mathfrak{S}_j\to \op{Gr}(k,n)$ is the inclusion of the relevant Schubert variety. Note that the pushforward map $(\iota_j)_\ast$ depends on the choice of orientation of $V$, but we take the canonical relative orientation of the Schubert varieties from Proposition~\ref{prop:orientschubert}. 

The intersection product is now computed by restricting the exterior product  
\[
(\iota_1)_\ast(1)\boxtimes\cdots\boxtimes(\iota_l)_\ast(1)\in \widetilde{\op{CH}}^{k(n-k)}(\op{Gr}(k,n)^{\times l})
\]
along the diagonal inclusion $\Delta:\op{Gr}(k,n)\hookrightarrow\op{Gr}(k,n)^{\times l}$. By \cite{calmes:fasel}, the proper pushforward commutes with the exterior product; in the case at hand there are no signs because the pushforwards start at $\widetilde{\op{CH}}^0$. In particular, we have 
\[
(\iota_1)_\ast(1)\boxtimes\cdots\boxtimes(\iota_l)_\ast(1)=(\iota_1\times\cdots\times\iota_l)_\ast(1). 
\]
Note that the data determining the pushforward $(\iota_1\times\cdots\times\iota_l)_\ast$ is exactly given by the orientations of the exterior product of the normal bundles of the Schubert varieties $\mathfrak{S}_j$. 

Now we consider the following pullback diagram
\[
\xymatrix{
\mathfrak{S}_1\cap\cdots\cap\mathfrak{S}_l\ar[rr]^{\Delta'} \ar[d]_\iota && \mathfrak{S}_1\times\cdots\times\mathfrak{S}_l \ar[d]^{\iota_1\times\cdots\times\iota_l} \\
\op{Gr}(k,n)\ar[rr]_\Delta && \op{Gr}(k,n)^{\times l}
}
\]
In this case, the map $\iota_1\times\cdots\times\iota_l$ is a regular immersion of smooth schemes, and the diagram is transversal in the sense of the base-change result \cite[Theorem 2.4.1]{AsokFaselEuler}. Consequently, we have 
\[
\Delta^\ast(\iota_1\times\cdots\times\iota_l)_\ast(1)=\iota_\ast\Delta'^\ast(1)=\iota_\ast(1). 
\]
In particular, for any point $x\in \mathfrak{S}_1\cap\cdots\cap\mathfrak{S}_l$ the corresponding form is $(\iota_x)_\ast(1)\in\widetilde{\op{CH}}^{k(n-k)}(\op{Gr}(k,n))$. Note that as usual the pushforward $(\iota_x)_\ast$ depends on choices, and in this case the relevant choices are given by the orientations of the normal spaces of the Schubert varieties intersecting at $x$. The relevant multiplicity of the point $x$ in the intersection product depends on the choice of generator of $\widetilde{\op{CH}}^{k(n-k)}(\op{Gr}(k,n))$. The natural choice is given by the canonical orientation of the tangent space of $\op{Gr}(k,n)$, cf. Proposition~\ref{prop:orientflag}; the corresponding generator of the top Chow--Witt group is the pushforward of $1$ along $\iota_x$, with the tangent-space orientation.\footnote{This is exactly the class in the Chow--Witt ring corresponding to the full Young diagram in $k\times(n-k)$-frame.} With this choice of isomorphism $\widetilde{\op{CH}}^{k(n-k)}(\op{Gr}(k,n)) \cong \op{GW}(F)$ the class of $(\iota_x)_\ast(1)$ (here with the orientation coming from the normal spaces of the Schubert varieties) is mapped exactly to  $\op{type}(W)$ as claimed.
\end{proof}

\begin{remark}
This result can also be used to show that the classes $(p_\Lambda)_\ast(1)$ are generators, cf. Remark~\ref{rem:generator}. Take an even Young diagram $\Lambda$ and its complementary Young diagram $\overline{\Lambda}$ (obtained as complement of $\Lambda$ in the $k\times (n-k)$-frame, rotated by 180 degrees). By classical Schubert calculus, the degree of their intersection in the top Chow group is 1, i.e., generic translates of the corresponding Schubert varieties intersect in a single rational point. By the above, the oriented intersection multiplicity of this rational point is the square class of the unit comparing the two natural orientations (from the tangent space of the Grassmannian and the direct sum of the two normal spaces for the Schubert varieties). This is a unit in the Grothendieck--Witt ring and therefore the canonical orientations on the two Schubert varieties for $\Lambda$ and $\overline{\Lambda}$ must be generators of the corresponding $\op{GW}(F)$-summands in the Chow--Witt ring.
\end{remark}

\begin{theorem}
\label{thm:typessing}
Let $F$ be a field of characteristic $\neq 2$, let $n$ be an even number and set $V=F^n$. Consider oriented Schubert classes $\mathfrak{S}_i\in \widetilde{\op{CH}}^\bullet(\op{Gr}(k,V),\mathscr{L}_i)$, $i=1,\dots,l$ such that the underlying Schubert varieties are in general position, $\sum_i \dim \mathfrak{S}_i=k(n-k)$ and  $\bigotimes_i\mathscr{L}_i=0$ in $\op{Pic}(\op{Gr}(k,V))/2$. For a subspace $[W]\in\mathfrak{S}_1\cap\cdots\cap\mathfrak{S}_l$, the oriented multiplicity of $[W]$ in the intersection product $[\mathfrak{S}_1]\cdots[\mathfrak{S}_l]$ is given by the type of $W$ in the sense of Definition~\ref{def:typesing}. 
\end{theorem}

\begin{proof}
The assumptions imply that the underlying Schubert varieties of the cycles $\mathfrak{S}_i$ are homologically transverse. 

Now we want to apply the oriented Tor formula, cf. Section~\ref{sec:torformula} which provides a commutative diagram
\[
\xymatrix{
\op{GW}^{n_1}_{\mathfrak{S}_1}(\op{Gr}(k,V),\mathscr{L}_1)\times\cdots\times \op{GW}^{n_l}_{\mathfrak{S}_l}(\op{Gr}(k,V),\mathscr{L}_l)\ar[r]^>>>>>>\star \ar[d]_{\alpha_{\mathfrak{S}_1}\times\cdots\times\alpha_{\mathfrak{S}_l}} & \op{GW}^{k(n-k)}_{\mathfrak{S}_1\cap\cdots\cap\mathfrak{S}_l}(\op{Gr}(k,V)) \ar[d]^{\alpha_{(\mathfrak{S}_1\cap\cdots\cap \mathfrak{S}_l)}} \\
\widetilde{\op{CH}}^{n_1}(\op{Gr}(k,V),\mathscr{L}_1)\times\cdots\times \widetilde{\op{CH}}^{n_l}(\op{Gr}(k,V),\mathscr{L}_l) \ar[r]_>>>>>>>>\times & \widetilde{\op{CH}}^{k(n-k)}(\op{Gr}(k,V)).
}
\]
Actually, we want to work with the corresponding diagram involving Witt groups (instead of Grothendieck--Witt groups) and $\mathbf{I}^\bullet$-cohomology (instead of Chow--Witt groups). The relevant class in $\op{W}^{n_1}_{\mathfrak{S}_1}(\op{Gr}(k,V),\mathscr{L}_1)\times\cdots\times \op{W}^{n_l}_{\mathfrak{S}_l}(\op{Gr}(k,V),\mathscr{L}_l)$ is the exterior product of the classes $p_\ast(1)\in \op{W}^{n_i}_{\mathfrak{S}_i}(\op{Gr}(k,V),\mathscr{L}_i)$ where $p:\op{Fl}(\Lambda)\to\op{Gr}(k,V)$ is the resolution of singularities of the Schubert variety $\mathfrak{S}_i$ (corresponding to the Young diagram $\Lambda$). 

To determine an explicit representative of the form 
\[
(p_\Lambda)_\ast(1)\in \op{W}^{n_i}_{\mathfrak{S}_i}(\op{Gr}(k,V),\mathscr{L}_i),
\]
we first note that the form $1$ in $\op{W}^0(\op{Fl}(\Lambda),\mathscr{O})$ is given by $\op{id}:\mathscr{O}_{\op{Fl}(\Lambda)}\xrightarrow{\simeq}\mathscr{O}_{\op{Fl}(\Lambda)}$. Next, we use the identification $\op{W}^0(\op{Fl}(\Lambda),\mathscr{O})\cong \op{W}^0(\op{Fl}(\Lambda),\omega_{p_\Lambda}\otimes p_\Lambda^\ast(\mathscr{L}_i))$ which is induced by the explicit choice of (canonical) isomorphism $\mathscr{N}^{\otimes 2}\xrightarrow{\cong} \omega_{p_\Lambda}\otimes p_\Lambda^\ast(\mathscr{L}_i)$: it maps a form $Q\xrightarrow{\simeq}\op{Hom}(Q,\mathscr{O})$ corresponding to $Q\otimes Q\to\mathscr{O}$ to $(Q\otimes\mathscr{N})\otimes(Q\otimes\mathscr{N})\to \omega_{p_\Lambda}\otimes p_\Lambda^\ast\mathscr{L}_i$ by tensoring with the given orientation. In the next step, the relevant push-forward is 
\[
(p_\Lambda)_\ast:\op{W}^0(\op{Fl}(\Lambda),\omega_{p_\Lambda}\otimes p_\Lambda^\ast(\mathscr{L}_i))\to\op{W}^{|\Lambda|}(\op{Gr}(k,V),\mathscr{L}),
\]
and the form we want to push forward is simply the orientation $\mathscr{N}\otimes\mathscr{N}\to \omega_{p_\Lambda}\otimes p_\Lambda^\ast(\mathscr{L}_i)$, cf. Section~\ref{sec:geometricreps}.
From the definitions  in, cf. \cite[Theorem 4.4]{calmes:hornbostel}, the push-forward on Witt groups is induced from the derived pushforward ${\op{R}}p_\ast$ on the level of triangulated categories of perfect complexes and the morphism $\zeta_K:{\op{R}}p_\ast\circ \sharp'_{\omega_{\op{Fl}(\Lambda)}}\to \sharp'_{\omega_{\op{Gr}(k,V)}}\circ{\op{R}}p_\ast$ exchanging push-forward and dualities. In our case, the push-forward is given by 
\begin{eqnarray*}
\op{R}(p_\Lambda)_\ast\mathscr{N}&\to&  \op{R}(p_\Lambda)_\ast\left(\op{Hom}(\mathscr{N},\mathscr{O}_{\op{Fl}(\Lambda)})\otimes\omega_{p_\Lambda}\otimes p_\Lambda^\ast(\mathscr{L}_i)\right)\\&\to&
\op{Hom}\left(\op{R}(p_\Lambda)_\ast\mathscr{N},\op{R}(p_\Lambda)_\ast(\omega_{p_\Lambda}\otimes p_\Lambda^\ast(\mathscr{L}_i))\right)\\&\to&
\op{Hom}\left(\op{R}(p_\Lambda)_\ast\mathscr{N},\mathscr{O}\right)\otimes \mathscr{L}_i,
\end{eqnarray*}
i.e., it is a form on $\op{R}(p_\Lambda)_\ast\mathscr{N}$ with values in the twisting bundle $\mathscr{L}_i$ (which we can assume is either $\mathscr{O}_{\op{Gr}(k,V)}$ or the determinant of the tautological sub- or quotient bundle). 

Now we can apply a base-change argument similar to the proof of Theorem~\ref{thm:types}. We want to apply the base-change formula for Witt groups, cf. \cite[Theorem 5.4]{calmes:hornbostel} and the discussion in Section~\ref{sec:wittbasics}, to the following diagram
\[
\xymatrix{
\prod_{i=1}^l p_i^{-1}(\mathfrak{S}_1\cap\cdots\cap\mathfrak{S}_l) \ar[r] \ar[d]_q & \op{Fl}(\Lambda_1)\times\cdots\times\op{Fl}(\Lambda_l) \ar[d]^{\prod p_i} \\
\op{Gr}(k,V)\ar[r]_\Delta & \op{Gr}(k,V)\times\cdots\times\op{Gr}(k,V)
}
\]
This diagram is cartesian and the upper left entry is the product of all the fibers of the resolutions $p_i:\op{Fl}(\Lambda_i)\to \op{Gr}(k,V)$ of the respective Schubert varieties $\mathfrak{S}_i$ over the intersection subscheme $\mathfrak{S}_1\cap\cdots\cap\mathfrak{S}_l$. Note that by assumption, this diagram is homologically transverse so that the base-change formula for Witt groups can be applied. Consequently, the restriction of the push-forward $\op{R}(p_1)_\ast(1)\boxtimes\cdots\boxtimes \op{R}(p_l)_\ast(1)$ along the diagonal $\Delta:\op{Gr}(k,V)\hookrightarrow\op{Gr}(k,V)^{\times l}$ is isomorphic to the push-forward of 1 along the morphism $q:\prod_{i=1}^l p_i^{-1}(\mathfrak{S}_1\cap\cdots\cap\mathfrak{S}_l)\to \op{Gr}(k,V)$. Note also that $q$ factors as structural morphism of the product of fibers over the intersection points, followed by the closed immersion of the intersection points into $\op{Gr}(k,V)$. By the compatibility of exterior product with push-forwards, the pushforward of 1 along the product morphism can alternatively be computed as $\star$-product of the pushforwards of 1 along the individual structure maps $p_i: p_i^{-1}(\mathfrak{S}_i)\to \op{Gr}(k,V)$. Note again that the pushforward depends on the choice of orientation, and we use the canonical ones from Proposition~\ref{prop:orientschubert}. For a point $[W]\in \mathfrak{S}_1\cap\cdots\cap\mathfrak{S}_l$, the $\star$-product of the pushforwards of 1 along the corresponding structure maps is then, by Definition~\ref{def:typesing}, the type of the corresponding solution subspace. This finishes the argument. 
\end{proof}

\begin{remark}
Note that the transversality result of Sierra, cf. Theorem~\ref{thm:sierra}, together with Proposition~\ref{prop:ratequiv} implies that over infinite fields the situation in Theorem~\ref{thm:typessing} above can always be achieved by taking generic translates.
\end{remark}

\section{Main results and consequences}
\label{sec:main}

Now we finally come to formulate the main results in oriented Schubert calculus. Fix a field $F$ of characteristic $\neq 2$, and let $V$ be an $n$-dimensional $F$-vector space. An \emph{oriented Schubert problem} for the Grassmannian $\op{Gr}(k,V)$ is a collection of oriented Schubert classes $[\mathfrak{S}_i]\in\widetilde{\op{CH}}^\bullet(\op{Gr}(k,n),\mathscr{L}_i)$, $i=1,\dots,l$ with $\sum_i\dim\mathfrak{S}_i=k(n-k)$, i.e., Schubert varieties equipped with the canonical orientations given by Theorem~\ref{thm:lift}. The transversality results imply that, up to translation by $\op{GL}_n$, we can assume that the intersection of the Schubert varieties $\mathfrak{S}_i$ is proper and homologically transversal, in particular, the intersection $\mathfrak{S}_1\cap\cdots\cap\mathfrak{S}_l$ is a zero-dimensional subscheme. A \emph{solution subspace} is a closed point $[W]\in\op{Gr}(k,V)$ in the intersection. The answer to an oriented Schubert problem is a quadratic form in $\op{GW}(F)$ which not only encodes the number of solution subspaces (and the degrees of the corresponding closed points, i.e., the degrees of the field extensions over which the solution is defined) but also additional information on the compatibility of the canonical orientations of the intersecting Schubert varieties, encoded in the types of Definitions~\ref{def:type} and \ref{def:typesing}. 

\subsection{Geometric cycles vs characteristic classes}

 We still need to identify the geometric generators in terms of the characteristic classes used to describe the multiplication in Section~\ref{sec:cohomology}. 

First, we note that there is some stabilization for the Grassmannians: fix a field $F$ and an $F$-vector space $V$ of dimension $n$. Mapping a $k$-dimensional subspace $\iota\colon W\subset V$ to $W\subseteq V\oplus F^m$ via $\iota\oplus 0^m$ induces a morphism $\op{Gr}(k,n)\to \op{Gr}(k,n+m)$. By construction, the tautological subbundle of $\op{Gr}(k,n+m)$ restricts to the tautological subbundle of $\op{Gr}(k,n)$ along this inclusion, and consequently the restriction maps the Pontryagin classes $\op{p}_i$ of $\op{Gr}(k,n+m)$ to the ones for $\op{Gr}(k,n)$. Moreover, the tautological quotient bundle of $\op{Gr}(k,n+m)$ restricts to the Whitney sum of the tautological quotient bundle for $\op{Gr}(k,n)$ and the trivial bundle of rank $m$ and the Pontryagin classes $\op{p}_j^\perp$ restrict accordingly. Similarly, mapping a $k$-dimensional subspace $\iota\colon W\subset V$ to $W\oplus F^m\subseteq V\oplus F^m$ via $\iota\oplus \op{id}^m$ induces a morphism $\op{Gr}(k,n)\to \op{Gr}(k+m,n+m)$. In this case, the tautological quotient bundle of $\op{Gr}(k+m,n+m)$ restricts to the tautological quotient bundle for $\op{Gr}(k,n)$ and the restriction maps the Pontryagin classes $\op{p}_j^\perp$ of $\op{Gr}(k+m,n+m)$ to the ones for $\op{Gr}(k,n)$. Note that smooth Schubert varieties are smaller Grassmannians embedded as above.

The statements concerning Pontryagin classes above imply that the restriction morphism on $\mathbf{W}^\bullet$-cohomology is the canonical quotient projection modulo the  ideal generated by the Pontryagin classes of the bigger Grassmannian which vanish on the smaller Grassmannian. A direct consequence of the above, we have the following: 

\begin{lemma}
\label{lem:stabilization}
Let $F$ be a field of characteristic $\neq 2$ and fix $1\leq k<n$. The Pontryagin class $\op{p}_j$ in $\op{H}^{2j}(\op{Gr}(k,n),\mathbf{W})$ is uniquely determined by its restriction to $\op{Gr}(j,n)$. The Pontryagin class $\op{p}_j^\perp\in \op{H}^{2j}(\op{Gr}(k,n),\mathbf{W})$ is uniquely determined by its restriction to $\op{Gr}(k,j)$.  
\end{lemma}

In particular, to compare Pontryagin classes and the explicit geometric cycles, it suffices to consider the top Pontryagin class $\op{p}_j$ or $\op{p}_j^\perp$ for a small Grassmannian. If we can show that the Pontryagin class $\op{p}_j$ on $\op{Gr}(j,n)$ is given by pushforward of $\langle 1\rangle$ from  the corresponding smooth Schubert variety (for the canonical orientation), then this will be true for $\op{Gr}(k,n)$ for all $k>j$; similarly for $\op{p}_j^\perp$: this follows from the base-change in the cartesian diagram
\[
\xymatrix{
\Sigma_j \ar[r]^= \ar[d] & \Sigma_j \ar[d] \\
\op{Gr}(j,n) \ar[r] & \op{Gr}(k,n).
}
\]

\begin{lemma}
\label{lem:compeuler}
Let $F$ be a field of characteristic $\neq 2$ and fix $1\leq k<n$. 
\begin{enumerate}
\item Assume $k$ is even. Denoting by  $\iota:\Sigma_{1^k} \hookrightarrow \op{Gr}(k,n)$ the closed immersion of the Schubert variety into the Grassmannian, we have $\iota_\ast(1)=\op{e}_k$ in $\op{H}^k(\op{Gr}(k,n),\mathbf{I}^k(\det\mathscr{E}_k^\vee))$. 
\item Assume $n-k$ is even. Denoting by  $\iota:\Sigma_{n-k} \hookrightarrow \op{Gr}(k,n)$ the closed immersion of the Schubert variety into the Grassmannian, we have $\iota_\ast(1)=\op{e}_{n-k}^\perp$ in $\op{H}^{n-k}(\op{Gr}(k,n),\mathbf{I}^{n-k}(\det\mathscr{E}_{n-k}^\vee))$. 
\end{enumerate}
\end{lemma}

\begin{proof}
We prove (1), (2) is proved similarly. First, we note that the vanishing locus of a generic section $s$ of $\mathscr{E}_k$ (or the tautological subspace bundle $\mathscr{S}$) is the Schubert variety $\Sigma_{1^k}$. The Euler class of $\mathscr{S}$ can be concretely represented as Koszul complex of the generic section $s$ of $\mathscr{S}$, by the computations in \cite[Sections 13, 14]{fasel:memoir}. On the other hand, the pushforward of 1 along the inclusion (for the canonical orientation which in this case is given by $\mathscr{O}_{\op{Gr}(k,n)}\cong \det\mathscr{E}_k\otimes \det\mathscr{E}_k^\vee$) is also identified with the Koszul complex for the section, cf. the discussion in \cite[Section 7.2]{calmes:hornbostel}. In particular, both $\iota_\ast(1)$ and $\op{e}_k$ are both represented by the Koszul complex for a generic section of $\mathscr{S}$, which proves the claim.
\end{proof}

\begin{remark}
\label{rem:complement}
We discuss an analogue of the complementary intersections in \cite[Section 4.2.2]{3264}. Let $F$ be a field of characteristic $\neq 2$, and fix $1\leq k<n$. Fix two flags $\mathcal{V}$ and $\mathcal{W}$ in $V$ which are transversal in the sense of \cite[Definition 4.4]{3264}. In particular, we can choose a basis $e_1,\dots,e_n$ of $V$ such that $V_i=\langle e_1,\dots,e_i\rangle$ and $W_j=\langle e_{n+1-j},\dots,e_n\rangle$. For two complementary (even) Young diagrams $\Lambda=(a_1,\dots,a_k)$ and $\Lambda'=(b_1,\dots,b_k)$ with $|\Lambda|+|\Lambda'|=k(n-k)$, the intersection of the corresponding Schubert varieties $\Sigma_\Lambda$ and $\Sigma_{\Lambda'}$ intersect transversely in a single $F$-rational point which is the subspace $U\subset V$ generated by the basis vectors $e_{n-k+i-a_i}$ for $i=1,\dots,k$.

The tangent space of the Grassmannian at the point $[U]$ is canonically identified with $\op{Hom}(U,V/U)$. Taking the basis $f_i:=e_{n-k+i-a_i}$ for $U$ and the $U$-cosets of the other basis vectors $e_j$, denoted as $g_j$, as basis for $V/U$, the corresponding maps $v\mapsto f_i^\vee(v)\cdot g_j$ form a basis for $\op{Hom}(U,V/U)$ which induces the canonical orientation. 

The tangent space of the Schubert variety $\Sigma_\Lambda$ at the point $[U]$ is canonically identified with the subspace of $\op{Hom}(U,V/U)$ which maps $V_{n-k+i-a_i}\cap U\subset U$ into $(V_{n-k+i-a_i}+U)/U\subset V/U$, cf. \cite[Proposition 4.1]{3264}. In particular, a basis of the tangent space is given by the maps $v\mapsto f_i^\vee(v)\cdot g_j$ such that $g_j\in V_{n-k+l-a_l}$ whenever $f_i\in V_{n-k+l-a_l}$. The other maps form a basis for the normal space to $\Sigma_\Lambda$ at $[U]$, and this basis induces the canonical orientation of the normal space. For the Schubert variety $\Sigma_{\Lambda'}$, the situation is exactly reversed, i.e., the basis for the tangent space to $\Sigma_\Lambda$ above is a basis for the normal space to $\Sigma_{\Lambda'}$. Consequently, the decomposition of the tangent space to the Grassmannian at $[U]$ as direct sum of the normal spaces to $\Sigma_\Lambda$ and $\Sigma_{\Lambda'}$ is compatible with the canonical orientations. Reformulated in the $\mathbf{I}^\bullet$-cohomology ring, the intersection product  $[\Sigma_\Lambda]\cdot[\Sigma_{\Lambda'}]$ is equal to $\langle 1\rangle$ times the canonical generator of the top cohomology (which is the class given by pushforward of $\langle 1\rangle$ along the inclusion of an $F$-rational point, with the canonical orientation of the tangent space).
\end{remark}

\begin{proposition}
\label{prop:identification}
Let $F$ be  a field of characteristic $\neq 2$, and fix $1\leq k< n$. Let $\Lambda$ be a rectangular even Young diagram with twist $t$. In the cohomology group $\op{H}^{|\Lambda|}(\op{Gr}(k,n),\mathbf{I}^{|\Lambda|}(\det^t))$, the class associated to $\Lambda$ by Definition~\ref{def:evenmap} agrees with the class $(\iota_\Lambda)_\ast(1)$ of Theorem~\ref{thm:lift}. 
\end{proposition}

\begin{proof}
By Lemma~\ref{lem:stabilization}, it suffices to establish the cases where the longer side of the rectangle $\Lambda$ is $k$ or $n-k$. We deal with the first case, the second is similar. In the first case, the class of Definition~\ref{def:evenmap} is a power of the Euler class $\op{e}_k$. By multiplicativity of the Euler classes, it is the Euler class of the appropriate tensor power of $\mathscr{E}_k$. Then the argument of Lemma~\ref{lem:compeuler} works: the class from Definition~\ref{def:evenmap} is represented by the Koszul complex for a generic section of the appropriate tensor power of $\mathscr{E}_k$. Similarly, the class $(\iota_\Lambda)_\ast(1)$ of Theorem~\ref{thm:lift} is also represented by the Koszul complex for a generic section of the appropriate tensor power of $\mathscr{E}_k$ (whose generic section is a Schubert variety for the Young diagram $\Lambda$). This proves the identification.
\end{proof}

\begin{remark}
In the spirit of Remark~\ref{rem:complement}, it is also possible to extend the Proposition~\ref{prop:identification} to show that the classes of Definition~\ref{def:evenmap} agree with the classes $(p_\Lambda)_\ast(1)$ of Theorem~\ref{thm:lift}. For an even Young diagram $\Lambda$ and a Schubert variety $\Sigma_{2a,2a}$, denote by $\Lambda'$ the Schubert variety for a Young diagram complementary to an even Young diagram appearing in the intersection of $\Sigma_\Lambda$ and $\Sigma_{2a,2a}$. The intersection of $\Sigma_\Lambda$, $\Sigma_{2a,2a}$ and $\Sigma_{\Lambda'}$ (for transversal flags) consists of a single $F$-rational point. Then it can be checked that the decomposition of the tangent space as direct sums of the normal spaces is compatible with the canonical orientations in the definition of the classes $(p_\Lambda)_\ast(1)$. This is similar to the proof of the Pieri formula in \cite[4.2.4]{3264} and is a version of the oriented Pieri formula for the geometric cycles $(p_\Lambda)_\ast(1)$. 
\end{remark}

\begin{theorem}
\label{thm:refcount}
Let $F$ be a field of characteristic $\neq 2$, let $n$ be an even number. Consider smooth oriented Schubert classes $\mathfrak{S}_i\in \widetilde{\op{CH}}^\bullet(\op{Gr}(k,n),\mathscr{L}_i)$, $i=1,\dots,l$ such that the underlying Schubert varieties are in general position, $\sum_i \dim \mathfrak{S}_i=k(n-k)$ and  $\bigotimes_i\mathscr{L}_i=0$ in $\op{Pic}(\op{Gr}(k,n))/2$. Then there exist natural numbers $a,b$ such that we have the following equality in $\op{GW}(F)\cong \widetilde{\op{CH}}^{k(n-k)}(\op{Gr}(k,n))$
\[
\sum_{W\in \mathfrak{S}_1\cap\cdots\cap\mathfrak{S}_l}\op{tr}_{F(W)/F}\langle\det\xi_W\rangle =a\langle 1\rangle+b\langle-1\rangle
\]
The numbers $a$ and $b$ are determined by:
\begin{enumerate} 
\item $a+b$ is the number of solution spaces over an algebraically closed field, equal to the degree of the intersection product $[\mathfrak{S}_1]\cdots[\mathfrak{S}_l]$ in $\op{CH}^{k(n-k)}(\op{Gr}(k,n))$. 
\item $a-b$ is the signed count of solution subspaces over $\mathbb{R}$, equal to the degree of the intersection product $[\mathfrak{S}_1]\cdots[\mathfrak{S}_l]$ in $\op{H}^{k(n-k)}(\op{Gr}_k(\mathbb{R}^n),\mathbb{Z})$. 
\end{enumerate}
\end{theorem}

\begin{proof}
This is a consequence of Proposition~\ref{prop:identification}: the two sides of the equation correspond to two ways of computing the intersection product. The left-hand side is the sum of all the local contribution of the solution subspaces in terms of types as done in Theorem~\ref{thm:types}. Note that this uses the geometric representatives of characteristic classes from Section~\ref{sec:geometricreps}. On the right-hand side, we have the degree of the intersection, and this is computed in terms of the formulas from Section~\ref{sec:cohomology}, which are formulated in terms of the Pontryagin-class presentation of the $\mathbf{W}$-cohomology ring. The fact that the geometric representatives are exactly identified with the images of Schubert classes under the folklore map from Proposition~\ref{prop:folklore} is exactly Proposition~\ref{prop:identification}. Hence the left-hand side is equal to the degree of the intersection of the oriented Schubert classes $[\mathfrak{S}_i]$ (where the generator of the top Chow--Witt group is induced from the canonical orientation on the tangent space of the Grassmannian). 

It remains to identify the degree of the oriented intersection as a quadratic form The class in $\op{GW}(F)$ is determined by its images in $\op{W}(F)$ and the rank. 

First, the rank is the image under the reduction morphism 
\[
\widetilde{\op{CH}}^{k(n-k)}(\op{Gr}(k,n))\to\op{CH}^{k(n-k)}(\op{Gr}(k,n)),
\]
i.e., it is the degree of the intersection in the Chow ring.  Second, the image in $\op{W}(F)$ is the image under the reduction morphism 
\[
\widetilde{\op{CH}}^{k(n-k)}(\op{Gr}(k,n))\to \op{H}^{k(n-k)}(\op{Gr}(k,n),\mathbf{I}^{k(n-k)}).
\]
As a consequence of Theorem~\ref{thm:multw}, the image in $\op{W}(F)$ is of the form $m\langle 1\rangle$. Since the reduction morphism is the quotient modulo $\mathbb{Z}\cdot[\mathbb{H}]$. In particular, the Chow--Witt degree of the intersection is necessarily of the form $a\langle 1\rangle+b\langle-1\rangle$, as claimed. 

To identify $a$ and $b$, note that the reduction morphism to the Chow ring maps both $\langle 1\rangle$ and $\langle-1\rangle$ to $1$, in particular, the image in the Chow ring will be $a+b$. The reduction to $\op{W}(F)$ maps $\langle -1\rangle=-\langle 1\rangle$, in particular the image will be $a-b$. This proves the remaining claims.
\end{proof}

\begin{remark}
In particular, for all oriented Schubert problems, the oriented degree of the intersection in the Chow--Witt ring is a simple quadratic form $a\langle 1\rangle+b\langle-1\rangle$ whose rank and signature can be determined by complex and real realization.
\end{remark} 

\subsection{Refined counts over special fields}
We discuss a couple of consequences of the above result over various fields. Over algebraically closed fields, Theorem~\ref{thm:refcount} simply recovers the classical results from Schubert calculus. For any real embedding $F\hookrightarrow\mathbb{R}$ the above result recovers signed counts. 

In the case where $F=\mathbb{F}_q$ is a finite field of odd characteristic, there are only two possibilities for the class $\det\xi_W\in\mathbb{F}_{q^d}^\times/(\mathbb{F}_{q^d}^\times)^2$ of a solution subspace $W$ defined over $\mathbb{F}_{q^d}$: if the determinant is a square, the splitting map $\xi_W$ is orientation-preserving, otherwise it is orientation-reversing. The following result is an analogue of \cite[Theorem 1]{kass:wickelgren}. 

\begin{proposition}
\label{prop:fq}
Let $F=\mathbb{F}_q$ be a finite field of odd characteristic. In the situation of Theorem~\ref{thm:refcount}, 
we have the following statement for the types of solution subspaces $W$:
\[
\#(W/\mathbb{F}_{q^d}, d\textrm{ even}, \xi_W\textrm{ oriented})+ \#(W/\mathbb{F}_{q^d}, d\textrm{ odd}, \xi_W\textrm{ non-oriented}) \equiv b \bmod 2.
\]
\end{proposition}

\begin{proof}
The discriminant of the right-hand side form in Theorem~\ref{thm:refcount}, viewed as an element of $\mathbb{F}_q^\times/(\mathbb{F}_q^\times)^2$ is $b$. But this is an even number because $\binom{2j}{2i}\equiv\binom{j}{i}\bmod 4$, in particular the discriminant of the right-hand side is $0\bmod 2$. Now we need to understand which subspaces contribute to the discriminant of the left-hand side. By \cite[Lemma 58]{kass:wickelgren},
the discriminant of $\op{tr}_{\mathbb{F}_{q^d}/\mathbb{F}_q}\langle \det\xi_W\rangle$ is nonzero if $d$ is even and $\det\xi_W$ is a square or if $d$ is odd and $\det\xi_W$ is a nonsquare. This proves the claim.
\end{proof}

\begin{proposition}
Let $F$ be a local field with odd residue characteristic. In the situation of Theorem~\ref{thm:refcount}, we have 
\[
\sum_{W\in\mathfrak{S}_1\cap\cdots\cap\mathfrak{S}_l}\op{res}\left(\op{tr}_{F(W)/F}\langle\det\xi_W\rangle\right)=0,
\]
i.e., the sum of the residues of the type-forms for solution subspaces is $0$. 
\end{proposition}

\begin{proof}
The right-hand side of the formula in Theorem~\ref{thm:refcount} is an unramified form, i.e., has trivial residue. 
\end{proof}

\begin{corollary}
Let $F$ be a number field. Then for every finite place $v$, the sum of the Hasse invariants at $v$ of the types for the solution subspaces must be $0$. 
\end{corollary}

\begin{remark}
If $a\langle 1\rangle+b\langle -1\rangle$ is the degree of the intersection of an oriented Schubert problem over the field $F$, then for every embedding of $F$ into a real closed field, the signed count of solution subspaces is $a-b$. 

If $a\langle 1\rangle+b\langle-1\rangle$ is the degree of the intersection of an oriented Schubert problem over $\mathbb{Q}$, and there is a single solution subspace defined over a number field of degree $[F:\mathbb{Q}]=a+b$, then the number of real embeddings of $F$ must be equal to one of the possible numbers of real solutions to the Schubert problem over $\mathbb{R}$. In particular, $a-b$ is a lower bound on the number of real embeddings of $F$.

Note that even in the above case of a single solution over a number field, the refined count doesn't restrict the possible field extension. Of course, if the comparison between the orientations for the solution subspace is $1$ (i.e., the splitting isomorphism between the tangent space to $\op{Gr}(k,V)$ and the direct sum of normal spaces), then the refined count would be the trace form of the number field. But this is no restriction: if the trace form of the number field is complicated (and has a lot of non-trivial local Hasse invariants), then this only means that the determinant of the splitting map must be similarly complicated. 
\end{remark}

\subsection{Comparison with twisted Witt groups}
\label{sec:comparisonwitt}

Now we are ready to compare the previous computations of $\mathbf{W}$-cohomology in Section~\ref{sec:cohomology} with the computations of total Witt groups in \cite{balmer:calmes}. 

\begin{theorem}
Let $F$ be a perfect field of characteristic $\neq 2$, and let $1\leq k<n$. The natural morphisms $\alpha:\op{W}^i(\op{Gr}(k,n),\mathscr{L})\to \op{H}^i(\op{Gr}(k,n),\mathbf{W}(\mathscr{L}))$ induce an isomorphism of commutative $\op{W}(F)$-algebras:
\[
\op{W}^i(\op{Gr}(k,n),\mathscr{L})\xrightarrow{\cong} \bigoplus_{j\equiv i\bmod 4} \op{H}^j(\op{Gr}(k,n),\mathbf{W}(\mathscr{L})).
\]
\end{theorem}

\begin{proof}
By Proposition~\ref{prop:compatpush}, we know that the morphisms $\alpha$ are compatible with pushforwards. In particular, the generators of $\op{W}^i(\op{Gr}(k,n),\mathscr{L})$ discussed in \cite[Theorem 6.1]{balmer:calmes} are mapped exactly to the generators of $\op{H}^i(\op{Gr}(k,n),\mathbf{W}(\mathscr{L}))$, cf. Propositions~\ref{prop:addbasis1} and \ref{prop:identification}. Consequently, we have a $\op{W}(F)$-linear homomorphism between finitely generated free $\op{W}(F)$-modules which is induced by a bijection of total bases, hence we get an isomorphism of $\op{W}(F)$-modules. Compatibility with the intersection products follows from the oriented Tor formula discussed in Section~\ref{sec:wittbasics}. 
\end{proof}

\begin{remark}
The above $\op{W}(F)$-algebra homomorphism provides a complete description of the multiplication in the total Witt groups $\bigoplus_{i,\mathscr{L}}\op{W}^i(\op{Gr}(k,n),\mathscr{L})$. Products of classes corresponding to even Young diagrams can be evaluated using the results in Section~\ref{sec:cohomology}. This resolves the open question of determining Littlewood--Richardson coefficients for the multiplication in Witt groups, cf. \cite[Remark 7.8]{balmer:calmes}. 
\end{remark}

\section{Sample enumerative applications}
\label{sec:results}

In this section, we discuss a couple of applications to refined enumerative geometry which demonstrate the type of results that can be obtained using  oriented Schubert calculus. What we provide are classical problems, either refined by lifting to the Chow--Witt (resp. $\mathbf{I}^\bullet$-cohomology) ring or modified by replacing Chern classes by Pontryagin classes. The classical problems (like lines on hypersurfaces, balanced subspaces etc.) can all be found in the book \cite{3264}, and some of the refinements have appeared in the literature on real Schubert calculus before. 

\subsection{Intersections with torsion classes}

We first note a simple consequence for intersection products involving torsion classes. In the orientable cases, the refined degree of the intersection is always a multiple of the hyperbolic plane. More precisely, we have:

\begin{theorem}
\label{thm:torsionhyp}
Let $F$ be a field of characteristic $\neq 2$, and consider oriented Schubert classes $[\mathfrak{S}_i]\in \widetilde{\op{CH}}^\bullet(\op{Gr}(k,n),\mathscr{L}_i)$, $i=1,\dots,l$  such that $\sum_i\dim\mathfrak{S}_i=k(n-k)$. 
\begin{enumerate}
\item Assume that $n$ is even and the class of $\bigotimes_i\mathscr{L}_i$ in $\op{Pic}(\op{Gr}(k,n))/2$ is trivial. If the corresponding Young diagram for one of the oriented Schubert classes fails to be even, then  in the notation of Theorem~\ref{thm:refcount}, we have
\[
\sum_{W\in\mathfrak{S}_1\cap\cdots\cap\mathfrak{S}_l}\op{tr}_{F(W)/F}\langle\det\xi_W\rangle =a\mathbb{H}
\]
where $a$ is half the degree of the intersection product $[\mathfrak{S}_1]\cdots[\mathfrak{S}_l]$ in the Chow ring $\op{CH}^{k(n-k)}(\op{Gr}(k,n))$. 
\item Assume that $n\bmod 2$ is different from the class of $\bigotimes_i\mathscr{L}_i$ in $\mathbb{Z}/2\mathbb{Z}\cong \op{Pic}(\op{Gr}(k,n))/2$. Then we have 
\[
\widetilde{\op{CH}}^{k(n-k)}(\op{Gr}(k,n),\otimes_i\mathscr{L}_i)\cong \op{CH}^{k(n-k)}(\op{Gr}(k,n)) 
\]
and the degree of the intersection $[\mathfrak{S}_1]\cdots[\mathfrak{S}_l]$ is the one in the Chow group. 
\end{enumerate}
\end{theorem}

\begin{proof}
By the computations of Chow--Witt rings of Grassmannians, cf. Section~\ref{sec:recall1}, we have two cases for the top Chow--Witt group: if $n\bmod 2$ equals the class of $\bigotimes_i\mathscr{L}_i$ in $\mathbb{Z}/2\mathbb{Z}\cong \op{Pic}(\op{Gr}(k,n))/2$, then $\widetilde{\op{CH}}^{k(n-k)}(\op{Gr}(k,n),\bigotimes_i\mathscr{L}_i)\cong \op{GW}(F)$. Otherwise, the natural projection $\widetilde{\op{CH}}^{k(n-k)}(\op{Gr}(k,n),\otimes_i\mathscr{L}_i)\to \op{CH}^{k(n-k)}(\op{Gr}(k,n))$ is an isomorphism, since the top $\mathbf{I}^\bullet$-cohomology group maps isomorphically to the top mod 2 Chow group. This already establishes (2). 

To prove (1), note that the assumption that one Young diagram is non-even implies that the corresponding class in $\mathbf{I}^\bullet$-cohomology is $\op{I}(F)$-torsion. Therefore, the whole oriented intersection product is $\op{I}(F)$-torsion and thus trivial in the top $\mathbf{I}^\bullet$-cohomology group (since that is $\op{I}(F)$-torsion-free). In particular, the intersection product in $\widetilde{\op{CH}}^{k(n-k)}(\op{Gr}(k,n),\bigotimes_i\mathscr{L}_i)\cong\op{GW}(F)$ is a class whose image in $\op{W}(F)$ is trivial, hence it must be a multiple of the hyperbolic plane. The appropriate multiple is determined by the requirement that the rank equals the degree of the intersection product in the top Chow group. This proves the claim.
\end{proof}

\begin{remark}
In particular, we can only expect quadratic forms different from hyperbolic planes (and thus non-trivial signed counts over the real numbers) if the Young diagrams in the Schubert problem are all even. In particular, the corresponding Schubert varieties will either have dimensions divisible by $4$ (i.e., completely even) or are products of such with Euler classes or the hook class $\op{R}$. 
\end{remark}

\begin{example}
We can determine the refined degree of the Grassmannian in the Pl\"ucker embedding $\op{Gr}(2,n+1)\hookrightarrow \mathbb{P}(\bigwedge^2 F^{n+1})$. Classically, this degree can be computed as $\deg (\op{c}_1^{2n-2})$. For the class $\overline{\op{c}}_1$, we have $\op{Sq}^2_{\mathscr{O}}(\overline{\op{c}}_1)=\overline{\op{c}}_1^2=\overline{\op{c}}_2+\sigma_{1,1}$ and $\op{Sq}^2_{\det}(\overline{\op{c}}_1)=0$. In particular, for $k(n-k)\geq 2$, the class $\overline{\op{c}}_1$ is not in the image of $\op{H}^1(\op{Gr}(2,n+1),\mathbf{I}^1)$, but it is in the image of $\op{H}^1(\op{Gr}(2,n+1),\mathbf{I}^1(\det))$. The unique lift of $\overline{\op{c}}_1$ to $\mathbf{I}^1$-cohomology is given by $\beta_{\det}(1)$. 

We want to compute $\beta_{\det}(1)^{2n-2}$. First note that the class $\beta_{\det}(1)^{2n-2}$ is in the image of $\beta$. In the case where $n$ is odd, the Grassmannian $\op{Gr}(2,n+1)$ is orientable, and the image of in $\op{H}^{2n-2}(\op{Gr}(2,n+1),\mathbf{I}^{2n-2})\cong\op{W}(F)$ is trivial. Consequently, $\beta_{\det}(1)^{2n-2}=0$. In the case where $n$ is even, we have $\op{H}^{2n-2}(\op{Gr}(2,n+1),\mathbf{I}^{2n-2})\cong\mathbb{Z}/2\mathbb{Z}$ and the image of $\beta$ is detected on $\op{Ch}^{2n-2}(\op{Gr}(2,n+1))$. Then  $\beta_{\det}(1)^{2n-2}$ is trivial if and only if $\overline{\op{c}}_1^{2n-2}$ is. Now
\[
\deg\op{c}_1^{2n-2}=\frac{(2n-2)!}{n!(n-1)!}
\]
is the $(n-1)$-th Catalan number which is odd if and only if $n=2^k$. Consequently, $\beta_{\det}(1)^{2n-2}$ is nontrivial if and only if $n\neq 2^k$. 

For the Chow--Witt ring, we have the following consequences, obtained by combining this computation of the $\mathbf{I}^\bullet$-cohomological degree with the classical computation, cf. \cite[Section 4.3]{3264}. If $n$ is odd, then $\widetilde{\op{CH}}^{2n-2}(\op{Gr}(2,n+1))\cong\op{GW}(F)$, and denoting the unique lift of $\op{c}_1$ by $\widetilde{\op{c}}_1$ we have $\widetilde{\op{c}}_1^{2n-2}=\frac{(2n-2)!}{2n!(n-1)!}\cdot\mathbb{H}$; the refined degree is half the $(n-1)$-th Catalan number times the hyperbolic plane. If $n$ is even, then $\widetilde{\op{CH}}^{2n-2}(\op{Gr}(2,n+1))\cong \op{CH}^{2n-2}(\op{Gr}(2,n+1))\cong\mathbb{Z}$; in this case, there is no refinement and the degree simply agrees with $\deg(\op{c}_1^{2n-2})$ in the Chow ring.

In the case where $n$ is odd, this provides a refinement of the enumerative problem of how many lines in $\mathbb{P}^n$ meet each of $2n-2$ general $(n-2)$-planes. Even though the quadratic form providing the refined count is only a copy of hyperbolic planes whose rank is the classical degree, this provides more information than the simple number. It is possible to associate to each of the lines meeting the $(n-2)$-planes a quadratic form (or better a class in $\op{GW}(F)$). The refined enumerative information is that the sum of classes in $\op{GW}(F)$ must be $\op{C}_{n-1}/2$ copies of the hyperbolic plane $\mathbb{H}$. For more information, cf. \cite{srinivasan:wickelgren}. 
\end{example}

\begin{example}
As another example, we can ask how many $(n-1)$-planes are contained in the intersection of two general quadrics in $\mathbb{P}^{2n}$. By \cite[Proposition 4.15]{3264}, the corresponding cycle is trivial in $\op{Ch}^\bullet(\op{Gr}(n-1,2n+1))$. Consequently, the cycle $\Phi$ from loc.cit. lifts to the Chow--Witt ring. Its self-intersection in $\mathbf{I}^\bullet$-cohomology will be zero. The self-intersection of the class in the Chow--Witt ring will be $4^n/2\cdot\mathbb{H}$, a sum of copies of the hyperbolic plane whose rank equals the classical count $4^n$. 
\end{example}

\subsection{Counting balanced subspaces}

Next we discuss an enumerative problem which can be refined to a problem involving non-torsion classes, the arithmetic count of balanced subspaces. This generalizes the signed count of balanced subspaces over $\mathbb{R}$ which was established by Feh{\'e}r and Matszangosz in \cite{feher:matszangosz}. 

Let $F$ be a field of characteristic $\neq 2$. We consider the following enumeration problem which is a higher-dimensional analogue of the classical problem of finding the number of lines in $\mathbb{P}^3$ intersecting 4 given lines in general position.

\begin{quote}
Given four $b$-dimensional subspaces in general position in a $2b$-dimensional space, what is the number of $2a$-dimensional subspaces having $a$-dimensional intersection with all the given subspaces? 
\end{quote}

\begin{definition}
\label{def:balanced}
Let $F$ be a field of characteristic $\neq 2$, and let $a<b$ be two natural numbers. Let $U$ be an $F$-vector space of dimension $2b$, and let $V_1,\dots,V_4\subset U$ be four $F$-subspaces of dimension $b$. A $2a$-dimensional subspace $W\subset U$ is \emph{balanced} (relative to the subspaces $V_i$) if the intersection $W\cap V_i$ has dimension $a$ for all $i$. 
\end{definition}

In the case where $F$ is algebraically closed, the number of balanced subspaces is determined by a theorem of Vakil \cite{vakil}:

\begin{proposition}
\label{prop:vakil}
Let $F$ be an algebraically closed field, and let $a<b$ be two natural numbers. The number of $2a$-dimensional balanced subspaces of a $2b$-dimensional space is $\binom{b}{a}$. 
\end{proposition}

The classical case is $a=1$, $b=2$: there are 2 lines meeting four lines in $\mathbb{P}^3$ in general position. The following result is now a generalization of the signed count of balanced subspaces in \cite{feher:matszangosz}. 

\begin{theorem}
\label{thm:fmgw}
Let $F$ be a field of characteristic $\neq 2$. Let $a=2i$ and $b=2j$ be even numbers with $a<b$. Then the following equality holds in $\op{GW}(F)\cong\widetilde{\op{CH}}^{4a(b-a)}(\op{Gr}(2a,2b))$:
\[
\sum_{\op{balanced} W}\op{tr}_{F(W)/F}\langle\det\xi_W\rangle=\frac{\binom{2j}{2i}+\binom{j}{i}}{2}\langle 1\rangle + \frac{\binom{2j}{2i}-\binom{j}{i}}{2}\langle-1\rangle.
\]
\end{theorem}

\begin{proof}
For a fixed $b$-dimensional subspace $V\subset U$, the set of $2a$-dimensional subspaces $W\subset U$ with $\dim W\cap V\geq a$ is the Schubert variety $\Sigma_\Lambda$ whose corresponding Young diagram is a rectangle with side lengths $a$ and $b-a$. Note that this is a completely even Young diagram by the assumptions on $a$ and $b$. In particular, the relevant oriented Schubert problem is given by the 4-fold self-intersection of the lifts of the Schubert class $\sigma_\Lambda$ given by the canonical orientation of the normal bundle of $\Sigma_\Lambda$, cf. Proposition~\ref{prop:orientschubert}. Then we can apply Theorem~\ref{thm:refcount}. 

The multiplicity can be computed separately: the dimension of the class in $\op{GW}(F)$ is the degree of the 4-fold self-intersection of $[Z]$ in the Chow ring, which by Proposition~\ref{prop:vakil} equals $\binom{b}{a}$. The image of the class in $\op{W}(F)$ is the degree of the 4-fold self-intersection in $\op{H}^n({\op{Gr}}(2a,2b),\mathbf{I}^n)$. By Proposition~\ref{prop:folklore}, this multiplicity is the same as the degree of the 4-fold self-intersection of the rectangular Young diagram with side lengths $i$ and $j$ in $\op{CH}^\bullet(\op{Gr}(a,b))$. Again, by Proposition~\ref{prop:vakil}, this degree is $\binom{j}{i}$. The degree in the Chow--Witt ring is then the class of the quadratic form in $\op{GW}(F)$ of dimension $\binom{2j}{2i}$ whose image in $\op{W}(F)$ is $\binom{j}{i}$. This class must necessarily be of the form $m\langle 1\rangle+n\langle -1\rangle$, and satisfy
\[
m+n=\binom{2j}{2i}, \qquad m-n=\binom{j}{i}.
\]
This provides the degree on the right-hand side of the equation.
\end{proof}

\begin{example}
The simplest case is $a=2$, $b=4$, in which the degree of the intersection problem in $\op{Gr}(4,8)$ is $4\langle 1\rangle+2\langle-1\rangle$. For $a=4$, $b=8$, the degree of the intersection problem in $\op{Gr}(8,16)$ is $38\langle 1\rangle+32\langle-1\rangle$, recovering the 70 solution subspaces from classical Schubert calculus and the signed count of $6$ over the reals.
\end{example}

\subsection{3-planes and highest power of $\op{p}_1$}

We discuss another enumerative problem involving non-torsion classes, the arithmetic count of 3-planes whose intersection with  $2n$ subspaces in $\mathbb{P}^{2n+3}$ of dimension $(2n-1)$ in general position is a line. This generalizes the results of Section 5.2 of \cite{finashin:kharlamov:3planes}.

For $\mathbb{P}^{2n-1}\subseteq\mathbb{P}^{2n+3}$, the set of 3-planes intersecting $\mathbb{P}^{2n-1}$ in a line is the Schubert variety $\Sigma_{2,2}$ in $\op{Gr}(4,2n+4)$. By Proposition~\ref{prop:orientschubert}, there is a unique lift of the class $\sigma_{2,2}$ to the Chow--Witt ring, given by the canonical orientation of the normal bundle for the Schubert variety $\Sigma_{2,2}$. Note also that, as discussed in Proposition~\ref{prop:pontryaginlift}, the lift of the Schubert class $\sigma_{2,2}$ to the Chow--Witt ring is -- up to the image of the Bockstein -- the first Pontryagin class $\op{p}_1$. 

For every 3-plane $W$ in the $2n$-fold self-intersection of $\sigma_{2,2}$, we have a quadratic form given by the type of the corresponding subspace, cf. Definition~\ref{def:type}, which compares the orientations of the Grassmannian with the canonical orientations of the normal spaces to the translates of $\Sigma_{2,2}$.  The following consequence of Theorem~\ref{thm:refcount} describes the refined count of 3-planes in the intersection: 

\begin{theorem}
\label{thm:fk}
Let $F$ be a field of characteristic $\neq 2$. Then the following equality holds in $\op{GW}(F)\cong \widetilde{\op{CH}}^{8n}(\op{Gr}(4,2n+4))$: 
\[
\sum_{3\op{-planes } [W]\in \sigma_{2,2}^{2n}}\op{tr}_{F(W)/F}\langle\det\xi_W\rangle=\frac{\op{D}_n+\op{C}_n}{2}\langle 1\rangle+ \frac{\op{D}_n-\op{C}_n}{2}\langle-1\rangle,
\]
where $\op{D}_n$ denotes the degree of $\sigma_{2,2}^{2n}$ in $\op{CH}^{8n}(\op{Gr}(4,2n+4))$ and $\op{C}_n=\frac{(2n)!}{n!(n+1)!}$ is the $n$-th Catalan number.
\end{theorem}

\begin{proof}
The dimension of $\op{Gr}(4,2n+4)$ is $8n$, and we want to compute the degree of $\op{p}_1^{2n}\in \widetilde{\op{CH}}^{8n}(\op{Gr}(4,2n+4))\cong\op{GW}(F)$. We apply Theorem~\ref{thm:refcount}, then the claim follows if $\op{D}_n$ is the degree of the intersection in the Chow ring and $\op{C}_n$ is the degree of the intersection in the $\mathbf{I}^\bullet$-cohomology ring. So it suffices to show that the degree of the intersection in $\mathbf{I}^\bullet$-cohomology equals the $n$-th Catalan number. The easiest way to go is to use Proposition~\ref{prop:folklore}. By \cite[Proposition 4.12]{3264},  $\deg\op{c}_1^{2n}\in\op{CH}^{2n}(\op{Gr}(2,n+2))$ is the $n$-th Catalan number $\op{C}_n$. Consequently, $\op{p}_1^{2n}$ is the quadratic form $\op{C}_n\langle 1\rangle$. 
\end{proof}

\begin{remark}
It seems that no closed-form expression for the numbers $\op{D}_n$ is known, cf. MathOverflow-question 302535 ``Closed form for integer series from enumerative problem?'', and the series also doesn't appear in the OEIS database. Using the \verb!SchurRings! package of \verb!Macaulay2! provides the following first terms of the series $\op{D}_n$, starting from $n=2$: 
\[
6, 145, 8806, 830622, 100317140, 14342519633, 2325250316950, 415755865974454, \dots
\]
Consequently, the first instances of the quadratic form in Theorem~\ref{thm:fk} are
\begin{center}
\begin{tabular}{|c||c|}
\hline
$n$ & $\frac{\op{D}_n+\op{C}_n}{2}\langle 1\rangle+\frac{\op{D}_n-\op{C}_n}{2}\langle-1\rangle$ \\
\hline\hline
2 & $4\langle 1\rangle+2\langle -1\rangle$\\\hline
3 & $75\langle 1\rangle+70\langle -1\rangle$\\\hline
4 & $4410\langle 1\rangle+4396\langle -1\rangle$\\\hline
5 & $415332\langle 1\rangle+415290\langle -1\rangle$\\\hline 
6 & $50158636\langle 1\rangle+50158504\langle -1\rangle$\\\hline 
7 & $7171260031\langle 1\rangle+7171259602\langle -1\rangle$\\\hline 
8 & $1162625159190\langle 1\rangle+1162625157760\langle -1\rangle$\\\hline 
\end{tabular}
\end{center}
\end{remark}

\end{document}